\documentclass[12pt,letterpaper]{article}
\usepackage[utf8]{inputenc}

\usepackage{geometry} %% Automatically sets paper size in PDF
\usepackage{mathtools}
\usepackage{color,latexsym,amssymb,amsthm,graphicx,subfigure}
\usepackage{cite}
\usepackage[percent]{overpic}
\usepackage{tikz-cd}
\usepackage{enumerate}
\usepackage{float}

\let\eps\varepsilon

\newcommand{\CC}{\mathbb{C}}
\newcommand{\RR}{\mathbb{R}}

\newcommand{\tubes}{\mathbb{T}}
\newcommand{\tube}{T}

\newcommand{\lines}{\mathcal{L}}
\newcommand{\II}{\operatorname{II}}
\newcommand{\dist}{\operatorname{dist}}

\newcommand{\itemizeEqnVSpacing}{\rule{0pt}{1pt}\vspace*{-12pt}}

\newtheorem{thm}{Theorem}[section]
\newtheorem{lem}{Lemma}[section]
\newtheorem{cor}{Corollary}[section]

\newtheorem{conj}{Conjecture}[section]
\newtheorem{prop}{Proposition}[section]

\newtheorem*{techDiscretizedSeveriThm}{Theorem \ref{discretizedSeveri}, technical version}

\theoremstyle{remark}
\newtheorem{rem}{Remark}[section]
\newtheorem{defn}{Definition}[section]

\begin{document}

\title{A discretized Severi-type theorem with applications to harmonic analysis}
\author{Joshua Zahl\thanks{University of British Columbia, Vancouver BC. Supported by a NSERC Discovery Grant.}}
\maketitle
\begin{abstract}
In 1901, Severi proved that if $Z$ is an irreducible hypersurface in $\mathbb{P}^4(\mathbb{C})$ that contains a three dimensional set of lines, then $Z$ is either a quadratic hypersurface or a scroll of planes. We prove a discretized version of this result for hypersurfaces in $\mathbb{R}^4$. As an application, we prove that at most $\delta^{-2-\varepsilon}$ direction-separated $\delta$-tubes can be contained in the $\delta$-neighborhood of a low-degree hypersurface in $\mathbb{R}^4$. 

This result leads to improved bounds on the restriction and Kakeya problems in $\mathbb{R}^4$. Combined with previous work of Guth and the author, this result implies a Kakeya maximal function estimate at dimension $3+1/28$, which is an improvement over the previous bound of $3$ due to Wolff. As a consequence, we prove that every Besicovitch set in $\mathbb{R}^4$ must have Hausdorff dimension at least $3+1/28$. Recently, Demeter showed that any improvement over Wolff's bound for the Kakeya maximal function yields new bounds on the restriction problem for the paraboloid in $\mathbb{R}^4$. 
\end{abstract}
\section{Introduction}\label{introSection}
In \cite{S}, Severi classified projective hypersurfaces in $\mathbb{P}^4(\CC)$ that contain many lines.
\begin{thm}[Severi]\label{SeveriThm}
Let $Z\subset \mathbb{P}^4(\CC)$ be an irreducible hypersurface. Let $\Sigma$ be the set of lines contained in $Z$. Then
\begin{itemize}
\item If $\dim(\Sigma)=4$, then $Z$ is a hyperplane.
\item If $\dim(\Sigma)=3$, then $Z$ is either a quadratic hypersurface or a scroll of planes. 
\end{itemize}
\end{thm}
This theorem was later generalized to higher dimensions by Segre \cite{Se}. See \cite{R} for further discussion and \cite[Appendix A]{SS} for a modern (and English) proof. 

Severi's theorem allows us to control the set of directions of lines inside a hypersurface.
\begin{cor}\label{SeveriCor}
Let $Z\subset \mathbb{P}^4(\CC)$ be an irreducible hypersurface. Then the lines contained in $Z$ point in at most a two-dimensional set of directions.
\end{cor}

Algebraic varieties containing many lines have recently become a topic of interest when studying the Kakeya and restriction problems. In \cite{KLT}, Katz, \L{}aba, and Tao observed that the Heisenbeg group 
$$
\mathbb{H}=\{(z_1,z_2,z_3)\in\CC^3\colon \operatorname{Im}(z_3) = \operatorname{Im}(z_1\bar z_2)\}
$$
is an ``almost counter-example'' to the Kakeya conjecture---it is a subset of $\CC^3$ that contains many complex lines, few of which lie in a common plane. In four dimensions, Guth and the author showed in \cite{GZ} that low-degree hypersurfaces containing many lines are the only possible obstruction to obtaining improved Kakeya estimates in $\RR^4$. Similar statements are implicit in the works of Guth \cite{G} and Demeter \cite{D}.

In this paper, we will prove a discretized version of Theorem \ref{SeveriThm}. Our theorem will classify the algebraic hypersurfaces in $\RR^4$ whose $\delta$-neighborhood, restricted to the unit ball, contains many unit line segments. In contrast to the classical situation studied by Severi, it is not true that if $Z(P)$ is an irreducible hypersurface in $\RR^4$ whose $\delta$-neighborhood contains many unit line segments, then $Z(P)$ must be a hyperplane, a quadric hypersurface, or a scroll of planes. For example, let $P(x) = x_1x_2 + \delta^{100}$. It is easy to verify that $P$ is irreducible, and the $\delta$-neighborhood of $Z(P)$ contains many unit line segments. Geometrically, $Z(P)$ is a small perturbation of the variety $\{x_1 = 0\}\cup \{x_2=0\}$, which is a union of two hyperplanes. In particular, large regions of the $\delta$-neighborhoods of $Z(P) \cap B(0,1)$ and $\{x_1 = 0\}\cup \{x_2=0\}\cap B(0,1)$ are comparable. In the example above, ``half'' of the unit line segments lying near $N_{\delta}(Z)$ are contained in the $\delta$--neighborhood of the hyperplane $\{x_1 = 0\}$, and ``half'' are contained in the $\delta$--neighborhood of the hyperplane $\{x_2=0\}$.

As a more extreme example, $Z(P)$ might be a small perturbation of the variety $Z_1\cup Z_2\cup Z_3\cup Z_4$, where $Z_1$ is an arbitrary low-degree hypersurface containing few lines; $Z_2$ is a scroll of planes; $Z_3$ is a hyperplane; and $Z_4$ is a quadratic hypersurface. Our discretized version of Severi's theorem follows this idea. Informally it says that if $Z$ is a hypersurface, then the unit line segments contained in $N_{\delta}(Z)$  can be partitioned into four classes in the spirit of the above example.

As an application of our techniques, we prove a variant of Corollary \ref{SeveriCor}, which says that the set of unit line segments in the unit ball lying near $N_{\delta}(Z)$ can point in at most $\delta^{-2-\eps}$ different $\delta$-separated directions. As discussed in Section \ref{KakeyaSection}, this result yields an improved bound for the Kakeya maximal function in $\RR^4$. In \cite{De}, Demeter proved that such an improvement for the Kakeya maximal function would yield new bounds on the restriction problem for the paraboloid in $\RR^4$. This will be discussed further in Section \ref{progRestrictionSec}. 

Before stating the main result of this paper, we will introduce some notation.

\begin{defn}\label{directionOfLinesDef}
Let $\lines$ be the set of lines in $\RR^4$. For each $\ell\in\lines$, define $\operatorname{Dir}(\ell)$ to be a unit vector in $\RR^4$ pointing in the same direction as $\ell$. By convention, we will choose the unit vector $v=(v_1,v_2,v_3,v_4)$ with $v_1\geq 0;\ v_2\geq 0$ if $v_1=0$; $v_3\geq 0$ if $v_1=v_2=0$; and $v_4 = 1$ if $v_1=v_2=v_3=0$. 

If $E$ is a set of unit vectors in $\RR^4$ and $\delta>0$, we will write $\mathcal{E}_{\delta}(E)$ to denote the $\delta$ covering number of $E$. More generally if $(X,d)$ is a metric space and $E\subset X$, $\mathcal{E}_{\delta}(E)$ will denote the $\delta$ covering number of $E$.
\end{defn}

Our proof will refer to ``rectangular prisms,'' which are the discretized analogues of lines, planes, and hyperplanes. These prisms will have ``long directions,'' which have length two, and ``short directions,'' which have length much smaller than two. Informally, we say a rectangular prism is $k$ dimensional if it has $k$ long directions (all such prisms will be contained in $\RR^4$). We say that a line is covered by a rectangular prism if the intersection has length at least two.

The following is a discretized version of Theorem \ref{SeveriThm}.
\begin{thm}\label{discretizedSeveri}
Let $P\in\RR[x_1,x_2,x_3,x_4]$ be a polynomial of degree $D$, and let $Z = Z(P)\cap B(0,1)$. Let $\delta,\kappa,u,s\in (0,1)$ be numbers satisfying $0<\delta<  u < s<1$ and $\delta<\kappa<1$ (if these conditions are not satisfied the theorem is still true, but it has no content). 

Define
$$
\Sigma = \{\ell\in\lines\colon |\ell\cap N_{\delta}(Z)|\geq 1 \}.
$$
Then we can write $\Sigma = \Sigma_1\cup\Sigma_2\cup\Sigma_3\cup\Sigma_4$, where

\begin{itemize}
\item There is a collection of $O_D\big(|\log\delta|^{O(1)}s^{-2}\big)$ rectangular prisms of dimensions $2\times s\times s \times s$ so that every line from $\Sigma_1$ is covered by one of these prisms.
\item There is a collection of $O_D\big((|\log\delta|/s)^{O(1)}u^{-1}\big)$ rectangular prisms of dimensions $2\times 2\times u \times u$ so that every line from $\Sigma_2$ is covered by one of these prisms.
\item There is a collection of $O_D\big(|\log\delta|^{O(1)}\big)$ rectangular prisms of dimensions $2 \times 2 \times 2 \times \kappa$ so that every line in $\Sigma_3$ is covered by one of these prisms.
\item There is a set $\Sigma_4^\prime\subset\Sigma_4$ with 
$$
\mathcal{E}_\delta(\Sigma_4^\prime)\gtrsim_{D}\big(us\kappa/|\log\delta|\big)^{O(1)}\mathcal{E}_\delta(\Sigma_4)
$$
and a quadratic hypersurface $Q$ so that for every line $\ell^\prime\in \Sigma_4^\prime$, there is a line $\ell$ contained in $Z(Q)$ with 
$$
\operatorname{dist}(\ell,\ell^\prime)\lesssim_D \big(|\log\delta|/(us\kappa)\big)^{O(1)}\delta.
$$
\end{itemize}
\end{thm}
We will prove Theorem \ref{discretizedSeveri} (or actually a slightly more technical generalization) in Section \ref{mainProofSection} below. As a corollary of (the more technical version of) Theorem \ref{discretizedSeveri}, we obtain the following discretized analogue of Corollary \ref{SeveriCor}.
\begin{cor}[Few directions near a low-degree variety]\label{mainThm}
Let $P\in\RR[x_1,\ldots,x_4]$ be a polynomial of degree $D$ and let $Z= Z(P)\cap B(0,1)$. Let $0<\delta<1$ and define 
$$
\Sigma = \{\ell\in\lines\colon |\ell\cap N_{\delta}(Z)|\geq 1 \}.
$$

Then for each $\eps>0$,
\begin{equation}\label{fewDirectionsInV}
\mathcal{E}_{\delta}\big(\operatorname{Dir}(\Sigma)\big) \lesssim_{D,\eps} \delta^{-2-\eps}.
\end{equation}
\end{cor}
In brief, Corollary \ref{mainThm} follows from Theorem \ref{discretizedSeveri} by choosing $s=|\log\delta|^{-C_1},\ u = |\log\delta|^{-C_2},\ \kappa=|\log\delta|^{-C_3}$, where $C_1,C_2,C_3$ are large absolute constants. Since the lines contained in a quadratic hypersurface in $\RR^4$ can point in few directions, the set of directions of lines in $\Sigma_4^\prime$ is small. Using a slightly more technical version of Theorem \ref{discretizedSeveri}, we can also guarantee that the set of directions of lines in $\Sigma_4$ is small. The lines in $\Sigma_1,\Sigma_2,$ and $\Sigma_3$ are handled by re-scaling and induction on scales. Corollary \ref{mainThm} will be proved in detail in Section \ref{proofOfMainCorSection}.
\begin{rem}
We could replace the condition $|\ell\cap N_{\delta}(Z)|\geq 1$ in the definition of $\Sigma$ with $|\ell\cap N_{\delta}(Z)|\geq c$ for any fixed constant $c>0$. Then the implicit constant in \eqref{fewDirectionsInV} would also depend on $c$.
\end{rem}
%
%
%
% \begin{rem}
% Corollary \ref{mainThm} is sharp up to the $\delta^{-\eps}$ loss, since if $Z(P)$ is a plane, quadratic hypersurface, or a hypersurface ruled by planes, then $\mathcal{E}_{\delta}\big(\operatorname{Dir}(\Sigma)\big)\sim \delta^{-2}$. 
% \end{rem}

In \cite{G}, Guth stated the following conjecture
\begin{conj}\label{guthConj}
Let $Z\subset\RR^d$ be a $m$-dimensional variety defined by polynomials of degree at most $D$, and let $\Sigma$ be the set of lines in $\RR^d$ satisfying $|\ell\cap N_{\delta}(Z)\cap B(0,1)|\geq 1$. Then for each $\delta>0$ and $\eps>0$, the set of directions of lines in $\Sigma$ can be covered by $O_{d,D,\eps}(\delta^{1-m})$ balls of radius $\delta$.
\end{conj}
When $m=2$, Conjecture \ref{guthConj} is straightforward, and the result was used by Guth in \cite{G2} to obtain improved restriction estimates in $\RR^3$. Corollary \ref{mainThm} proves Conjecture \ref{guthConj} in the case $d=4,m=3$. For $m\geq 4$ the conjecture remains open.

\emph{Addendum added April 2018}: Conjecture \ref{guthConj} was recently proved in all dimensions by Katz and Rogers \cite{KR}.
\subsection{Progress on the Kakeya conjecture}
Recall the Kakeya maximal function conjecture. In the statement below, a $\delta$-tube is the $\delta$-neighborhood of a unit line segment.
\begin{conj}\label{KakeyaConj}
Let $\tubes$ be a set of $\delta$-tubes in $\RR^d$ that point in $\delta$-separated directions. Then for each $\eps>0$,
\begin{equation}\label{generalKakeyaConj}
\Big\Vert\sum_{T\in\tubes}\chi_T\Big\Vert_{p^\prime} \lesssim_\eps\delta^{1-d/p-\eps},\quad p = d.
\end{equation}
\end{conj}
Conjecture \ref{KakeyaConj} has been proved when $d=2$ by C\'ordoba \cite{C} and remains open for $d\geq 3$. If the exponent $p=d$ in \eqref{generalKakeyaConj} is replaced by a number $1\leq p\leq d$, then  \eqref{generalKakeyaConj} is called a Kakeya maximal function estimate in $\RR^d$ at dimension $p$.

Using the results of Guth and the author from \cite{GZ}, Theorem \ref{mainThm} can be used to obtain a Kakeya maximal function estimate in $\RR^4$ at dimension $3+1/28$.
\begin{thm}\label{maximalFnThm}
Let $\tubes$ be a set of $\delta$-tubes in $\RR^4$ that point in $\delta$-separated directions. Then for each $\eps>0$, 
\begin{equation}\label{maximalFnBd}
\Big\Vert\sum_{T\in\tubes}\chi_T\Big\Vert_{p^\prime} \lesssim_\eps\delta^{1-4/p-\eps},\quad p = 3 + 1/28.
\end{equation}
\end{thm}
\begin{cor}\label{dimOfBesicovitch}
Every Besicovitch set in $\RR^4$ has Hausdorff dimension at least $3+1/28.$
\end{cor}
Theorem \ref{maximalFnThm} will be proved in Section \ref{KakeyaSection}. It is an improvement over the previous bound $p=3$, which was due to Wolff \cite{W}. Previously, \L{}aba and Tao \cite{LT} proved that every Besicovitch set in $\RR^4$ must have upper Minkowski dimension at least $3+\eps_0$ for some positive constant $\eps_0>0$\footnote{Without attempting to optimize their arguments, \L{}aba and Tao obtained the estimate $\eps_0\geq 2^{-30}$. A careful analysis of their methods would likely yield a larger value of $\eps_0$.}. In a similar vein, Tao \cite{T} proved that every Kakeya set in $\mathbb{F}_p^4$ must have cardinality at least $cp^{3+1/16}$. This was later improved by Dvir \cite{D} and then Dvir, Kopparty, Saraf, and Sudan \cite{DKSS}, who proved nearly sharp bounds on the size of Kakeya sets in $\mathbb{F}_p^n$ for every $n$.

\subsection{Progress on the restriction conjecture}\label{progRestrictionSec}
Let $f\colon [-1,1]^{d-1}\to\CC$. For each $x = (\underline x, x_d)\in\RR^d$, define the extension operator $Ef$ by 
$$
Ef(x)=\int_{[-1,1]^{d-1}}f(\xi) e^{\xi\cdot\underline x + |\xi|^2 x_n }d\xi.
$$
The restriction conjecture for the paraboloid relates the size of $f$ and $Ef$. 
\begin{conj}\label{restrictionConjStatement}
For each $q>\frac{2d}{d-1}$ and each $f\colon [-1,1]^d\to \CC$, we have
\begin{equation}\label{restrictionConj}
\Vert Ef\Vert_{q}\lesssim_{q,d}\Vert f\Vert_{\infty}.
\end{equation}
\end{conj}
Conjecture \ref{restrictionConjStatement} has been proved when $d=2$ by Fefferman and Zygmund \cite{F, Z}. For $d\geq 2$ the problem remains open; the current best bounds are due to Guth \cite{G2, G}. When $d=4$, \eqref{restrictionConj} is known for $q>14/5$. In \cite{D}, Demeter proved that improvements to the Kakeya maximal function conjecture in $\RR^4$ would lead to progress on the restriction conjecture. 
\begin{thm}[\cite{D}, Theorem 1.4]\label{DThm}
Let $d=4$. If \eqref{generalKakeyaConj} holds for some $p > 3$, then \eqref{restrictionConj} holds for some $q<14/5$.  
\end{thm}
When Theorems \ref{maximalFnThm} and \ref{DThm} are combined, they yield an improved restriction estimate for the paraboloid in $\RR^4$. Inserting the exponent $p = 3+1/28$ into Demeter's argument yields the exponent $q = \frac{14}{5} - \frac{2}{416515}.$

\subsection{Proof sketch}
In this section we will survey the main ideas in the proof of Theorem \ref{discretizedSeveri}. Let $P\in\RR[x_1,\ldots,x_4]$ be a polynomial of degree at most $D$ and let $Z = Z(P) \cap B(0,1)$. We wish to understand the set of lines that satisfy $|\ell\cap N_{\delta}(Z)|\geq 1$. 

For the purposes of this sketch, we will assume that $\nabla P(z)\sim 1$ for all $z\in Z$ and that $|P(x)|\leq \delta$ for all $x\in N_{\delta}(Z)$. While these assumptions certainly need not hold in general, a reduction performed in Section \ref{replacingPSection} allows us to reduce to the case where a similar (though slightly more technical) statement holds.

Let $\ell$ be a line satisfying $|\ell\cap N_{\delta}(Z)|\geq 1$ and let $x\in \ell\cap N_{\delta}(Z)$. Let $\ell(t)$ be a unit speed parameterization of $\ell$ with $\ell(0)=x$. Then $P(\ell(t))$ is a univariate polynomial of degree at most $D$, and $|P(\ell(t))|$ is small for many values of $t$. More precisely, we have
$$
|\{t\in [-1,1]\colon |P(\ell(t))|\lesssim\delta \}|\gtrsim 1.
$$
This means that all of the coefficients of $P(\ell(t))$ have magnitude $\lesssim \delta$ (the implicit constant may depend on $D$, but we will suppress this dependence here). In particular, if $v$ is a unit vector pointing in the same direction as $\ell$ and if $z\in Z$ satisfies $\operatorname{dist}(x,z)\leq\delta$, then 
\begin{equation}\label{smallGradient}
|v\cdot \nabla P(z)|\lesssim\delta\quad\textrm{and}\quad|(v\cdot \nabla)^2 P(z)|\lesssim\delta.
\end{equation} 
For each $z\in Z(P)$, the set of vectors $\{v\in S^3\colon v\cdot\nabla P(z) = 0,\ (v\cdot\nabla)^2 P(z) = 0\}$ is called the quadratic cone of $Z(P)$ at $z,$ and it is closely related to the second fundamental form of $Z(P)$ at $z$. 

We wish to understand the relationship between the set of vectors satisfying \eqref{smallGradient} and the set of vectors in the quadratic cone of $Z(P)$ at $z$. It thus seems reasonable to ask: if $z\in Z$ and if a unit vector $v\in S^3$ satisfies \eqref{smallGradient}, must it be the case that $v$ is contained in the $\lesssim\delta$-neighborhood of the quadratic cone of $Z(P)$ at $z$?

In short, the answer is no. As a first example, consider the polynomial $P_1(x_1,x_2,$ $x_3,x_4) = x_1 + \delta x_2^2$, and let $z=0$. Then the quadratic cone of $Z(P_1)$ at $z=0$ is $\{(v_1,v_2,v_3,v_4)\in S^3 \colon v_1 = v_2 = 0\}.$ However, the set of vectors satisfying \eqref{smallGradient} is much larger; it is comparable to $\{(v_1,v_2,v_3,v_4)\in S^3 \colon |v_1|\lesssim\delta \}.$ In this example, the $\delta$-neighborhood of $Z(P_1)\cap B(0,1)$ is comparable to the $\delta$-neighborhood of a hyperplane. In Section \ref{wellCurvedReductionSec} we will expand on this observation. We will prove a technical variant of the following idea: if the coefficients of the second fundamental form of $Z(P)$ are all very small at a typical point, then large pieces of $Z(P)\cap B(0,1)$ can be contained in the thin neighborhood of a hyperplane. The lines having large intersection with these pieces will end up in the set $\Sigma_3$ from the statement of the theorem. 

As a second example, consider the polynomial $P_2(x_1,x_2,x_3,x_4) = x_1 + x_2^2$ and let $z=0$. Then the quadratic cone of $Z(P_2)$ at $z=0$ is again $\{(v_1,v_2,v_3,v_4)\in S^3 \colon v_1 = v_2 = 0\},$ while the set of vectors satisfying \eqref{smallGradient} is comparable to $\{(v_1,v_2,v_3,v_4)\in S^3 \colon |v_1|\lesssim\delta,\ |v_2|\lesssim \delta^{1/2}\}.$ In this example, at least one coefficient of the second fundamental form of $Z(P_2)$ at $z=0$ has large magnitude, and the quadratic cone of $Z(P_2)$ at $z=0$ is a plane. In Section \ref{stronglyBroadReductionCase}, we will show that if at least one coefficient of the second fundamental form of $Z(P)$ is large at a typical point, and if the quadratic cone is comparable to either a plane or a union of two planes, then most of the lines contained in $N_{\delta}(Z(P)\cap B(0,1))$ can be covered by a union of one and two dimensional prisms. These lines will end up in the sets $\Sigma_1$ and $\Sigma_2$ from the statement of the theorem. 

Now let $z\in Z(P)\cap B(0,1)$ and suppose that at least one coefficient of the second fundamental form of $Z(P)$ at $z$ is large and that the quadratic cone of $Z(P)$ at $z$ is not comparable to either a plane or a union of two planes. Then the set of vectors satisfying \eqref{smallGradient} is comparable to the $\delta$-neighborhood of the quadratic cone of $Z(P)$ at $z$. In Section \ref{broadVarietiesSection}, we will show that if this happens at a typical point, then morally speaking $Z(P)$ must be a quadratic hypersurface. More precisely, many of the lines lying near $N_{\delta}(Z(P)\cap B(0,1))$ are almost contained in a quadratic hypersurface. These lines will be end up in the set $\Sigma_4$ from the statement of the theorem.

\subsection{Thanks}
The author would like to thank Larry Guth and Ciprian Demeter for helpful discussions.
\section{A primer on real algebraic geometry}
Our proof will use several facts about semi-algebraic sets. Further background can be found in \cite{BCR}. A semi-algebraic set is a boolean combination of sets of the form
\begin{equation}\label{defnSemiAlg}
S=\{x\in\RR^d\colon P_1(x)=0, \ldots, P_{k}(x)=0,\ Q_1(x)>0,\ldots,Q_{\ell}(x)>0\},
\end{equation}
where $P_1,\ldots,P_k$ and $Q_1,\ldots,Q_{\ell}$ are polynomials. We define the complexity of $S$ to be the minimum of $\deg(P_1)+\ldots+\deg(P_k)+\deg(Q_1)+\ldots+\deg(Q_\ell)$, where the minimum is taken over all representations of $S$ of the form \eqref{defnSemiAlg}. We define the complexity of a semi-algebraic set to be the sum of the complexities of its constituent components of the form \eqref{defnSemiAlg} above.

If $S,T\subset\RR^d$ are semi-algebraic sets of complexity $E_1$ and $E_2$ respectively, then $S\cup T,\ S\cap T,$ and $S\backslash T$ are semi-algebraic sets of complexity $O_{d,E_1,E_2}(1)$. If $\pi\colon\RR^{d}\to\RR^e$ is a projection, then $\pi(S)$ is a semi-algebraic set of complexity $O_{d,E_1}(1)$. 

If $S\subset\RR^d,\ T\subset\RR^e$ are semi-algebraic sets, we say that a function $f\colon S\to T$ is semi-algebraic of complexity $E$ if the graph of $f$ is a semi-algebraic set of complexity $E$. Clearly if $f\colon S\to T$ is a semi-algebraic bijection of complexity $E$, then $f^{-1}\colon T\to S$ is also a semi-algebraic bijection of complexity $E$.

If $S\subset\RR^d$ is a semi-algebraic set, we define its dimension $\dim(S)$ to be the largest integer $e$ so that there is a subset $S^\prime\subset S$ that is homeomorphic to the open $e$-dimensional cube $(0,1)^e$. If $S\subset \RR^d$ has dimension $e$ and complexity $E$, then there is an $e$-dimensional real algebraic variety $S\subset Z\subset\RR^d$ that is defined by polynomials of degree $O_{d,E}(1)$. If $S$ and $T$ are semi-algebraic sets, and if there is a semi-algebraic bijection $f\colon S\to T$, then $S$ and $T$ have the same dimension.

Occasionally, we will refer to semi-algebraic subsets of the sphere $S^d$ or the affine Grassmannian $\operatorname{Grass}(d;e)$ of $e$-dimensional affine subspaces of $\RR^d$. To do this, we will identify $S^d$ or $\operatorname{Grass}(d;e)$ with a semi-algebraic set in $\RR^N$ for some $N=O_{d}(1)$.

In the remainder of this section, we will list several standard results about real algebraic sets that will be used throughout the proof.
\begin{lem}[Milnor-Thom theorem]\label{MTTheorem}
Let $S\subset\RR^d$ be a semi-algebraic set of complexity at most $E$. Then $S$ has $O_{E,d}(1)$ connected components. 
\end{lem}
This is a variant of the Milnor-Thom Theorem \cite{M,T}. See Barone-Basu \cite{BB} for the above formulation.
\begin{lem}[Wongkew \cite{Wo}]
Let $Z\subset \RR^d$ be a real algebraic variety of dimension $e$ that is defined by polynomials of degree at most $D$. Then for each $u>0$, we have
$$
|N_{u}(Z\cap B(0,1))|\leq\sum_{j=0}^e C_{d,j} (Du)^{d-e},
$$
where the numbers $C_{d,j}$ are constants depending only on $d$ and $j$.
\end{lem}
\begin{cor}
Let $S\subset B(0,1)\subset \RR^d$ be a semi-algebraic set of dimension $e$ and complexity at most $E$. Then for each $u>0$, we have
$$
|N_u(S)|\lesssim_{E,d} u^{d-e}.
$$
\end{cor}
Since $\mathcal{E}_{u}\big(N_{w}(S)\big)\lesssim u^{-d}|N_{u+w}(S)|$, we obtain the following corollary.
\begin{lem}[Covering number of neighborhoods of semi-algebraic sets]\label{coveringNumberSemiAlg}
Let $S\subset B(0,1)\subset \RR^d$ be a semi-algebraic set of dimension $e$ and complexity at most $E$. Then for each $u>0$, we have
$$
\mathcal{E}_u(N_{u}(S)) \lesssim_{E,d} u^{-e}.
$$
If $0<u<w$, then
$$
\mathcal{E}_u(N_{w}(S)) \lesssim_{E,d} u^{-d} w^{d-e} =  u^{-e} (w/u)^{e-d}.
$$
\end{lem}
A similar result can be found in \cite[Theorem 5.9]{Yo}.
The next lemma is a quantitative version of the statement that every connected smooth semi-algebraic set is path connected.
\begin{lem}\label{semiAlgebraicPath}
Let $S\subset\RR^d$ be a semi-algebraic set of complexity at most $E$ and diameter at most one. Suppose as well that $S$ is a connected smooth manifold. Then any two points in $S$ can be connected by a smooth curve of length $O_{d,E}(1)$. 
\end{lem}
We will also need the following multi-dimensional Remez-type inequality. See, e.g. \cite[Theorem 2]{BG}.
\begin{lem}\label{multiDimRemez}
Let $P\in\RR[x_1,\ldots,x_d]$ be a polynomial of degree at most $D$. Let $\Omega\subset\RR^d$ be an open convex set. Let $m = \sup_{x\in\Omega}|P(x)|$. Then for each $0<\lambda<1$, 
\begin{equation}\label{volumeEstimate}
|\{x\in\Omega\colon |P(x)|\leq\lambda m \}|\leq 4d |\Omega|\lambda^{1/D}.
\end{equation}
\end{lem}
\begin{cor}\label{multiDimRemezSphere}
Let $P\in\RR[x_1,\ldots,x_d]$ be a homogeneous polynomial of degree $D$. Let $m = \sup_{|x|=1}|P(x)|$. Then for each $0<\lambda<1$, 
\begin{equation}\label{volumeEstimate}
\mu_{\sigma}\Big(\{x\in S^{d-1}\colon |P(x)|\leq\lambda m \}\Big)\lesssim_{d} \lambda^{1/O_D(1)},
\end{equation}
where $\mu_{\sigma}$ is the Haar probability measure on the sphere. 
\end{cor}
\begin{lem}[Selecting a point from each fiber]\label{selPtFiberProp}
Let $X\subset\RR^d$ be a semi-algebraic set and let $f\colon X\to Y$ be a semi-algebraic map. Suppose that both $X$ and $f$ have complexity at most $E$. Then there is a semi-algebraic set $X^\prime\subset X$ of complexity $O_{E,d}(1)$ so that $f(X^\prime) = f(X)$ and $f\colon X^\prime\to Y$ is an injection.
\end{lem}
\begin{proof}
Define a semi-algebraic ordering ``$>$'' on $\RR^d$ with the following properties. (A): If $x,x^\prime\in\RR^d$ then exactly one of the following holds: $x>x^\prime,\ x=x^\prime,$ or $x<x^\prime$. (B): The set  
$$
O = \{(x,x^\prime)\in\RR^d\times\RR^d\colon x<x^\prime\}
$$
is semi-algebraic of complexity $O_d(1)$.

An example of such an ordering is the lexicographic order on $x = (x_1,\ldots,x_d)$. Observe that
$$
\{(x,x^\prime)\in X\times X\colon f(x) = f(x^\prime)\}
$$
is semi-algebraic of complexity $O_{E}(1)$. Thus
$$
B = \{(x,x^\prime) \in X\times X \colon\ f(x) = f(x^\prime),\ x<x'\}
$$
is semi-algebraic of complexity $O_{E,d}(1)$. Let $\pi\colon X\times X\to X$ be the projection to the first coordinate. Then
$$
X^\prime = X\backslash \pi(B)=\{x\in X\colon x\geq x^\prime\ \forall\ x^\prime \in\ X\ \textrm{with}\ f(x) = f(x^\prime)\}
$$
is semi-algebraic of complexity $O_{E,d}(1)$. Note that $f(X^\prime) = f(X)$ and that $f\colon X^\prime\to Y$ is an injection. Indeed, if $x,x^\prime\in W$ with $f(x) = f(x^\prime)$ then $x\geq x^\prime$ and $x^\prime\geq x$, which implies $x=x^\prime$.
\end{proof}
\subsection{Bundles of lines}\label{bundlesOfLinesSection}
The map $\ell\mapsto\operatorname{Dir}(\ell)$ described in Definition \ref{directionOfLinesDef} is badly behaved for lines lying in (or near) the hyperplane $x_1 = 0$. To avoid this issue, we will only consider lines $\ell\in\lines$ that make a small angle with the $e_1$ direction. Abusing notation slightly, we will re-define $\lines$ to be the set of lines in $\RR^4$ that make an angle $\leq 1/10$ with the $e_1$ direction. 
\begin{defn}
For $Z\subset\RR^4$ and $0<\delta<c$, define
$$
\Sigma_{\delta,c}(Z)=\{\ell\in\lines\colon |\ell\cap N_{\delta}(Z)|\geq c\}.
$$
Define $\Sigma_{\delta}(Z)=\Sigma_{\delta,1}(Z)$.
\end{defn}
Let $\ell\in\lines$. Define $v(\ell)$ to be the unit vector $v\in \RR^4$ that points in the same direction as $\ell$ and satisfies $\angle(v,e_1)\leq 1/10$.
\begin{defn}
Let $\Sigma\subset\lines$ be a set of lines. For each $x\in\RR^4$, define 
$$
\Sigma(x) = \{\ell\in\Sigma\colon x\in\ell\}.
$$
\end{defn}
\begin{defn}\label{defnGammaZSigma}
Let $Z\subset\RR^4$ and let $\Sigma\subset\lines$  be a set of lines. Define
$$
\Gamma(Z,\Sigma)=\{(x,\ell)\in Z\times\Sigma\colon x\in\ell\}.
$$

If $\Gamma\subset\Gamma(Z,\Sigma)$, then for each $x\in Z$ define
$$
\Gamma(x) = \{\ell\in\lines\colon (x,\ell)\in \Gamma\},
$$
and for each $\ell\in\Sigma$, define
$$
\Gamma(\ell) = \{x \in \RR^4 \colon (x,\ell)\in \Gamma\}.
$$
\end{defn}
%
%
%
% \begin{defn}
% Let $Z\subset\RR^4,$ $\Sigma\subset\lines$, and $\Gamma\subset\Gamma(Z,\Sigma)$. For each $x\in Z$ and $r>0$, define

% $$
% \Gamma_r(x)=\bigcup_{x^\prime\in B(x,r)}\Gamma(x^\prime)=\{\ell\in\lines\colon (x^\prime,\ell)\in \Gamma\ \textrm{for some}\ x^\prime\in B(x,r)\}.
% $$
% In practice, we will have $r = K\delta$ for $K$ a small positive integer.
% \end{defn}
%
%
%
%We will sometimes abuse notation and identify $\Gamma$ with a subset of $Z\times S^3$ via the bijection $(x,\ell)\mapsto(x, v(\ell))$. We will also abuse notation and identify $\Gamma_u$ with a subset of $Z\times S^3$ via the map $(x,\ell)\mapsto (x,v(\ell))$; note that this is not a bijection, since if $(x,\ell)\in\Gamma_u$ then the line $\ell$ need not contain the point $x$.

%
%
%
%
%
%
%
%
%
%
\section{Replacing $P$ by a better-behaved polynomial}\label{replacingPSection}
In this section we will perform a convenient technical reduction. Informally speaking, this reduction lets us assume that the polynomial $P$ from the statement of Theorem \ref{discretizedSeveri} obeys the bounds
\begin{equation}\label{wishfullThinking}
\begin{split}
&|\nabla P(x)|\sim 1\quad\ \textrm{for all}\ x\in Z(P)\cap B(0,1),\\
&|P(x)|\lesssim\delta\quad\ \ \ \phantom{.} \textrm{for all}\ x\in N_{\delta}(Z(P))\cap B(0,1).
\end{split}
\end{equation}
If \eqref{wishfullThinking} were true, it would yield a lot of useful information about the unit line segments contained in $N_{\delta}(Z(P))$. For example, if $\ell\in\Sigma_\delta(Z)$ and $x \in Z(P)\cap N_{\delta}(\ell)\cap B(0,1)$, then $v(\ell)$ is almost contained in the tangent plane $T_x(Z(P))$. While \eqref{wishfullThinking} need not be true, the following proposition will still allow us to recover some of the useful consequences of \eqref{wishfullThinking}.

\begin{prop}\label{goodPolyProp}
Let $P$ be a polynomial of degree at most $D$. Let $Z = Z(P)\cap B(0,1)$, let $\delta>0$, and let $\Sigma\subset \Sigma_{\delta}(Z)$ be a semi-algebraic set of complexity at most $E$. Then there exist sets $\Sigma_1,\ldots,\Sigma_b$; polynomials $P_1,\ldots,P_b$, and sets $\Gamma_1,\ldots,\Gamma_b$ so that the following holds.
\begin{enumerate}
\item\label{sizeOfB} $b\lesssim_{D,E}|\log\delta|$.
\item\label{semiAlgCplxty} Each $P_j$ has degree at most $D$. Each set $\Sigma_j$ and $\Gamma_j$ is semi-algebraic of complexity $O_{D,E}(1)$.
\item\label{GammaJInZJ} $\Gamma_j\subset \Gamma(N_{\delta}(Z_j),\Sigma_j)$, where
$$
Z_j = \{x\in Z(P_j)\cap B(0,1)\colon 1\leq|\nabla P_j(x)|\leq 2\}.
$$
\item\label{linesCover} $\Sigma=\bigcup_{j=1}^b \Sigma_j$ %The lines in $\Sigma_1,\ldots,\Sigma_b$ ``cover'' the lines in $\Sigma$ at scale $\delta$: For each $\ell\in\Sigma$, there is an index $j$ and a line $\ell^\prime\in\Sigma_j$ with $\operatorname{dist}(\ell,\ell^\prime)\leq\delta$. 
\item\label{GammaJCovers} For each $\ell\in\Sigma_j,$ we have $|\Gamma_j(\ell)|\gtrsim_{D}|\log\delta|^{-1}$.
%\item\label{nablaPOnEll} If $(x,\ell)\in\Gamma_j$, then $1\lesssim_D |\nabla P_j(x)|\leq 2$.
\item\label{smallDerivatives} For each $(x,\ell)\in \Gamma_j$, each $y\in Z_j$ with $\operatorname{dist}(x,y)\leq\delta$, and each non-negative integer $i$, we have 
\begin{equation}\label{smallDirectionalDerivative}
|(v(\ell)\cdot \nabla)^i P_j(y)| \lesssim_{D,i} |\log\delta|^i \delta.
\end{equation}
\end{enumerate} 
\end{prop}
 \begin{proof}
First, we can assume without loss of generality that the largest coefficient of $P$ has magnitude 1. If $Z(P)\cap B(0,1)=\emptyset$ then $Z = \emptyset$ and thus $\Sigma=\emptyset$ so the result is trivial. Thus we can also assume that at least one non-constant coefficient of $P$ has magnitude $\gtrsim_D 1$.

For each point $z\in B(0,1)$, define 
$$
m(z)=\sup_{x\in B(z,\delta)}|\nabla P(x)|.
$$

Observe that $|\nabla P(x)|^2$ is a polynomial, all of whose coefficients have magnitude $\lesssim_D 1$ and at least one coefficient has magnitude $\gtrsim_D 1$. Thus   
\begin{equation}\label{boundOnSizeOfMZ}
\delta^{O_D(1)}\lesssim_D m(z)\lesssim_D 1\quad \textrm{for all}\ z\in B(0,1).
\end{equation}

For each $z\in N_{\delta}(Z)\cap B(0,1)$, we have 
$$
P(B(z,\delta))\subset \big[-\delta m(z),\ \delta m(z)\big].
$$ 
By Lemma \ref{multiDimRemez}, there is a constant $c \gtrsim_D 1$ so that for each $z\in B(0,1)$, 
\begin{equation}\label{mostOfGradientIsLarge}
\big|\big\{ x\in B(z,\delta)\} \colon |\nabla P(x)| \geq c m(z)\big\}\big|\geq \frac{99}{100} |B(x,\delta)|.
\end{equation}
This implies that $P(B(z,\delta))$ contains an interval of length $\gtrsim_D \delta m(z)$.
%
%  Note that if $z_1,z_2\in\RR^4$ satisfy $|z_1-z_2|<\delta$, then $|B(z_1,\delta)\cap B(z_2,\delta)|\geq \frac{1}{10}|B(z_1,\delta)|$. Thus for each $z\in N_{\delta}(Z)\cap B(0,1)$, we have 
% $$
% |B(z,\delta)\cap N_{\delta}(Z)|\geq \frac{1}{10}|B(z,\delta)|.
% $$
% Combining this inequality with \eqref{mostOfGradientIsLarge}, we have that for each $z\in N_\delta(Z)\cap B(0,1)$,
% \begin{equation}\label{mostOfBallIsGood}
% \big|\big\{ x\in B(z,\delta)\cap N_{\delta}(Z)\} \colon |\nabla P(x)| \geq c m(z)\big\}\big|\geq \frac{1}{100} |B(z,\delta)|.
% \end{equation}
%
Let $m_1,\ldots,m_b,$ $b\lesssim_{D}|\log\delta|$ be a geometric sequence of non-negative numbers with $m_1 = \delta^{O_D(1)},\ m_b = O_D(1)$ and $m_{j+1}=2m_j$. If we select $m_1$ and $b$ appropriately, then for each $z\in N_{\delta}(Z)\cap B(0,1)$ there exists an index $j$ so that $m_j\leq m(z)< m_{j+1}.$ For each index $j$, let $w_{j,1},\ldots,w_{j,h},\ h = O_D(1)$ be real numbers in $[-\delta m(z), \delta m(z)]$ so that for every $z\in N_{\delta}(Z)\cap B(0,1)$ satisfying $m_j\leq m(z)< m_{j+1}$, we have 
$$
\{w_{j,1},\ldots,w_{j,h}\} \cap P(B(z,\delta))\neq\emptyset,$$
i.e. at least one of the values $w_{j,1},\ldots, w_{j,h}$ is contained in $P(B(z,\delta))$. This can always be achieved since $P(B(z,\delta))$ must contain an interval of length $\gtrsim_D \delta m(z)$. We can select $h=O_D(1)$ to be independent of $j$. Observe that there are $O_{D}(|\log\delta|)$ pairs of indices $j,h$, which establishes Item \ref{sizeOfB} in the statement of the lemma (later in the proof we will re-index the pairs $(j,h)$ to use a single indexing variable).

For each $j=1,\ldots, b$ and each $k = 1,\ldots, h$, define
$$
X_{j,k} = \{ z\in N_{\delta}(Z)\cap B(0,1)\colon m_j \leq m(z) < m_{j+1},\ \textrm{and}\ w_{j,k}\in P\big(B(z,\delta)\big)\}.
$$
Then each set $X_{j,k}$ is semi-algebraic of complexity $O_D(1)$, and 
$$
N_{\delta}(Z)\cap B(0,1)= \bigcup_{j=1}^b \bigcup_{k=1}^h X_{j,k}. 
$$

For each $\ell\in\Sigma$, there exists indices $j,k$ so that $|\ell\cap X_{j,k}|\geq c_1/|\log\delta|$, where $c_1\gtrsim_D 1$.
%  Define
% $$
% X_{j,k}^\prime = X_{j,k}\cap \{|\nabla P(x)| \geq c m_j\},
% $$
% where $c$ is the constant from \eqref{mostOfGradientIsLarge}. By \eqref{mostOfBallIsGood}, we have
% $$
% |N_{\delta}(\ell)\cap X_{j,k}^\prime|\geq\frac{c_1}{100}|\log\delta|\ |N_{\delta}(\ell)\cap B(0,1)|.
% $$
% Thus by Fubini's theorem, there exists a line $\ell^\prime$ with $\operatorname{dist}(\ell,\ell^\prime)<\delta$ and 
% $$
% |\{ \ell^\prime\cap X_{j,k}^\prime\}|\geq \frac{c_1}{100}|\log\delta|.
% $$ 
% Since the set $\ell^\prime\cap X_{j,k}^\prime$ is semi-algebraic of complexity $O_D(1)$, it contains an interval of length $\geq c_2/|\log\delta|$, where $c_2\gtrsim_D 1$. 
Since the set $\ell\cap X_{j,k}$ is semi-algebraic of complexity $O_D(1)$, it contains an interval of length $\geq c_2/|\log\delta|$, where $c_2\gtrsim_D 1$. Define
$$
\Sigma_{j,k}=\{\ell\in \Sigma\colon \ell\cap X_{j,k}\ \textrm{contains an interval of length}\ \geq c_2/|\log\delta|\}.
$$
% where
% $$
% N_{\delta}(\Sigma)=\{\ell^\prime\in\lines\colon \operatorname{dist}(\ell,\ell^\prime)<\delta\ \textrm{for some}\ \ell\in\Sigma\}.
% $$
With this definition, each set $\Sigma_{j,k}$ is semi-algebraic of complexity $O_{D,E}(1)$ and $\Sigma = \bigcup\Sigma_{j,k}$, so Item \ref{linesCover} in the statement of the lemma is satisfied.

Define
$$
\Gamma_{j,k}=\{(x,\ell)\colon \ell\in\Sigma_{j,k},\ x\ \textrm{is contained in an interval in}\ \ell\ \cap\ X_{j,k}\ \textrm{of length}\ \geq c_2/|\log\delta|\}.
$$
Then $|\Gamma_{j,k}(\ell)|\gtrsim_{D}|\log\delta|^{-1}$ for each $\ell\in\Sigma_{j,k}$, so Item \ref{GammaJCovers} is satisfied. Define
$$
P_{j,k}(z) = m_j^{-1} (P(z) - h_{j,k});
$$ 
each polynomial $P_{j,k}(z)$ has degree $\leq D$, so Item \ref{semiAlgCplxty} is satisfied. If $(x,\ell)\in \Gamma_{j,k}$, then
\begin{equation}\label{P1SmallOnW}
|P_{j,k}(y)|\leq 4\delta\quad\ \textrm{for all}\ y\in B(x,\delta).
\end{equation}
This is because $m(x) < 2m_j$, so $|\nabla P_{j,k}(y)|\leq 2$ for all $y\in B(x,\delta)$. \eqref{P1SmallOnW} then follows from the fact that  $B(x,\delta)\cap N_{2\delta}(Z(P_{j,k}))\neq\emptyset$.

%Next, observe that if $(x,\ell)\in \Gamma_{j,k}$, then $1\lesssim_D |\nabla P_{j,k}(x)|\leq 2$, so Item \ref{nablaPOnEll} is satisfied. 
We have that if $(x,\ell)\in \Gamma_{j,k}$, then $x\in N_{\delta}(Z_{j,k})$, where
$$
Z_{j,k} = \{z \in Z(P_{j,k})\cap B(0,1)\colon 1\leq |\nabla P_{j,k}(z)|\leq 2\}.
$$
Thus Item \ref{GammaJInZJ} is satisfied.  

It remains to verify Item \ref{smallDerivatives}. Fix indices $j,k$ and let $\Gamma = \Gamma_{j,k}$. Let $(x,\ell)\in \Gamma$ and let $y\in Z_{j,k}$ with $\operatorname{dist}(x,y)\leq\delta$. Then there is a line segment $I\subset \Gamma(\ell)$ of length $\gtrsim_D |\log\delta|^{-1}$ containing $x$. By \eqref{P1SmallOnW} we have that $|P_{j,k}(z)|\leq 4\delta$ for all $z\in N_{\delta}(I)$. 
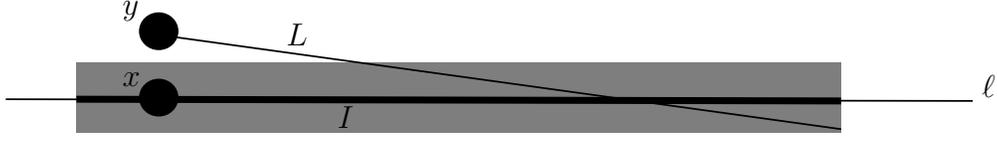
\begin{figure}[h!]
 \centering
\begin{overpic}[width=0.9\textwidth]{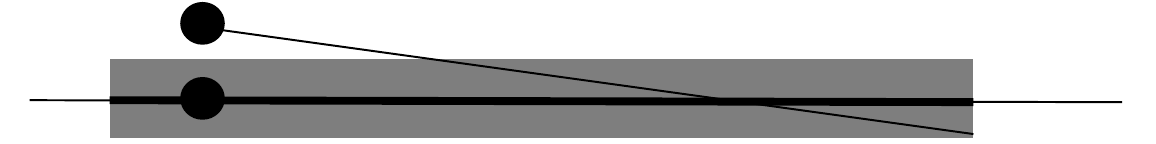}
 \put (98,4) {$\ell$}
 \put (14,12) {$y$}
 \put (14,5) {$x$}
 \put (30,9) {$L$}
 \put (35,1) {$I$}
\end{overpic}
 \caption{The points $x$ and $y$ (black circles); the lines $\ell$ and $L$ (thin black lines); the line segment $I$ (thick black line) and the region $N_{\delta}(I)$ (grey rectangle). Observe that if $\operatorname{dist}(x,y)\leq \delta$, then $|L\cap N_{\delta}(I)|\gtrsim |I|$}\label{xLineFig}
\end{figure}
Let $L$ be a line containing $y$ with $|L\cap N_{\delta}(I)|\gtrsim_D |\log\delta|^{-1}$ (see Figure \ref{xLineFig}). Let $L(t)\colon \RR\to L$ be a unit speed parameterization of $L$, with $L(0)=y$ (i.e. $L(t) = y + tv(L)$ ). Then $P_{j,k}(L(t))$ is a univariate polynomial of degree $\leq D$ that satisfies $|P_{j,k}(L(t))|\leq 4\delta$ for all $t$ in an interval $J\subset [0,1]$ of length $\gtrsim_D |\log\delta|^{-1}$. This implies that 
$$
\Big|\frac{d^i}{dt^i}P_{j,k}(L(t))|_{t=0}\Big|\lesssim_{D,i} |\log\delta|^{-i} \delta, 
$$
which gives us \eqref{smallDirectionalDerivative}. 

The sets $\{\Sigma_{j,k}\}$ and $\{\Gamma_{j,k}\}$, and the polynomials $\{P_{j,k}\}$ satisfy the conclusions of Lemma \ref{goodPolyProp}. All that remains is to re-index the indices $j,k$ to use a single indexing variable. 
\end{proof}
\section{Curved varieties and the second fundamental form}\label{wellCurvedReductionSec}
In this section, we will consider the region where $Z$ has small second fundamental form. We will show that lines lying near this region must be contained in a thin neighborhood of a hyperplane; this will be the set of lines $\Sigma_3$ from the statement of Theorem \ref{discretizedSeveri}. This result will be proved in Proposition \ref{FlatVarietiesNearAPlaneProp}, which is the main result of this section. 

%To do this, we will prove a quantitative version of the following observation: If $Z(P)\subset\RR^4$ is a low degree hypersurface, and if the second fundamental form of $Z=Z(P)\cap B(0,1)$ is small at every point of $Z$, then $Z$ can be contained in a union of thin neighborhoods of hyperplanes. 

\subsection{A primer on the second fundamental form}
Define the functions $\varphi_i\colon\RR^4\to\RR^4,\ i=0,1,2,3$ by
\begin{align*}
&\varphi_0(x_1,x_2,x_3,x_4) = (\phantom{-}x_1,\phantom{-}x_2,\phantom{-}x_3,\phantom{-}x_4),\\
&\varphi_1(x_1,x_2,x_3,x_4) = (-x_2,\phantom{-}x_1,-x_4,\phantom{-}x_3),\\
&\varphi_2(x_1,x_2,x_3,x_4) = (-x_3,\phantom{-}x_4,\phantom{-}x_1,-x_2),\\
&\varphi_3(x_1,x_2,x_3,x_4) = (-x_4,-x_3,\phantom{-}x_2,\phantom{-}x_1).
\end{align*}
Note that for each $x\in\RR^4$, $\varphi_0(x),\varphi_1(x),\varphi_2(x),$ and $\varphi_3(x)$ have the same magnitude and are orthogonal. 

Let $P\in\RR[x_1,\ldots,x_4]$ and let $x\in Z(P)$. Suppose that $\nabla P(x)\neq 0$ and that $Z(P)$ is a smooth manifold in a neighborhood of $x$. Define $N(x) =\frac{\nabla P(x)}{|\nabla P(x)|}$. For each $i,j\in\{1,2,3\}$, define 
$$
a_{ij}(x) = \Big(\varphi_i(\nabla P(x))\cdot\nabla_y\Big)\ \Big(\varphi_j(\nabla P(x))\cdot\nabla_y\Big)\ P(y)\ \Big|_{y=x}.
$$
To untangle the above definition: $\nabla P(x)$ is a vector in $\RR^4$; $\varphi_i(\nabla P(x))$ is a vector in $\RR^4$; $\big(\varphi_i(\nabla P(x))\cdot\nabla_y\big)$ is an operator acting on functions $f\colon\RR^4\to\RR$. Similarly, $\big(\varphi_j(\nabla P(x))\cdot\nabla_y\big)$ is an operator acting on functions $f\colon\RR^4\to\RR$. We apply these operators to the function $P(y)$, and then evaluate the resulting function at the point $y=x$. 

Note that for each $i,j\in\{1,2,3\}$,  $a_{ij}\in\RR[x_1,\ldots,x_4]$ is a polynomial of degree $O(\deg P)$. Define
$$
II(x)= \frac{1}{|\nabla P(x)|^3}\left[\begin{array}{ccc} 
a_{11}(x)&a_{12}(x)&a_{13}(x)\\
a_{21}(x)&a_{22}(x)&a_{23}(x)\\
a_{31}(x)&a_{32}(x)&a_{33}(x)\\
\end{array}\right].
$$
Then $II(x)$ is the second fundamental form of $Z(P)$ at $x$, written in the basis $\frac{\varphi_1(x)}{|\varphi_1(x)|},\ \frac{\varphi_2(x)}{|\varphi_2(x)|},\ \frac{\varphi_3(x)}{|\varphi_3(x)|}$ (this is a basis for the tangent space $T_x(Z(P))$ ). Note that if the polynomial $P(x)$ is replaced by $tP(x)$ for $t\neq 0$, then the matrix $II(x)$ is unchanged. For each $x\in\RR^4$, define $\Vert II(x)\Vert_{\infty}$ to be the $\ell^\infty$ norm of the entries of $II(x)$ (so $\Vert II(x)\Vert_{\infty}$ is a function from $\RR^4$ to $\RR$). Observe that if $P\in\RR[x_1,\ldots,x_4]$ is a polynomial of degree at most $D$, then for each $\kappa>0$, the set 
$$
\{x\in Z(P)\colon \Vert II(x)\Vert_{\infty}>\kappa \}
$$ 
is semi-algebraic of complexity $O_D(1)$.

Note that if $0\in Z(P)$ and if $N(0)=(0,0,0,1)$, then in a neighborhood of the origin we can write $Z(P)$ as the graph $x_4 = f(x_1,x_2,x_3)$, with $f(0,0,0)=0$ and $\nabla f(0)=(0,0,0)$. Then
$$
II(0)= \left[\begin{array}{ccc} 
\partial_{x_1x_1}f(0)&\partial_{x_1x_2}f(0)&\partial_{x_1x_3}f(0)\\
\partial_{x_2x_1}f(0)&\partial_{x_2x_2}f(0)&\partial_{x_2x_3}f(0)\\
\partial_{x_3x_1}f(0)&\partial_{x_3x_2}f(0)&\partial_{x_3x_3}f(0)\\
\end{array}\right].
$$ 
\begin{lem}
Let $P\in\RR[x_1,\ldots,x_4]$ and let $M\subset Z(P)$ be a smooth manifold. Suppose that $\nabla P(x)\neq 0$ for all $x\in M$, $\Vert II(x)\Vert_\infty\leq \kappa$ for all $x\in M$, and that every pair of points in $M$ can be connected by a smooth curve of arclength $\leq t$. Then the image of the Gauss map $N(M)$ can be contained in a ball of diameter $t \kappa$. 
\end{lem}
\begin{proof}
Fix a point $x_0\in M$. For each $x\in M$, let $\gamma(s)$ be a unit-speed paramaterization of a smooth curve with $\gamma(0)=x_0$ and $\gamma(s_0)=x$. By hypothesis, we can select such a curve with $s_0\leq t$. Then $|\frac{d}{ds}N(\gamma(s))|\leq \Vert II(\gamma(s)) \Vert_{\infty}\leq \kappa$; here $N(\gamma(s))$ is a map from $\RR$ to the unit sphere $S^3\subset\RR^4$, and $|\cdot|$ denotes the Euclidean norm of the four-dimensional vector $\frac{d}{ds}N(\gamma(s)).$ We conclude that $N(x)$ is contained in the ball (in $S^3$) centered at $x_0$ of radius $t\kappa$.
\end{proof}
\begin{lem}[Hypersurfaces with small second fundamental form lie near a hyperplane]\label{secondFundFormHyperplaneLem}
Let $S\subset\RR^4$ be a semi-algebraic set contained in $B(0,1)$ of complexity at most $E$. Suppose that $S$ is a connected smooth three-dimensional manifold and that $\Vert \II(x)\Vert_{\infty}\leq \kappa$ for all $x\in S$. Then $S$ can be contained in the $O_E(\kappa)$--neighborhood of a hyperplane.
\end{lem}
\begin{proof}
After applying a rigid transformation, we can assume that $0\in S$ and $N(0)=(0,0,0,1)$. Let $x\in S$. By Lemma \ref{semiAlgebraicPath}, we can find a smooth curve $\gamma\subset S$ of length $O_E(1)$ whose endpoints are $0$ and $x$. Let $\gamma(s)$ be a unit-speed parameterization of this curve, so $\gamma(0)=0$ and $\gamma(s_0)=x$, for some $s_0  = O_E(1)$. 

Since $\frac{d}{ds}\gamma(s)\in T_{\gamma(s)}S$, we must have $|\frac{d}{ds}\gamma(s)\cdot(0,0,0,1)|\lesssim s_0 \kappa$. In particular, the $x_4$ coordinate of $\gamma(s)$ must have magnitude $\lesssim s_0 s \kappa=O_E(\kappa)$. Thus after applying a rigid transformation, $S$ is contained in the $O_E(\kappa)$-neighborhood of the hyperplane $\{x_4=0\}$.
\end{proof}
\begin{lem}[Hypersurfaces with large second fundamental form escape every hyperplane]\label{largeFundFormEscapesPlane}
Let $P\in\RR[x_1,\ldots,x_4]$ be a polynomial of degree at most $D$ and let $Z\subset Z(P)\cap B(0,1)$ be a semi-algebraic set of complexity at most $E$. Suppose that $1\leq |\nabla P(x)|\leq 2$ and $\Vert II(x)\Vert_{\infty}\geq \kappa$ at every point $x\in Z$. Then for every hyperplane $H$ and every $0<a\leq b\leq 1$, we have
\begin{equation}\label{volumeBoundCurvedSurfaceHyperplane}
| N_{a}(Z)\cap N_{b}(H)| \lesssim_{D,E} a (b/\kappa)^{1/2}.
\end{equation}
\end{lem}
\begin{proof}
Let $\rho\gtrsim_{D} 1$ be a small parameter to be determined later. Partition $Z$ into $O_{\rho,D,E}(1)=O_{D,E}(1)$ connected semi-algebraic sets so that on each set $Z^\prime$, each of $\partial_{x_i}P(x),\ i\in \{1,2,3,4\}$ vary by at most an additive factor of $\rho$. It suffices to establish \eqref{volumeBoundCurvedSurfaceHyperplane} for each of these sets individually. Fix one of these sets $Z^\prime$. After applying a rotation, we can assume that for all $x\in Z^\prime$ we have $1-\rho\leq |\partial_{x_4}P(x)|\leq 2+\rho$ and $|\partial_{x_i}P(x)|\leq\rho,\ i=1,2,3$. In particular, each of the vectors $\varphi_i(N(x)),\ i=0,1,2,3$ are constant up to an additive factor of $\rho$ for all $x\in Z^\prime$.

Since $\Vert II(x)\Vert_{\infty}\geq \kappa$ for all $x\in Z^\prime$, and $1\leq |\nabla P(x)|\leq 2$, there is a unit vector $v\in T_x(Z^\prime)$ with $|(v\cdot\nabla)^2P(x)|\geq \kappa$. Phrased differently, there is a unit vector $(v_1,v_2,v_3)\in\RR^3$ so that
$$
\big|\big( (v_1 \varphi_1(N(x)) + v_2\varphi_2(N(x)) + v_3 \varphi_3(N(x)) )\cdot\nabla\big)^2P(x)\big|\geq \kappa.
$$ 

Since the map 
$$
(v_1^\prime,v_2^\prime,v_3^\prime) \mapsto \big|\big( (v_1^\prime \varphi_1(N(x)) + v_2^\prime \varphi_2(N(x)) + v_3^\prime \varphi_3(N(x)) )\cdot\nabla\big)^2P(x)\big|^2
$$
is homogeneous of degree four, there is a constant $c>0$ so that for all $v^\prime = (v_1^\prime,v_2^\prime,v_3^\prime)$ with $\angle(v,v^\prime)\leq c$, we have 
$$
 \big|\big( (v_1^\prime \varphi_1(N(x)) + v_2^\prime \varphi_2(N(x)) + v_3^\prime \varphi_3(N(x)) )\cdot\nabla\big)^2P(x)\big|\geq \frac{99}{100}\kappa.
$$

Thus after further partitioning $Z^\prime$ into $O_D(1)$ connected semi-algebraic sets, we have that for each such set $Z^{\prime\prime}$, there is a unit vector $v$ so that
$$
\big|\big( (v_1^\prime \varphi_1(N(x)) + v_2^\prime \varphi_2(N(x)) + v_3^\prime \varphi_3(N(x)) )\cdot\nabla\big)^2P(x)\big|\geq \frac{99}{100}\kappa
$$
for all $x\in Z^{\prime\prime}$ and all vectors $v^\prime=(v_1^\prime,v_2^\prime,v_3^\prime)\in S^2\subset\RR^3$ with $\angle(v,v^\prime)\leq c$. It suffices to establish \eqref{volumeBoundCurvedSurfaceHyperplane} for each of these sets $Z^{\prime\prime}$ individually. Fix such a set.

Fix a point $x_0\in Z^{\prime\prime}$ and apply a rigid transformation so that $N(x_0)=(0,0,0,1)$ and $v=(1,0,0)$. Note that $\varphi_1(N(x_0)) =(0,0,-1,0)$. If $\rho\gtrsim_D1$ is selected sufficiently small, then $\angle \big(\varphi_1(N(x_0)), \varphi_1(N(x))\big)\leq c/2$ for all $x\in Z^{\prime\prime}$. This means that for all $x\in Z^{\prime\prime}$, if $v^\prime\in \RR^4$ is a unit vector in $T_x(Z^{\prime\prime})$ with $\angle(v^\prime, (0,0,-1,0))\leq c/2$, then $|(v^\prime\cdot\nabla)^2 P(x)|\geq \frac{99}{100}\kappa$. 

This means that we can write $Z^{\prime\prime}$ as the graph of a function $f(x_1,x_2,x_3)$; more precisely, for each $(x_1,x_2,x_3,x_4)\in Z^{\prime\prime}$, we can write $x_4 = f(x_1,x_2,x_3)$, and 
\begin{equation}\label{boundOnSecondDerivative}
|\partial_{x_3}^2 f(x_1,x_2,x_3)|\gtrsim_D \kappa.
\end{equation}

Let $\pi(x_1,x_2,x_3,x_4)=(x_1,x_2,x_3)$, so $Z^{\prime\prime}$ is the graph of $f$ above $\pi(Z^{\prime\prime})$. By \eqref{boundOnSecondDerivative}, we have that if $I \subset \pi(Z^{\prime\prime})$ is a line segment pointing in the direction $(0,0,1)$ and if $|f(x_1,x_2,x_3)|\leq 2b$ for all $(x_1,x_2,x_3)\in I$, then we must have $|I|\lesssim_D (b/\kappa)^{1/2}$. Next, let $L$ be a line in $\RR^3$ pointing in the direction $(0,0,1)$. Then 
$$
\Big|\{(x_1,x_2,x_3)\in L\colon |f(x_1,x_2,x_3)|\leq 2b\}\Big|\lesssim_D (b/\kappa)^{1/2},
$$
where $|\cdot|$ denotes one dimensional Lebesgue measure. Thus by Fubini's theorem,
$$
\Big|\{(x_1,x_2,x_3)\in \pi(Z^{\prime\prime})\colon |f(x_1,x_2,x_3)|\leq 2b\}\Big|\lesssim_D (b/\kappa)^{1/2},
$$
where $|\cdot|$ denotes three dimensional Lebesgue measure. Thus
$$
|\{(x_1,x_2,x_3,x_4)\in Z^{\prime\prime} \colon |x_4|\leq 2b\}\Big|\lesssim_D (b/\kappa)^{1/2},
$$
where again $|\cdot|$ denotes three dimensional Lebesgue measure. We conclude that 
\begin{equation}\label{volumeIntersectionWithHyperplane}
|N_{a}(Z^{\prime\prime})\cap N_b(H)|\lesssim_{D,E} a \big|Z^{\prime\prime}\cap N_{a+b}(H)\big|  \lesssim_{D,E} a(b/\kappa)^{1/2}.
\end{equation}
Since \eqref{volumeIntersectionWithHyperplane} holds for each of the $O_{D,E}(1)$ connected sets $Z^{\prime\prime}$, we obtain \eqref{volumeBoundCurvedSurfaceHyperplane}.
\end{proof}

\subsection{Flat varieties are contained near a hyperplane}
\begin{prop}\label{FlatVarietiesNearAPlaneProp}
Let $0<\delta<\kappa$ and let $c>0$. Let $P\in\RR[x_1,\ldots,x_4]$ be a polynomial of degree at most $D$. Define 
$$
Z=\{x\in Z(P)\cap B(0,1)\colon \nabla P(x)\neq 0,\ \Vert II(x)\Vert_\infty\leq \kappa\}.
$$ 
Then there is a set of $O_{D}(c^{-O(1)})$ rectangular prisms of dimensions $2\times 2\times 2\times\kappa$ so that every line in $\Sigma_{\delta,c}(Z)$ is contained on one of these prisms.  
\end{prop}
\begin{proof}
Let $Z_1,\ldots,Z_p,\ p = O_{D}(1)$ be the connected components of $Z$ (without loss of generality, we can assume that each of these components is a smooth 3-dimensional manifold.) Each line $\ell\in \Sigma_{\delta,c}(Z)$ satisfies $|\ell\cap N_{\delta}(Z_j)|\geq c/p$ for some index $j$. By Lemma \ref{secondFundFormHyperplaneLem}, each connected component $Z_j$ can be contained in the $\kappa^\prime = O_{D}(\kappa)$--neighborhood of a hyperplane; call this hyperplane $H_j$. 

Finally, for each index $j$, we can select $O_{D}(c^{-O(1)})$ rectangular prisms of dimensions $2\times 2\times 2\times\kappa$ so that every line $\ell$ satisfying $|\ell\cap N_{\kappa^\prime}(H_j)|\geq c/p$ must must be covered by one of these prisms. 
\end{proof}

\section{Multilinearity and quantitative broadness}
In this section we will explore the notion of ``broadness,'' which was introduced by Guth in \cite{G} to study the restriction problem. Throughout this section, we will often refer to the following ``standard setup.''
\begin{defn}[Standard setup]\label{standardSetupDefn}
Let $d$ be a positive integer. Let $Z\subset\RR^d$ and let $\Phi\subset Z\times S^{d-1}$ be semi-algebraic sets of complexity at most $E$. Let $\pi_Z \colon \Phi\to Z$ and $\pi_S \colon \Phi \to S^{d-1}$ be the projection of $\Phi$ to $Z$ and $S^{d-1}$, respectively. For each $z\in Z$, define $\Phi(z) = \pi_{Z}^{-1}(z).$
\end{defn}
\subsection{$(m,A)$-broadness}\label{QuantBroadnessSection}
\begin{defn}\label{defnBroadness}
Let $d,\Phi,$ and $Z$ be defined as in the standard setup from Definition \ref{standardSetupDefn}. For each positive integer $0\leq m\leq d-1$, let $\mathcal{S}_{m}$ be the set of $m$ dimensional unit spheres contained in $S^{d-1}$ (recall that a zero dimensional unit sphere in $S^{d-1}$ is just a pair of antipodal points). We will identify $\mathcal{S}_{m}$ with a semi-algebraic subset of $\RR^N$ for some $N=O_{d}(1)$. Let $A\geq 1$ be an integer and let $u\geq 0$. Define
\begin{equation*}
\begin{split}
(m,A)\operatorname{-Narrow}_u(\Phi) = \big\{z\in Z\colon \exists\ S_1,\ldots,S_A \in&\ \mathcal{S}_{m}\ \operatorname{s.t.}\ \forall\ v\in \pi_S(\Phi(z)),\\
& \exists\ i\in\ \{1,\ldots,A\}\ \textrm{s.t.}\ \angle(v, S_i\}\leq u \big\}.
\end{split}
\end{equation*}
In words, $z\in (m,A)\operatorname{-Narrow}_u(\Phi)$ if and only if there is a set of $A$ $m$-dimensional unit spheres $S_1,\ldots,S_A$ so that every vector $v\in \pi_S(\Phi(z))$ makes an angle at most $u$ with one of these spheres. Observe that if $m\leq m^\prime,\ A\leq A^\prime,$ and $u\leq u^\prime$, then $(m,A)\operatorname{-Narrow}_u(\Phi)\subset (m^\prime,A^\prime)\operatorname{-Narrow}_{u^\prime}(\Phi)$.

Define 
$$
(m,A)\operatorname{-Broad}_u(\Phi) = Z\ \backslash\ (m,A)\operatorname{-Narrow}_u(\Phi).
$$
If $(m,A)\operatorname{-Narrow}_u(\Phi)=Z$, we say that $\Phi$ is $(m,A)$-narrow at width $u$. If $(m,A)\operatorname{-Broad}_u(\Phi) = Z$, we say that $\Phi$ is $(m,A)$-broad at width $u$.
\end{defn}
The sets $(m,A)\operatorname{-Broad}_u(\Phi)$ and $(m,A)\operatorname{-Narrow}_u(\Phi)$ have complexity $O_{d,E}(1)$. In practice we will have $d=4$ so the complexity is $O_E(1)$. 
\begin{lem}\label{coveringNumberOfDistinguishedSpheres}
Let $d,\Phi,$ and $Z$ be defined as in the standard setup from Definition \ref{standardSetupDefn}. Let $u\leq s$. Let $m\geq 1$. Suppose that $\Phi$ is $(m,1)$-narrow at width $u$ and $(m-1,1)$-broad at width $s$. Define
$$
A = \{(z, S)\in Z\times \mathcal{S}_{m}\colon \angle(v, S)\leq u\ \forall\ v\in \pi_S(\Phi(z))\}.
$$

If we identify $Z\times \mathcal{S}_{m}$ with a subset of $\RR^d\times \RR^N,\ N =O_d(1)$, then 
\begin{equation}
\mathcal{E}_u(A) \lesssim_{d,E} u^{-\operatorname{dim}(Z)}\ s^{-\operatorname{codim}(Z)},
\end{equation}
where $\operatorname{codim}(Z) = d + N - \operatorname{dim}(Z) = O_{d}(1)$.
\end{lem}
\begin{proof}
Since the constants $d$ and $E$ are fixed, all implicit constants may depend on these quantities. First, since $\Phi$ is $(m,1)$-narrow at width $u$, i.e.~$(m,1)\operatorname{-Narrow}_u(\Phi)= Z$, we have that $\pi\colon A\to Z$ is onto. Apply Lemma \ref{selPtFiberProp} to select a semi-algebraic set $A^\prime\subset A$ of complexity $O(1)$ so that $\pi\colon A^\prime\to Z$ is a bijection. In particular, we have $\dim(A^\prime)=\dim(Z).$

We claim that if $(z, S),\ (z,S^\prime)\in A$, then $\operatorname{dist}(S,S^\prime)\lesssim u/s,$ where $\dist(S,S^\prime)$ denotes the Euclidean distance in $\RR^N$ between the points in $\RR^N$ identified with $S$ and $S^\prime$. Indeed, note that 
$$
\pi_S(\Phi(z))\subset N_{u}(S)\cap N_{u}(S^\prime).
$$
Since $z\in (m-1,1)\operatorname{-Broad}_s(\Phi)$ we have that $\pi_S(\Phi(z))\subset N_{u}(S)\cap N_{u}(S^\prime)$ cannot be contained in the $s$--neighborhood of a $(m-1)$-dimensional unit sphere, and thus $N_u(S)\cap N_u(S^\prime)$ cannot be contained in the $s$--neighborhood of a $(m-1)$-dimensional unit sphere. This implies $\operatorname{dist}(S,S^\prime)\lesssim u/s$.

By Lemma \ref{coveringNumberSemiAlg}, 
\begin{equation*}
\begin{split}
\mathcal{E}_u(A) & \lesssim \mathcal{E}_u \big(N_{u/s}(A^\prime)\big)\\
&\lesssim u^{\operatorname{dim}(A^\prime)}\ s^{-\operatorname{codim}(A^\prime)}\\
&=u^{\operatorname{dim}(Z)}\ s^{-\operatorname{codim}(Z)}.\qedhere
\end{split}
\end{equation*}
\end{proof}
In practice, we will use Lemma \ref{coveringNumberOfDistinguishedSpheres} in the special case $Z\subset\RR^4,\ \Phi\subset Z\times S^3\subset\RR^8$, and $m=2$. The lemma will help us analyze the situation where a hypersurface $Z(P)\subset\RR^4$ is ruled by planes. 
\subsection{Strong broadness}\label{strongBroadnessSection}
If $\mathcal{V}\subset S^{d-1}$ is semi-algebraic of complexity at most $E$, then by Lemma \ref{MTTheorem} there is a number $K_{d,E}$ so that $\mathcal{V}$ is a union of at most $K_{d,E}$ connected components. This means that for each $s>0$, either $\mathcal{V}$ can be contained in the $s$--neighborhood of a union of $K_{d,E}$ vectors, or $\mathcal{V}$ contains a connected component of diameter at least $s$.
\begin{defn}\label{defnSBroad}
Let $d,\Phi,$ and $Z$ be defined as in the standard setup from Definition \ref{standardSetupDefn}. Let $E_1=O_{d,E}$ be an integer so that $\pi_S(\Phi(z))$ has complexity at most $E_1$ for each $z\in Z$. 

Define
$$
1\operatorname{-SBroad}_s(\Phi) = (1,K_{d,E_1})\operatorname{-Broad}_s(\Phi).
$$
The ``S'' stands for ``Strong.'' If $1\operatorname{-SBroad}_s(\Phi) =Z$, we say that $\Phi$ is strongly $1$-broad at width $s$.
\end{defn}
Note that if $z\in 1\operatorname{-SBroad}_s(\Phi)$, then $\pi_S(\Phi(z))$ contains a connected component of diameter $\geq s$. Conversely, if $\pi_S(\Phi(z))$ contains a connected component of diameter at least $K_{d,E_1}s$, then $z\in 1\operatorname{-SBroad}_s(\Phi)$. It would be preferable to just directly define $1\operatorname{-SBroad}_s(\Phi) $ to be the set of points $z\in Z$ so that $\pi_S(\Phi(z))$ contains a connected component of diameter at least $s$, but it is not clear whether $1\operatorname{-SBroad}_s(\Phi)$ would be semi-algebraic with this definition. Under Definition \ref{defnSBroad}, $1\operatorname{-SBroad}_s(\Phi)\subset Z$ is semi-algebraic of complexity $O_{d,E}(1)$. 
\begin{rem}
It would be straightforward to define a notion of strong $m$-broadness for each $m\geq 1$, but this definition is not particularly useful if $m> 1$. One of the key properties of strong 1-broadness is that if $z\in 1\operatorname{-SBroad}_s(\Phi)$, then 
$$
\mathcal{E}_s\big(\pi_S(\Phi(z))\big)\gtrsim_{d,E}s^{-1}.
$$
Unfortunately, the analogous statement for strong $m$-broadness (with $s^{-1}$ replaced by $s^{-m}$) need not be true. 
\end{rem}
\begin{lem}\label{entropyOfNarrowGamma}
Let $d,\Phi,$ and $Z$ be defined as in the standard setup from Definition \ref{standardSetupDefn} and let $s>0$. Suppose that $1\operatorname{-SBroad}_{s}(\Phi)=\emptyset$. Then
$$
\mathcal{E}_s(\Phi)\lesssim_{d,E} s^{-\operatorname{dim}(Z)}.
$$
\end{lem}
\begin{proof}
Since $d$ and $E$ are fixed, all implicit constants may depend on these quantities. For each $k=1,\ldots, O(1)$, define
$$
A_k = \{(z,v_1,\ldots,v_k) \colon z\in Z, v_1,\ldots,v_k \in \pi_S(\Phi(z)),\ \angle(v_i,v_j)\geq s\ \textrm{if}\ i\neq j\}.
$$
Let $\pi_k\colon A_k\to Z$ be the projection to $Z$, and let $Z_k = \pi_k(A_k)$. Since $1\operatorname{-SBroad}_{s}(\Phi)=\emptyset$, we have that
\begin{equation}\label{everythingCovered}
\bigcup_{k=1}^{O(1)} Z_k = \pi_Z(\Phi).
\end{equation}

Apply Lemma \ref{selPtFiberProp} to each map $\pi_k\colon A_k\to Z_k$ to obtain sets $A_k^\prime \subset A_k$ so that the projection map $\pi_k\colon A_k^\prime \to Z_k$ is a bijection. Define $Z_k^\prime = Z_k \backslash \bigcup_{j>k} Z_j$. Note that each set $Z_k^\prime$ has dimension $\leq \dim(Z)$ and is semi-algebraic of complexity $O(1)$.

Define 
$$
\Phi_k = \{(z,v)\in \Phi\colon z\in Z_k^\prime\}.
$$ 
Then by \eqref{everythingCovered}, we have
\begin{equation}\label{everythingCoveredGamma}
\Phi = \bigsqcup_{k=1}^{O(1)}\Phi_k.
\end{equation}

Note that if $z\in Z_k^\prime$ and if $(z, v_1,\ldots,v_k)=\pi_k^{-1}(z)$ (here $\pi_k\colon A_k^\prime\to Z_k^\prime$ is a bijection, so it has a well-defined inverse), then every vector in $\pi_S(\Phi(z))$ must be close to one of the vectors $v_1,\ldots,v_k$; more precisely,
\begin{equation}\label{vSigmaClose}
\pi_S(\Phi(z)) \subset N_{s}(\{v_1,\ldots, v_k\}).
\end{equation}
Indeed, if \eqref{vSigmaClose} did not hold, then there exists $v_{k+1}\in \pi_S(\Phi(z))\ \backslash\ N_{s}(\{v_1,\ldots,v_k\}).$ But  then $(z,v_1,\ldots,v_k,v_{k+1})\in A_{k+1}$, so $z\in Z_{k+1}$, which contradicts the assumption that $z\in Z_k^\prime$. 

For each index $k$ and each $j=1,\ldots,k$, define the projections $\pi_{k,j}\colon A_k^\prime\to Z_k^\prime\times S^{d-1}$ by $(z,v_1,\ldots,v_k)\mapsto (z,v_j)$. Define $W_k = \bigcup_{j=1}^k \pi_{k,j}(A_k^\prime)$. Thus $W_k\subset Z_k^\prime\times S^{d-1}$, and for each $z\in Z_k^\prime$, we have that $W_k\cap (\{z\}\times S^{d-1})$ is the set $\{(z, v_1),\ldots, (z,v_k)\}$, where $(z, v_1,\ldots,v_k)=\pi_k^{-1}(z)$. 

Observe that $W_k$ is a semi-algebraic set of dimension $\leq \dim(Z)$ and complexity $O(1)$, and thus by Lemma \ref{coveringNumberSemiAlg},
\begin{equation}\label{complexityWk}
\mathcal{E}_s(W_k)\lesssim s^{-\dim(Z)}.
\end{equation}

On the other hand, by \eqref{vSigmaClose} we have that
\begin{equation}\label{coveringNumberOfGammak}
\mathcal{E}_s(\Phi_k)\lesssim \mathcal{E}_s(W_k). 
\end{equation}
The lemma now follows from \eqref{everythingCoveredGamma}, \eqref{vSigmaClose}, \eqref{complexityWk}, and \eqref{coveringNumberOfGammak}.
\end{proof}
\subsection{Bundles of lines and the standard setup}
In this section we will relate the objects $Z,\Sigma$ and $\Gamma$ from Section \ref{bundlesOfLinesSection} with the standard setup from Definition \ref{standardSetupDefn}. 

\begin{defn}\label{defnAssociated}
Let $\delta>0$ and let $P\in\RR[x_1,\ldots,x_4]$ be a polynomial of degree at most $D$. Let $Z\subset  Z(P)\cap B(0,1)$ and $\Gamma\subset\Gamma(N_{\delta}(Z),\lines)$ be semi-algebraic sets of complexity at most $E$. For each $x\in N_{\delta}(Z)$, define $f_Z(x)$ to be the point $z\in Z$ that minimizes $\operatorname{dist}(x,z)$. If more than one such point exists, select the one that is minimal under the lexicographic order (any semi-algebraic total order would be equally good). Then the set $\{(x, z)\in N_{\delta}(Z)\times Z\colon z=f_Z(x)\}$ is semi-algebraic of complexity $O_{E,D}(1)$. Define the set $\Phi \subset Z \times S^3$ to be the set of ordered pairs
$$
\Phi = \{(f_Z(x), v(\ell)) \colon (x,\ell)\in\Gamma\}.
$$
Then $\Phi\subset Z\times S^3$ is semi-algebraic of complexity $O_{D,E}(1)$. We will say that set $\Phi$ is associated to $\Gamma$ (the set $Z$ and the parameter $\delta$ will be obvious from context). 
\end{defn}

By construction, the map $\Gamma\to \Phi,\ (x,\ell)\mapsto (f_Z(x), v(\ell))$ is onto. If $\Phi^\prime\subset \Phi$, we will define $\Gamma^\prime$ to be the pre-image of $\Phi^\prime$ under this map. 

\section{Narrow varieties}\label{stronglyBroadReductionCase}
In this section, we will consider the region where $Z$ either fails to be robustly 1-broad, or fails to be $(2,2)$-broad and has large second fundamental form. We will show that the lines having large intersection with this region must be covered by a union of two and three dimensional prisms (i.e. prisms with either two or three ``long'' directions). These sets of lines will be the sets $\Sigma_1$ and $\Sigma_2$ from the statement of Theorem \ref{discretizedSeveri}. The main results of this section are Proposition \ref{narrowVarietiesProp}, which describes what happens when $Z$ fails to be robustly $1$-broad, and Proposition \ref{22NarrowCovered}, which describes what happens when $Z$ has large second fundamental form and fails to be $(2,2)$-broad. 
\subsection{$1$-Narrow varieties are ruled by lines}

\begin{prop}\label{narrowVarietiesProp}
Let $0<\delta<s$ and $c>\delta$. Let $P\in\RR[x_1,\ldots,x_4]$ be a polynomial of degree at most $D$; let $Z= Z(P)\cap B(0,1)$; let $\Sigma\subset \Sigma_{\delta,c}(Z)$ be a semi-algebraic set of complexity at most $E$. Let $\Gamma\subset\Gamma(N_{\delta}(Z),\Sigma)$ be a semi-algebraic set of complexity at most $E$. Let $\Phi\subset Z\times S^3$ be associated to $\Gamma$, in the sense of Definition \ref{defnAssociated}. Suppose that
\begin{align}
&1\operatorname{-SBroad}_{s}(\Phi)=\emptyset,\label{noBroadS}\\
&|\Gamma(\ell)|\geq c\quad \textrm{for every}\ \ell\in\Sigma.\label{GammaEllLargeSize1} 
\end{align}

Then there is a set of $O_{D,E}(c^{-1}s^{-2})$ rectangular prisms of dimensions $2\times s\times s\times s$, so that every line from $\Sigma$ is covered by one of the prisms.
\end{prop}
\begin{proof}

By Lemma \ref{entropyOfNarrowGamma}, 
\begin{equation}\label{XHasSmallCoveringNumber}
\mathcal{E}_s(\Phi)\lesssim_{D,E} s^{-3}.
\end{equation}

Let $\mathcal{R}_{\operatorname{max}}$ be a maximal set of essentially disjoint rectangular prisms of dimensions $2\times s/4\times s/4\times s/4$ that intersect $B(0,1)$. For each $R\in \mathcal{R}_{\operatorname{max}}$, define $v(R)$ to be the direction of the long axis of $R$, and define
$$
R^* = \{(x,v)\in \RR^4\times S^3\colon x\in R,\ \angle(v, v(R))\leq s/4\}.
$$

Observe that every line intersecting $B(0,1)$ is covered by some rectangular prism from $\mathcal{R}_{\operatorname{max}}$, and that for each $C\geq 1$, the $C$--fold dilates of the sets $\{CR^*\}_{R\in \mathcal{R}_{\operatorname{max}}}$ are $O_C(1)$--overlapping. 
 
Let $(x,\ell)\in\Gamma$ and let $(f_Z(x), v(\ell))$ be the corresponding element of $\Phi$. Note that if $\ell\times \{v(\ell)\}\cap R^*\neq\emptyset$ then $\ell\cap R\neq\emptyset$ and $\angle(v(\ell), v(R))\leq s/4$, and thus $\ell$ is covered by the 4-fold dilate of $R$. Furthermore, if this happens then $\Phi\cap R^*\neq\emptyset$. Thus to prove the lemma, it suffices to show that 

\begin{equation}\label{boundOnNumberofT}
|\{R\in \mathcal{R}_{\operatorname{max}} \colon R^* \cap \Phi\neq\emptyset\}|\lesssim_{D,E} c^{-1}s^{-2}. 
\end{equation}
Let $R\in \mathcal{R}_{\operatorname{max}}$ and suppose $R^*\cap \Phi \neq\emptyset$. We will show that
\begin{equation}\label{largeIntersectionWithTStar}
\mathcal{E}_s(\Phi\cap 4R^*)\geq c/s.
\end{equation}
To see this, let $(z,v)\in R^* \cap \Phi$. Then there exists a point $(x,\ell)\in\Gamma$ with $\operatorname{dist}(x,z)<\delta$ and $v(\ell)=v$. Thus
$$
\angle( v(\ell), v(R))\leq s/4.
$$
Thus $\ell\cap B(0,1)\subset \ell\cap 4S $. Since $\ell\in\Sigma,$ we have $\mathcal{E}_s(\Gamma(\ell))\geq  c/s$ and thus there exist $\geq c/s$ $s$-separated points $x^\prime\in\Gamma(\ell)$ with $(x^\prime,v(\ell))\in S^*$. Of course for each of these points $x^\prime\in\Gamma(\ell)$ there exists a point $z^\prime\in B(x^\prime,\delta)$ with $(z^\prime,v(\ell))\in \Phi\cap R^*$, which establishes \eqref{largeIntersectionWithTStar}. Since the sets $\{4R^*\}_{\mathcal{R}_{\operatorname{max}}}$ are $O(1)$ overlapping, by combining \eqref{XHasSmallCoveringNumber} and \eqref{largeIntersectionWithTStar} we have \eqref{boundOnNumberofT}.
\end{proof}

\subsection{$2$-narrow varieties are ruled by planes}\label{nonStrongThreeBroadReductionSec}
\begin{lem}\label{3NarrowCovered}
Let $\delta,u,s,\kappa,c,D,E$ be parameters with $0<\delta<u<s<c$ and $\delta<\kappa$. Then there exists a number $w\gtrsim_{D,E}(us\kappa c)^{O(1)}$ so that the following holds.

Let $P\in\RR[x_1,\ldots,x_4]$ be a polynomial of degree at most $D$ and let 
\begin{equation}\label{defnOfZ}
Z \subset \{z\in Z(P)\cap B(0,1)\colon 1\leq |\nabla P(z)|\leq 2,\ \Vert II(z)\Vert_\infty \geq \kappa\}.
\end{equation}
Let $\Sigma\subset\Sigma_{\delta,c}(Z)$, and let $\Gamma\subset\Gamma(N_{\delta}(Z),\Sigma)$. Suppose that $Z,\Sigma,$ and $\Gamma$ are semi-algebraic sets of complexity at most $E$. Let $\Phi$ be associated to $\Gamma$, in the sense of Definition \ref{defnAssociated}. Suppose that
\begin{align}
&Z\subset (2,1)\operatorname{-Narrow}_{w}(\Phi),\label{Zsubset21narrowW}\\
&|\Gamma(\ell)|\geq c\quad\textrm{for all}\ \ell\in\Sigma. \label{GammaEllLargeSize}
\end{align}

Then we can write $\Sigma = \Sigma^\prime\cup\Sigma^{\prime\prime}$, where the lines in $\Sigma^\prime$ can be covered by $O_{D,E}(c^{-O(1)}s^{-2})$ rectangular prisms of dimensions $2\times s\times s\times s$, and the lines in $\Sigma^{\prime\prime}$ can be covered by $O_{D,E}((cs)^{-O(1)}u^{-1})$ rectangular prisms of dimensions $2\times 2\times u\times u$.
\end{lem}

\begin{proof}
Since $D$ and $E$ are fixed, whenever we write $A\lesssim B$, the implicit constant may depend on these quantities. Let  
$$
\Phi_0 = \{ (z,v)\in \Phi\colon z\in 1\operatorname{-SBroad}_{s}(\Phi) \},
$$ 
and let $\Phi_0^\prime=\Phi\backslash \Phi_0$. Define $\Gamma_0$ to be the pre-image of $\Phi_0$ under the map $\Gamma\to \Phi$, and define $\Gamma_0^\prime=\Gamma\backslash\Gamma_0$. Note that $\Gamma_0^\prime$ is the pre-image of $\Phi_0^\prime$. $\Gamma_0^\prime$ is a set of complexity $E_1 = O_{D,E}(1)$; in particular, the constant $E_1$ can be chosen to be independent of $s$ and $\delta$. 

Let 
\begin{equation*}
\begin{split}
\Sigma_0 &= \{\ell\in\Sigma\colon|\Gamma_0(x)| > c/2\},\\
\Sigma_0^\prime & = \{\ell\in\Sigma\colon|\Gamma_0^\prime(x)| > c/2\}.
\end{split}
\end{equation*}

We will put the lines in $\Sigma_0^\prime$ into $\Sigma^\prime$; by Proposition \ref{narrowVarietiesProp}, these lines can be covered by $O(s^{-2}c^{-O(1)})$ rectangular prisms of dimensions $2\times s\times s\times s$.

Define 
$$
Z_1=\{z\in Z\colon \Phi_0(z)\neq\emptyset \}\subset 1\operatorname{-SBroad}_{s}(\Phi).
$$
Let $w>(us\kappa c)^{O(1)}$ be a number that will be determined below. Define
\begin{equation}\label{defnA1}
A_1=\{(z,S)\in Z_1\times \mathcal{S}_1\colon \angle(v,S)\leq w\ \forall\ v\in \pi_S(\Phi(z))\}.
\end{equation}
Since $Z_1\subset 1\operatorname{-SBroad}_{s}(\Phi)$, we have that if $(z,S),(z,S^\prime)\in A_1$, then 
\begin{equation}\label{smallAnglePiPiPrime}
\angle(S,S^\prime)\leq 2w/s.
\end{equation}
This is because there exists two vectors $v,v^\prime\in \pi_S(\Phi(z))$ with $\angle(v,v^\prime)\geq s$, and these vectors satisfy $\angle(v,S)\leq w; \angle(v,S^\prime)\leq w;\ \angle(v^\prime,S)\leq w;$ and $\angle(v^\prime,S^\prime)\leq w$. Observe that $A_1$ obeys the hypotheses of Lemma \ref{coveringNumberOfDistinguishedSpheres}.

Define $\pi\colon A_1\to Z_1$ to be the projection $(z,S)\mapsto z$. Then by \eqref{Zsubset21narrowW} $\pi(A_1)=Z_1$. Use Lemma \ref{selPtFiberProp} to select a set $A_1^\prime\subset A_1$ so that $\pi\colon A_1^\prime\to Z_1$ is a bijection. By \eqref{smallAnglePiPiPrime}, we have that 
\begin{equation}\label{X1NearX1Prime}
A_1 \subset N_{2w/s}(A_1^\prime).
\end{equation}

For each $z\in Z_1$, define $S(z)$ to be the (unique) great circle $S$ so that $(z,S)\in A_1^\prime$. The function $S(z)$ is semi-algebraic of complexity $O(1)$. If $(x,\ell)\in \Gamma_0$, then $f_Z(x)\in Z_1$. Thus for each $\ell\in\Sigma$, the set 
$$
S\circ f_Z\circ\Gamma_0(\ell) =\{S(f_Z(x))\colon x\in \Gamma_0(\ell)\}\subset\mathcal{S}_1
$$ 
is a union of $O(1)$ connected components in $\mathcal{S}_1$ (recall that $\mathcal{S}_1$ is the parameter space of one-dimensional great circles in $S^3$). 

For each $S\in\mathcal{S}_1$, define $\operatorname{span}(S)$ to be the two-dimensional vector space in $\RR^4$ that contains the great circle $S$, i.e.~$\operatorname{span}(S)=\{rv\colon r\in\RR,\ v\in S \}$. For each $z\in Z_1$, Define $\Pi(z) = z + \operatorname{Span}(S(z))$; this is an affine 2-plane containing $z$.

Since $S\circ f_Z\circ\Gamma_0(\ell)$ is a union of $O(1)$ connected components, heuristically, this means that either (A): $\{\Pi\circ f_Z(z)\colon z\in\Gamma_0(\ell)\}$ can be covered by the thickened neighborhoods of a small number of planes containing $\ell$, or (B): the union of the planes in $\{\Pi\circ f_Z(z)\colon z\in\Gamma_0(\ell)\}$ fill out a large fraction of $N_{\delta}(Z_1)$. We will make this heuristic precise in the arguments below.

let $h\gtrsim c$ be a constant that will be determined below. Define
\begin{equation*}
\begin{split}
Y = \big\{(\ell,S, x)\in\Sigma_0&\times \mathcal{S}_1 \times \RR^4\colon x\in\Gamma_0(\ell),\\
 &\forall\ x^\prime\in \ell\cap B(x,h),\ \textrm{we have}\ x^\prime \in\Gamma_0(\ell)\ \textrm{and}\ \angle(S\circ f_Z(x^\prime),S)<u\big\}.
\end{split}
\end{equation*}

Let $\pi_{\lines}(\ell,S,x)= \ell$ and define $\Sigma_1 = \pi_{\lines}(Y)$, $\Sigma_2 = \Sigma_0\backslash\Sigma_1$. In words, $\Sigma_1$ is the set of lines $\ell\in\Sigma_0$ so that there exists a line segment $I\subset\Gamma_0(\ell)$ of length $2h$ with the property that the great circle $S\circ f_Z(x^\prime)$ does not change much as $x^\prime$ moves along $I$. 
\begin{rem}
Heuristically, if the variety $Z(P)$ was ruled by planes, and if every line was contained in one of these planes, then $\Sigma_1=\Sigma_0$. We will show that if this is the case, then $\Sigma_1$ can be partitioned into disjoint pieces that do not interact with each other (geometrically, if $Z(P)$ is ruled by planes and if every line lies in one of these planes, then we can write $Z(P)$ as a disjoint union of planes and consider each of these planes individually). 
\end{rem}

On the other hand, $\Sigma_2$ is the set of lines $\ell\in\Sigma_0$ so that for every interval $I\subset\Gamma_0(\ell)$ of length $2h$, $S\circ f_Z(I)$ has diameter $\geq u$. Since $S\circ f_Z \circ \Gamma_0(\ell)$ is a union of $O(1)$ connected components, if $h\gtrsim c$ is selected sufficiently small, then there must exist an interval $I\subset\Gamma_0(\ell)$ of length $2h$ so that $S\circ f_Z(I)$ is connected. 
\begin{rem}
Heuristically, if the variety $Z(P)$ is a small perturbation of a hyperplane, then it could be the case that $\Sigma_2=\Sigma_0$. We will show that if this is the case, then most of $Z(P)$ can be contained in a thin neighborhood of a hyperplane, and this will contradict the assumption that $\Vert II(z)\Vert_\infty$ is large on $Z$.
\end{rem}

\subsubsection*{Understanding lines in $\Sigma_1$}
Apply Lemma \ref{selPtFiberProp} to the surjection $\pi_\lines\colon Y\to\Sigma_1$ to obtain a set $Y^\prime\subset Y$ so that $\pi_\lines\colon Y^\prime\to\Sigma_1$ is a bijection. Define $S(\ell)$ to be the circle in $\mathcal{S}_1$ containing the vector $v(\ell)$ that makes the smallest angle with the circle $\pi_{\mathcal{S}_1}\circ\pi^{-1}_{\lines}(\ell)$. Observe that if $\ell\in\Sigma_1$ and if $(\ell,S,x)=\pi_{\lines}^{-1}(\ell)$, then $(x, v(\ell))\in \Gamma_0$ and thus $\angle(v(\ell),S)<w\leq u$, so $\angle(S, S(\ell))< u$.

 %With this definition we have that for all $y\in\Gamma_0(\ell)$, 

For each $\ell\in\Sigma_1$, define $\Pi(\ell)=\ell+\operatorname{span}(S(\ell))$; this is an affine plane containing $\ell$ that points in the directions spanned by $S(\ell)$.

Define 
$$
\Gamma_1=\{(x,\ell)\in\Gamma_0\colon\ \angle(S(\ell), S(x))<2u\}.
$$
Then for each $\ell\in\Sigma_1$, we have $|\Gamma_1(\ell)|\geq 2h$. This is because $\ell\in\Sigma_1$ implies that there exists a point $(\ell,S,x)=\pi_{\lines}^{-1}(x)\in Y,$ and thus there exists an interval $I\subset \Gamma_0(\ell)$ of length $2h$ containing $x$ so that for every point $x^\prime\in I$, we have 

\begin{equation*}
\begin{split}
\angle(S(\ell),S(x^\prime))&\leq\angle(S(\ell), S)+\angle(S(x^\prime),S)\\
&< u+u\\
&=2u.
\end{split}
\end{equation*}

Let $\Phi_1$ be the set associated to $\Gamma_1$, in the sense of Definition \ref{defnAssociated}. Define
$$
\Phi_2 = \{(z,v)\in \Phi_1\colon z \in 1\operatorname{-SBroad}_{s}(\Phi_1) \}.
$$
Define $\Gamma_2\subset\Gamma_1$ to be the pre-image of $\Phi_2$ and define $\Gamma_{2}^\prime=\Gamma_{1}\backslash\Gamma_{2}$ (this is the pre-image of $\Phi_2^\prime = \Phi_1\backslash \Phi_2$). Define 
\begin{equation*}
\begin{split}
\Sigma_{1}^\prime&=\{\ell\in\Sigma_1\colon |\Gamma_2(\ell)|\geq h\},\\
\Sigma_{1}^{\prime\prime}&=\{\ell\in\Sigma_1\colon |\Gamma_2^\prime(\ell)|\geq h\}.
\end{split}
\end{equation*}
$\Gamma_2^\prime$ and $\Sigma_{1}^{\prime\prime}$ are semi-algebraic sets of complexity $O(1)$; the lines in $\Sigma_1^{\prime\prime}$ will be added to $\Sigma^\prime$; by Proposition \ref{narrowVarietiesProp}, these lines can be covered by $O(s^{-2}h^{-O(1)})=O(s^{-2}c^{-O(1)})$ rectangular prisms of dimensions $2\times s\times s\times s$.

Define $\Sigma^{\prime\prime} = \Sigma_{1}^\prime$ (recall that $\Sigma^\prime$ and $\Sigma^{\prime\prime}$ are the output of the lemma). We will show that these lines can be covered by $O(u^{-1} (sc)^{-O(1)})$ rectangular prisms of dimensions $2\times 2\times u\times u$. 

 Define $Z_2 = \{z\in Z_1\colon \Phi_2(z) \neq\emptyset\}$. Define 
$$
A_2 = \{(z, S)\in A_1\colon z\in Z_2\}.
$$
This is the set of pairs $(z,S)$ so that $z\in Z_2$ and all of the vectors from $\pi_S(\Phi(z))$ (and thus all of the vectors from $\pi_S(\Phi_2(z)) $) are contained in the $2u$--neighborhood of $S$. In practice, it will be more convenient to work with the set 
$$
\tilde A_2 = \{(z, z + \operatorname{span}(S)\colon (z,S)\in A_2\};
$$
this is the set of pairs $(z,\Pi)$, where $\Pi$ is an affine 2-plane containing $z$ that is spanned by the vectors in $S(z)$. 

Let $\mathcal{R}_{\operatorname{max}}$ be a maximal set of essentially disjoint rectangular prisms of dimensions $2\times 2\times u\times u$ that intersect $B(0,1)$. For each $R\in \mathcal{R}_{\operatorname{max}}$, define $\Pi(R)$ to be the affine plane concentric with the long axes of $R$, and define 
$$
R^*= \{(x,\Pi)\in R \times\operatorname{Grass}(4;2)\colon  \operatorname{dist}(\Pi,\Pi(R))\leq u \}.
$$
The expression $\operatorname{dist}(\Pi,\Pi(R))$ should be interpreted as follows: Select a semi-algebraic embedding of the affine Grassmannian $\operatorname{Grass}(4;2)$ into $\RR^N$; then $\operatorname{dist}(\Pi,\Pi(R))$ is the Euclidean distance between the points in $\RR^N$ corresponding to the images of $\Pi$ and $\Pi(R)$.

Observe that for each $C$, the $C$-fold dilates $\{CR^*\}_{R\in\mathcal{R}_{\operatorname{max}} }$ are $O_C(1)$--fold overlapping. For each $z\in Z_2$, there is a prism $R\in \mathcal{R}_{\operatorname{max}}$ so that $(z, \Pi(z))\in R^*$. This value of $R$ satisfies
\begin{equation}\label{PiStarMeetsA2}
R^*\cap \tilde A_2\neq\emptyset.
\end{equation}

For this $R$, we also have 
\begin{equation}\label{ellInN2wPi}
\Pi(z) \cap B(0,1)\subset 4R,
\end{equation}
where $4R$ is the four-fold dilate or $R$. Note as well that if $(x,\ell)\in \Gamma_1$ with $x\in f_Z^{-1}(z)$, then $\ell$ is covered by $4R$. This is because $\ell\cap 2R\neq\emptyset$ and 
\begin{equation*}
\begin{split}
\angle(\ell,\Pi(R))&\leq\angle(\Pi(\ell), \Pi(z))+\angle(\Pi(z), \Pi(R))\\
&\leq 2u.
\end{split}
\end{equation*} 

Thus we must establish the bound
\begin{equation}\label{SStarSmall}
|\{R\in \mathcal{R}_{\operatorname{max}} \colon R^*\cap \tilde A_2\neq\emptyset\}|\lesssim (hs)^{-O(1)}u^{-1}.
\end{equation}

By Lemma \ref{coveringNumberOfDistinguishedSpheres},
\begin{equation}\label{X2SmallCoveringNumber}
\mathcal{E}_u(\tilde A_2)=\mathcal{E}_u(A_2)\lesssim s^{-O(1)} u^{-3}.
\end{equation}

Since the sets $\{CR^*\colon R\in \mathcal{R}_{\operatorname{max}}\}$ are $O_C(1)$ overlapping, in order to prove \eqref{SStarSmall}, it suffices to prove that if $R^*\cap \tilde A_2\neq\emptyset,$ then
\begin{equation}\label{SStarIntersectX1}
\mathcal{E}_u(8R^*\cap \tilde A_2)\gtrsim (hs)^{O(1)}u^{-2}.
\end{equation}

Suppose that $R^*\cap \tilde A_2\neq\emptyset$ and let $(z,\Pi)\in R^*\cap \tilde A_2$. Since $z\in 1\operatorname{-SBroad}_{s}(\Phi_1)$, there are $\geq s/u$ lines $\ell$ that point in pairwise $\geq u$ separated directions with $(x,\ell)\in \Gamma_1$ for some $x\in f_Z^{-1}(z)$. For each of these lines $\ell$, we have 
\begin{equation*}
\begin{split}
\angle(\Pi(\ell),\Pi(R))&\leq \angle(\Pi(\ell), \Pi(z)) + \angle(\Pi(z), \Pi(R)) \\
& < 4u.
\end{split}
\end{equation*} 
On each of these lines, we have $|\Gamma_1(\ell)|\geq h$, so we can select $\geq h/(2u)$ points that are all pairwise $u$ separated and have distance $\geq h/2$ from $z$. For each such point $x^\prime,$ we have 
\begin{equation*}
\begin{split}
\angle\big(\Pi\circ f_Z(x^\prime),\Pi(R)\big)&\leq \angle\big(\Pi\circ f_Z(x^\prime),\Pi(\ell)\big)+\angle\big(\Pi(\ell),\Pi(R)\big)\\
&\leq 8u.
\end{split}
\end{equation*} 
This gives us a set of $\geq hs/(2u^2)\gtrsim (sc)^{O(1)}u^{-2}$ points $f_Z(x^\prime)\in Z_1$ that are pairwise $\gtrsim u$ separated, are contained in $Z_1\cap N_{8u}(\Pi)$, and satisfy $\angle(\Pi(u),\Pi(R))\leq 8u$. This establishes \eqref{SStarIntersectX1}.
\subsubsection*{Understanding lines in $\Sigma_2$} 
We will prove that if $w\geq (cush\kappa)^{O(1)}$ is sufficiently small, then the lines in $\Sigma_2$ can be covered by $O(s^{-2})$ rectangular prisms of dimensions $2\times s\times s\times s$.  

Observe that for each $\ell\in\Sigma_2$, $\Gamma_0(\ell)$ has complexity $O(1)$. Thus if we select $h\gtrsim c$ sufficiently small, then there is an interval $I\subset \Gamma_0(\ell)$ of length $\geq 2h$ so that $\Pi(I)$ is connected. Since $\ell\in\Sigma_2$, $\Pi(I)$ must also have diameter $\geq u$. We will now fix a value of $h$ so that this holds. Define
 
$$
\Gamma_3 = \{(x,\ell)\in\Gamma_0\colon \ell\in\Sigma_2,\ \exists\ \textrm{an interval}\ x\in I\subset\Gamma_0(\ell)\ \textrm{of length}\ 2h\}.
$$
Observe that $|\Gamma_3(\ell)|\geq 2h\gtrsim c$ for all $\ell\in\Sigma_2$. For each $\ell\in\Sigma_2$, let $\gamma_\ell\subset\operatorname{Grass}(4;2)$ be a connected set of diameter $\geq u$ that is contained in $\Pi\circ f_Z(\Gamma(\ell))$. Let $\Phi_3$ be the set associated to $\Gamma_3$, in the sense of Definition \ref{defnAssociated}. Define
$$
\Phi_3^\prime = \{(z,v)\in \Phi_3\colon z\in 1\operatorname{-SBroad}_{s}(\Phi_3)\},
$$
and let $\Phi_3^{\prime\prime}=\Phi_3\backslash \Phi_3^\prime$. Let $\Gamma_3^\prime $ be the pre-image of $\Phi_3^\prime$ under the map $\Gamma_3\to \Phi_3$, and define $\Gamma_3^{\prime\prime}=\Gamma_3\backslash\Gamma_3^\prime$. Define

\begin{equation*}
\begin{split}
\Sigma_3 &= \{\ell\in\Sigma_2\colon |\Gamma_3^\prime(\ell)|\geq h\},\\
\Sigma_3^\prime &= \{\ell\in\Sigma_2\colon |\Gamma_3^{\prime\prime}(\ell)|\geq h\}.
\end{split}
\end{equation*}

Arguing as above, the lines in $\Sigma_3^\prime$ can be covered by $O(c^{-O(1)}s^{-2})$ rectangular prisms of dimensions $2\times s\times s\times s$. We must now repeat the above process one more time. Define 
$$
\Gamma_4 = \{(x,\ell)\in\Gamma_0\colon \ell\in\Sigma_3\}.
$$
Let $\Phi_4$ be the set associated to $\Gamma_4$, in the sense of Definition \ref{defnAssociated}. Define
$$
\Phi_4^\prime = \{(z,v)\in \Phi_4\colon z\in 1\operatorname{-SBroad}_{s}(\Phi_4)\},
$$
and let $\Phi_4^{\prime\prime}=\Phi_4\backslash \Phi_4^\prime$. Let $\Gamma_4^\prime $ be the pre-image of $\Phi_4^\prime$ under the map $\Gamma_4\to \Phi_4$, and define $\Gamma_4^{\prime\prime}=\Gamma_4\backslash\Gamma_4^\prime$. Define

\begin{equation*}
\begin{split}
\Sigma_4 &= \{\ell\in\Sigma_3\colon |\Gamma_4^\prime(\ell)|\geq h/2\},\\
\Sigma_4^\prime &= \{\ell\in\Sigma_3\colon |\Gamma_4^{\prime\prime}(\ell)|\geq h/2\}.
\end{split}
\end{equation*}
Again, the lines in $\Sigma_4^\prime$ can be covered by $O(c^{-O(1)}s^{-2})$ rectangular prisms of dimensions $2\times s\times s\times s$. We will show that if $w\geq (csu\kappa)^{O(1)}$ is sufficiently small, then $\Sigma_4$ is empty. The basic idea is as follows: If $\Sigma_4$ is not empty, then we will find lines $\ell_0$ and $\ell^\prime$ that are (quantitatively) skew so that there are many (i.e.~about $w^{-1}$) points $x\in\ell_0$ where the plane $\Pi\circ f_Z(x)$ intersects $\ell^\prime$. This implies that the plane $\Pi\circ f_Z(x)$ is contained in the hyperplane $H$ spanned by $\ell$ and $\ell^\prime$. Each of these planes $\Pi\circ f_Z(x)$ contains many lines (almost) contained in $\Sigma_2$, which will imply that the $w$--neighborhood of these planes each intersect $N_{w}(Z)$ in a set of measure roughly $w^2$. Since there are roughly $w^{-1}$ such planes, this implies that $|N_{w}(H)\cap N_w(Z)|$ has size roughly $w$. This contradicts the fact that $\Vert II(x)\Vert_\infty\geq\kappa$ on $Z$, which implies that $|N_{w}(H)\cap N_w(Z)|$ has size at most $\kappa^{-1/2}w^{3/2}$.

We will now make this argument precise. For each $\ell\in\Sigma_2$, let $T_{\ell}$ be the $w$-neighborhood of $\ell\cap B(0,1)$, and let $Y_2(T_{\ell})\subset T$ be the $w$-neighborhood of $\Gamma_0(\ell)$. Observe that $|Y_2(T_\ell)|\gtrsim c|T_\ell|$ for each such $\ell\in\Sigma_2$. Let $\tubes_2$ be a maximal set of essentially distinct tubes from $\{ T_{\ell}\colon \ell\in\Sigma_2\}$. 

Similarly, for each $\ell\in\Sigma_3$, let $T_{\ell}$ be the $w$-neighborhood of $\ell\cap B(0,1)$, and let $Y_3(T_{\ell})\subset T$ be the $w$-neighborhood of $\Gamma_3(\ell)$. Again, we have $|Y_3(T_\ell)|\gtrsim c|T_\ell|$ for each such $\ell\in\Sigma_3$. Let $\tubes_3$ be a maximal set of essentially distinct tubes from $\{ T_{\ell}\colon \ell\in\Sigma_3\}$. 

\begin{rem}\label{lotsOfTubesOutsideAPlane}
Observe that for every $T\in\tubes_3$ and every plane $\Pi$ containing the line coaxial with $T$, there are $\gtrsim (sc)^{O(1)}w^{-2}$ tubes $T^\prime\in \tubes_2$ with $T\cap T^\prime\neq\emptyset$, $\angle(v(T),v(T^\prime))\gtrsim (sch)^{O(1)}$, and $\angle(v(T^\prime), \Pi)\gtrsim  (uch)^{O(1)}$.
\end{rem}

Our next task is to establish the bound
\begin{equation}\label{tubes0Small}
|\tubes_2|\lesssim c^{-O(1)}w^{-3}.
\end{equation}
For each $x\in Z$, define
$$
\tubes_2(x) = \{T\in\tubes_2\colon x\in Y_2(T)\}.
$$ 
For each $T\in\tubes_2,$ define $v(T)=v(\ell)$, where $\ell$ is the line coaxial with $T$, and define
$$
T^* = \{(x,v)\in\RR^4\times S^3\colon\ x\in T,\ \angle(v,v(T))\leq w\}.
$$ 
Then the sets $\{T^*\}_{T\in\tubes_2}$ are  $O(1)$ overlapping.

Recall the set $A_1$ from \eqref{defnA1}. By Lemma \ref{coveringNumberOfDistinguishedSpheres}, $\mathcal{E}_w(A_1)\lesssim s^{-O(1)}w^{-3}$. Define
$$
G = \{(z,v)\in Z_1\times S^3\colon v\in S(z) \}.
$$
By Lemma \ref{coveringNumberSemiAlg}, we have 
\begin{equation}\label{YSmall}
\mathcal{E}_w(G)\lesssim s^{-O(1)}w^{-4}.
\end{equation} 

On the other hand, if $T_\ell\in\tubes_2$, then there are $\gtrsim \frac{c}{w}$ $w$-separated points $z\in T_\ell\cap Z_1$ with $\angle(v(\ell),\Pi(z))\leq w$ and thus $(z, v(T))\in G$. Thus if $T\in\tubes_2$, we have 
\begin{equation}\label{TStarIntersectsY}
N_{w}(T^*\cap G)\gtrsim cw^{-1}.
\end{equation}
Combining \eqref{YSmall} and \eqref{TStarIntersectsY}, we obtain \eqref{tubes0Small}.

We will now show that $\Sigma_4$ is empty. Suppose not. Let $\ell_0\in\Sigma_4$. Let $T_0$ be the $w$-neighborhood of $\ell\cap B(0,1)$, and let $Y_4(T)$ be the $w$-neighborhood of $\Gamma_4^\prime(\ell_4)$. Let $z_1,\ldots,z_{p}$, $p\gtrsim (csu)^{O(1)} w^{-1}$ be points in $f_Z(\Gamma_4^\prime(\ell_0))\subset Y_4(T_0)$ so that the planes $\Pi(z_i)$ point in $w$--separated directions. (Recall that each of these planes makes an angle $\leq w$ with $v(\ell_0)=v(T_0)$) 

For each $i=1,\ldots,p$, since $z_i\in Y_4(T_0)$, we can select a set $\tubes^{(i)}$ of $\gtrsim (cs)^{O(1)}w^{-1}$ tubes from $\tubes_3$ passing through $z_i$, so that each of these tubes makes an angle $\gtrsim (cs)^{O(1)}$ with $v(T_0)$. Since the set $\tubes^{(i)}$ is contained in the $s^{-O(1)}w$--neighborhood of the plane $\Pi(z_i)$, and these planes point in $w$--separated directions, we can refine the set of indices $1,\ldots,p$ to a new indexing set $1,\ldots,p^\prime$ with $p^\prime \gtrsim (csu)^{O(1)} w^{-1}$ so that the corresponding sets of tubes $\{\tubes^{(i)}\}_{i=1}^{p^\prime}$ are disjoint. As discussed in Remark \ref{lotsOfTubesOutsideAPlane}, for each index $i$ and each $T\in\tubes^{(i)}$, there exist $\gtrsim (csu)^{O(1)}w^{-2}$ tubes from $\tubes$ that intersect $T$, make an angle $\gtrsim (csu)^{O(1)}$ with $v(T)$, and that make an angle $\gtrsim (csu)^{O(1)}$ with the plane spanned by $v(T_0)$ and $v(T)$ (and thus make an angle $\gtrsim (csu)^{O(1)}$ with the plane $\Pi(z_i)$). We can also require that each of these tubes intersect $T$ in a point that has distance $\gtrsim (csu)^{O(1)}$ from $T_0$, i.e. each of these tubes is  $\gtrsim (csu)^{O(1)}$ skew to $T_0$.

Thus if $C=O(1)$ is chosen sufficiently large, there are $\gtrsim (csu)^{O(1)} w^{-4}$ pairs
\begin{equation}\label{goodPairs}
\begin{split}
\{(T,T^{\prime})\in\tubes_3\times\tubes_2\colon & T_0 \cap T\neq\emptyset,\ \angle(v(T_0),v(T))\gtrsim (csu)^{C},\\
& T\cap T^{\prime}\neq\emptyset,\ \angle(v(T),v(T^\prime))\gtrsim (csu)^{C},\\
& T\ \textrm{and}\ T^{\prime}\ \textrm{are}\ \gtrsim (csu)^{C}\ \textrm{skew}\}.
\end{split}
\end{equation}

Since $|\tubes_2|\lesssim c^{-O(1)}w^{-3}$, we can select a tube $T^\prime\in\tubes_2$ that is $\gtrsim (csu)^{C}$ skew to $T_0$, so that there are $\gtrsim (csu)^{O(1)} w^{-1}$ tubes $T\in\tubes_3$ with $(T,T^\prime)\in\eqref{goodPairs}.$ Note that at most $\lesssim (csu)^{-1}$ of these tubes $T\in\tubes_3$ can lie in the $w$--neighborhood of a common plane containing $v(T_0)$, since $T^\prime$ is $\gtrsim (csu)^{C}$ skew to $T_0$. Thus we can select $\gtrsim (csu)^{O(1)} w^{-1}$ tubes $T$ with $(T,T^\prime)\in\eqref{goodPairs}$ so that the planes $\{\operatorname{span}(v(T_0), v(T)\}$ point in $w$--separated directions.

Let $H$ be the hyperplane containing the lines coaxial with $T_0$ and $T^\prime$. Observe that if $(T,T^\prime)\in\eqref{goodPairs}$, and if $x\in T_0\cap T$, then $\angle(\Pi(x), H)\lesssim (csu)^{-O(1)}w$. Re-indexing the sets of tubes $\{\tubes^{(1)},\ldots,\tubes^{(p^\prime)}\}$ again, we can select sets $\tubes^{(1)},\ldots,\tubes^{(p^{\prime\prime)}},\ p^{\prime\prime}\gtrsim (csu)^{O(1)}w^{-1}$ so that every set of tubes $\tubes^{(i)}$ contains a tube $T$ with $(T,T^\prime)\in\eqref{goodPairs}$. This means that for each index $i$, the tubes in $\tubes^{(i)}$ are contained in the $\lesssim (csu)^{-O(1)}w$--neighborhood of the hyperplane $H$ (see Figure \ref{twoTubePlaneFig}). We have

$$
\Big|\bigcup_{i=1}^{p^{\prime\prime}}\bigcup_{T\in\tubes^{(i)}} Y_2(T)\Big|\gtrsim (csu)^{O(1)}w.
$$
Since $Y_2(T)\subset N_{w}(Z)$ for each tube $T$ in the above union, we have
$$
|N_{w}(Z) \cap N_{\lesssim(csu)^{-O(1)}w }(H)|\gtrsim (csu)^{O(1)}w.
$$
\begin{figure}[h!]
 \centering
\begin{overpic}[width=0.6\textwidth]{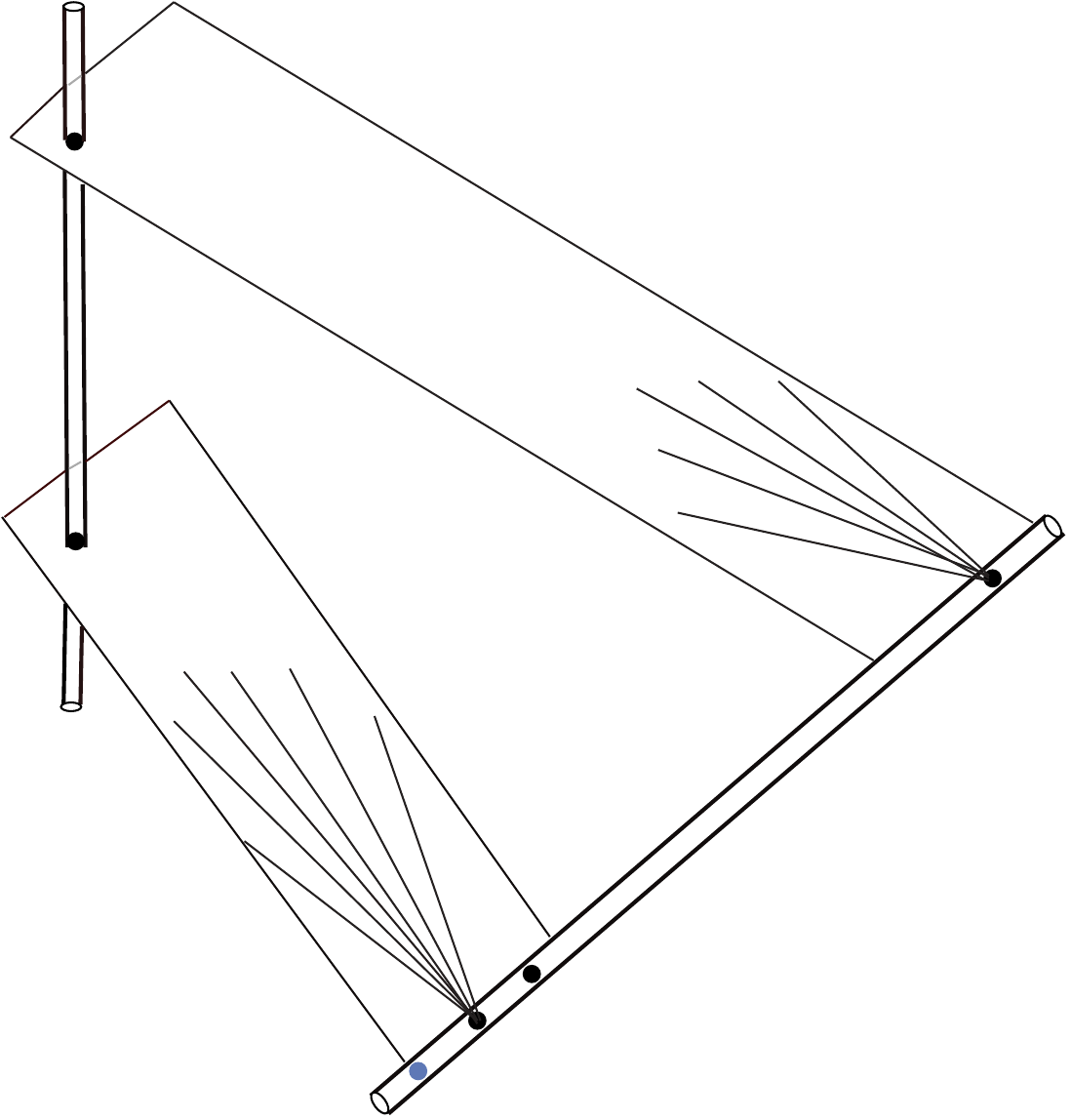}
 \put (38, 1) {$z_1$}
 \put (44, 6) {$z_2$}
 \put (49, 10) {$z_3$}
 \put (90, 45) {$z_{p^{\prime\prime}}$}
 \put (51,66) {$\tubes^{(p^{\prime\prime})}$}
 \put (9,51) {$T^\prime\cap \Pi(z_2)$}
 \put (9,70) {$T^\prime$}
 \put (65,22) {$T_0$}
 \put (9,87) {$T^\prime\cap \Pi(z_{p^{\prime\prime}})$}
\end{overpic}
\caption{The tubes $T_0$ and $T^\prime$; the points $z_1,\ldots,z_{p^{\prime\prime}}\in T_0$, the planes $\Pi(z_1),\ldots,\Pi(z_{p^{\prime\prime}})$, and the sets of tubes $\tubes^{(1)},\ldots,\tubes^{(p^{\prime\prime})}$. This entire figure is ostensibly contained in $\RR^4$, but in fact it is contained in the $\lesssim (csu)^{-O(1)}w$--neighborhood of the hyperplane spanned by the lines coaxial with $T_0$ and $T^\prime$. }\label{twoTubePlaneFig}
\end{figure}
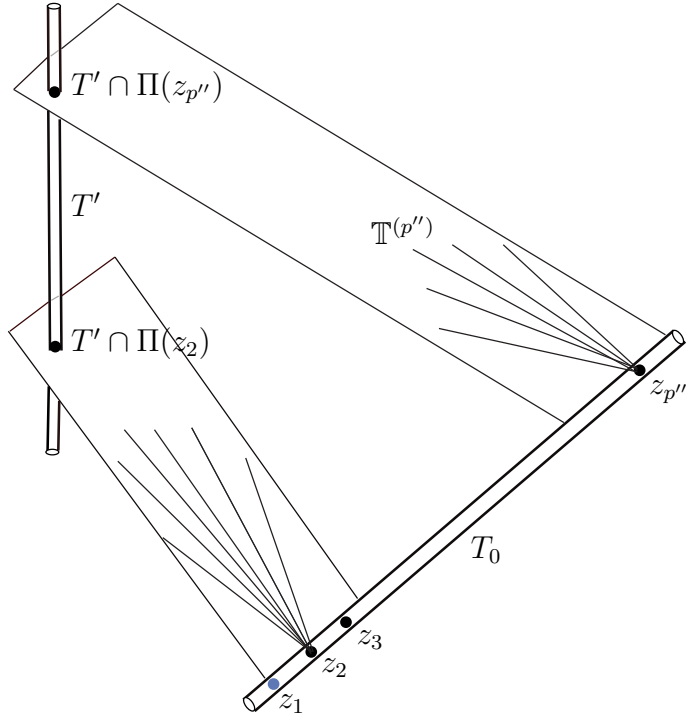

But by Lemma \ref{largeFundFormEscapesPlane}, we have
$$
|N_{w}(Z) \cap N_{\lesssim (csu)^{-O(1)} w }(H)|\lesssim  (csu)^{-O(1)}\kappa^{-1/2} w^{3/2},
$$
and thus
\begin{equation}\label{impossibleIneqW}
(csu)^{O(1)}w \lesssim (csu)^{-O(1)}\kappa^{-1/2} w^{3/2}.
\end{equation}
If $w\gtrsim (c\kappa us)^{O(1)}$ is chosen sufficiently small, then \eqref{impossibleIneqW} is impossible, which contradicts the assumption that $\Sigma_4\neq\emptyset$. We conclude that $\Sigma_4=\emptyset$, which completes the proof of Lemma \ref{3NarrowCovered}.
\end{proof}

Lemma \ref{3NarrowCovered} can be used to understand hypersurfaces that are doubly-ruled by planes.
\begin{prop}\label{22NarrowCovered}
Let $\delta,u,s,\kappa,c,D,E$ be parameters with $0<\delta<u<s<c$ and $\delta<\kappa$. Then there exists a number $w\gtrsim_{D,E}(us\kappa c)^{O(1)}$ so that the following holds.

Let $P\in\RR[x_1,\ldots,x_4]$ be a polynomial of degree at most $D$ and let 
\begin{equation}\label{defnOfZ}
Z \subset \{z\in Z(P)\cap B(0,1)\colon 1\leq |\nabla P(z)|\leq 2,\ \Vert II(z)\Vert_\infty \geq \kappa\}.
\end{equation}
Let $\Sigma\subset\Sigma_{\delta,c}(Z)$, and let $\Gamma\subset\Gamma(N_{\delta}(Z),\Sigma)$. Suppose that $Z,\Sigma,$ and $\Gamma$ are semi-algebraic sets of complexity at most $E$. Let $\Phi$ be associated to $\Gamma$, in the sense of Definition \ref{defnAssociated}. Suppose that
\begin{align}
&Z\subset (2,2)\operatorname{-Narrow}_{w}(\Phi),\label{Z22Narrow}\\
&|\Gamma(\ell)|\geq c\quad\textrm{for all}\ \ell\in\Sigma. %\label{GammaEllLargeSize}
\end{align}

Then we can write $\Sigma = \Sigma^\prime\cup\Sigma^{\prime\prime}$, where the lines in $\Sigma^\prime$ can be covered by $O_{D,E}(c^{-O(1)}s^{-2})$ rectangular prisms of dimensions $2\times s\times s\times s$, and the lines in $\Sigma^{\prime\prime}$ can be covered by $O_{D,E}((cs)^{-O(1)}u^{-1})$ rectangular prisms of dimensions $2\times 2\times u\times u$.
\end{prop}
\begin{proof}
Let $w\gtrsim_{D,E}(us\kappa c)^{O(1)}$ be the constant from Lemma \ref{3NarrowCovered} associated to the values $\delta^\prime=\delta,u^\prime=u,s^\prime=s,\kappa^\prime=\kappa,c^\prime=c/2,D^\prime=D,E^\prime = O_{D,E}(1)$. Define  
$$
A = \{(z,S_1,S_2)\in Z \times (\mathcal{S}_1)^2 \colon\ \min\big(\angle(v, S_1),\ \angle(v,S_2)\big)\leq w\ \forall\ v\in \pi_S(\Phi(z))\}.
$$
Since $Z\subset (2,2)\operatorname{-Narrow}_{w}(\Phi)$, the projection $A\mapsto Z$ is onto. Use Lemma \ref{selPtFiberProp} to select a set $A^\prime\subset A$ so that the map $\pi\colon A^\prime\to Z$ is a bijection. Define the semi-algebraic functions $S_1(z)$ and $S_2(z)\colon Z\to \mathcal{S}_1$ so that for each $z\in Z,$ $(z,S_1(z),S_2(z))\in A^\prime$. For $i=1,2$, define
$$
\Gamma_i = \{(x,\ell)\in\Gamma\colon \angle( v(\ell),S_i\circ f_Z(x)\leq w \}.
$$
By \eqref{Z22Narrow}, for every $(x,\ell)\in\Gamma$ we have that 
$$
\angle\big(v(\ell),\ S_1\circ f_Z(x)\big)\leq w \quad\textrm{and/or}\quad \angle\big(v(\ell),\ S_2\circ f_Z(x)\big)\leq w. 
$$
Thus $\Gamma=\Gamma_1\cup\Gamma_2$. For $i=1,2,$ define 
$$
\Sigma_i=\{\ell\in\Sigma\colon |\Gamma_i(\ell)|\geq c/2\}.
$$ 
Then $\Sigma = \Sigma_1\cup\Sigma_2$. We have that $Z,\Gamma_i,$ and $\Sigma_i$, $i=1,2$, are semi-algebraic of complexity $E^\prime=O_{D,E}(1)$. For $i=1,2$, apply Lemma \ref{3NarrowCovered} to the data $P, Z, \Sigma_i, \Gamma_i,$ with the parameters $\delta^\prime,u^\prime,s^\prime,\kappa^\prime,c^\prime,D^\prime,$ and $E^\prime$ described above, and let $\Sigma_i^{\prime}$ and $\Sigma_i^{\prime\prime}$ be the output from the lemma. Define $\Sigma^\prime = \Sigma_1^\prime\cup \Sigma_2^\prime$ and define $\Sigma^{\prime\prime}=\Sigma_1^{\prime\prime}\cup\Sigma_2^{\prime\prime}$.
\end{proof}

\section{Broad varieties}\label{broadVarietiesSection}
In this section, we will consider the region where $Z$ is robustly 1-broad, $(2,2)$-broad, and has large second fundamental form. We will show that many of the lines that have large intersection with this region must lie near a quadratic hypersurface; this will be the set of lines $\Sigma_4^\prime$ from the statement of Theorem \ref{discretizedSeveri}. This result will be proved in Proposition \ref{discretizedSegreNonDegenProp}, which is the main result of this section.  

\subsection{Neighborhoods of quadratic curves}
\begin{lem}\label{smallNeighborhoodOfQuadraticCurve}
Let $Q(x_1,x_2) = a_{11}x_1^2 + a_{22}x_2^2 + a_{12}x_1x_2$ be a monic quadratic polynomial. Let 
\begin{equation*}
\begin{split}
S_1 &= S^2 \cap Z\Big( \big(a_{12}+\mathcal{R}[(a_{12}^2 - 4a_{11}a_{22})^{1/2}]\big)x_1 + 2a_{22}x_2\Big),\\
S_2 & = S^2 \cap Z\Big( \big(a_{12}-\mathcal{R}[(a_{12}^2 - 4a_{11}a_{22})^{1/2}]\big)x_1 + 2a_{22}x_2\Big),
\end{split}
\end{equation*}
where $\mathcal{R}[z]$ is the real part of $z$. Then there is an absolute constant $C$ so that for each $t> 0$,
\begin{equation}\label{neighborhoodOfZq}
\{x\in S^2\colon |Q(x)|\leq t\} \subset N_{Ct^{1/2}}(S_1\cup S_2).
\end{equation}
\end{lem}
\begin{rem}
The requirement that $Q$ be monic can be replaced by the condition that the largest coefficient of $Q$ has magnitude $A>0$. Then the constant $C$ in \eqref{neighborhoodOfZq} depends on $A$.
\end{rem}

\begin{defn}\label{defnDegenerage}
Let $Q\in\RR[x_1,x_2,x_3]$ be a homogeneous polynomial of degree 2. We say that $Q$ is $w$-degenerate if there exist great circles $S_1,S_2\subset S^2$ so that $Z(Q)\cap S^2 \subset N_{w}(S_1\cup S_2)$. If $Q$ is not $w$-degenerate, then we call it $w$-non-degenerate.
\end{defn}

\begin{lem}\label{smallPieceGreatCircle}
Let $Q\in \RR[x_1,x_2,x_3]$ be a homogeneous quadratic polynomial. Suppose that $Q$ is $w$-non-degenerate. Then for every $t<w$ and for every great circle $S \subset S^2$, we have
$$
|S_2 \cap Z(Q) \cap N_{t}(S)|\lesssim (t/w)^{1/2},
$$
where $|\cdot|$ denotes one-dimensional Lebesgue measure.
\end{lem}

\begin{lem}\label{quadraticDichotomy}
Let $Q(x_1,x_2,x_3)$ be a monic homogeneous polynomial of degree 2 and let $w>0.$ Then there is an absolute constant $c>0$ so that at least one of the following two things must hold

\begin{enumerate}
\item\label{situation2} There exist two great circles $S_1,S_2\subset S^2$ so that
$$
\{x\in S^2\colon |Q(x)|\leq cw^2\}\subset  N_{w}(S_1\cup S_2).
$$
\item\label{situation1} For each $t>0$, we have
$$
\{x\in S^2\colon |Q(x)|\leq ct\}\subset  N_{t/w^2}(Z(Q)).
$$
\end{enumerate}
\end{lem}
\begin{proof}

Let $c_1>0$ be a constant to be determined later. We will consider two cases. Case (A): there exists a point $p\in Z(Q)\cap S^2\subset \RR^3$ where the map $x\mapsto Q(x)$ has small derivative, i.e.~$|DQ(p)|\leq c_1w^2$. We will show that Item \ref{situation2} must hold. After a rotation, we can assume that $p = (1,0,0)$. After applying this rotation, we have 
$$
Q(x_1,x_2,x_3) = a_{22}x_2^2 + a_{33}x_3^2 + a_{12}x_1x_2 + a_{13}x_1x_3 + a_{23}x_2x_3,
$$ 
and 
$$
DQ(p)=(\partial_{x_2}Q(1,0,0),\ \partial_{x_3}Q(1,0,0)) = (a_{12},a_{13}).
$$
Thus
$$
Q(x_1,x_2,x_3) =   a_{22}x_2^2 + a_{33}x_3^2 + a_{23}x_2x_3 + O(c_1w^2)(x_1x_2 + x_1x_3),
$$
where at least one of $a_{22},a_{33},a_{23}$ has magnitude $\sim 1$.

By Lemma \ref{smallNeighborhoodOfQuadraticCurve}, we have that if $c_2>0$ is chosen sufficiently small (independent of $c_1$), then
$$
\{x\in S^2\colon |a_{22}x_2^2 + a_{33}x_3^2 + a_{23}x_2x_3|\leq 10 c_2 w^2 \} \subset N_{w}(S_1\cup S_2),
$$
where 
\begin{equation*}
\begin{split}
S_1 &= S^2 \cap Z\Big( \big(a_{23}+\mathcal{R}[(a_{23}^2 - 4a_{22}a_{33})^{1/2}]\big)x_2 + 2a_{33}x_3\Big),\\
S_2 & = S^2 \cap Z\Big( \big(a_{23}-\mathcal{R}[(a_{23}^2 - 4a_{22}a_{33})^{1/2}]\big)x_2 + 2a_{33}x_3\Big).
\end{split}
\end{equation*}

Next, if $c_1>0$ is chosen sufficiently small (depending on $c_2$), then
$$
\{x\in S^2\colon |Q(x)|\leq c_2w^2\}\subset \{x\in S^2\colon|a_{22}x_2^2 + a_{33}x_3^2 + a_{23}x_2x_3|\leq 10 c_2w^2\},
$$
which completes the analysis of Case (A). 

Now suppose we are in Case (B): $|DQ(p)|\geq c_1w^2$ for all $p\in Z(Q)\cap S^2$. Since $Q$ is quadratic, $DQ\colon\RR^3\to\RR^3$ is a linear map. If $c_3>0$ is chosen sufficiently small (depending on $c_1$), then $|Q(x)|\leq c_3t$ implies $\dist(t, Z(Q))\leq t/w^2$. Thus Item \ref{situation1} must hold. 

To complete the proof, choose $c = \min(c_2,c_3)$.
\end{proof}
\begin{lem}\label{quadraticConeOfP}
Let $P\in\RR[x_1,x_2,x_3,x_4]$ be a polynomial of degree at most $D$. Let 
$$
Z \subset \{x\in Z(P)\cap B(0,1),\ 1\leq|\nabla P(x)| \leq 2,\ \Vert II(x)\Vert_{\infty}\geq \kappa\}.
$$
Let $\Gamma\subset \Gamma(N_{\delta}(Z),\mathcal{L})$ and let $\Phi\subset Z\times S^3$ be associated to $\Gamma$, in the sense of Definition \ref{defnAssociated}. Suppose that
\begin{itemize}
\item $Z\subset (2,2)\operatorname{-Broad}_w(\Phi)$.
\item For each $z\in Z$ and each $v\in \pi_S(\Phi(z))$, we have 
\begin{equation}\label{smallDotProduct}
|(v\cdot\nabla)^jP(z)|\leq K\delta,\quad j=1,2.
\end{equation}
\end{itemize}
Then for each $z\in Z$, the vectors in $\pi_S(\Phi(z))$ are contained in the $\delta \big(K/(w\kappa)\big)^{O(1)}$-neighborhood of the set
\begin{equation}\label{coneOfX}
\mathcal{C}_z = \{v \in S^3\colon (v\cdot\nabla) P(z) = 0,\ (v\cdot\nabla)^2 P(z) = 0\}.
\end{equation}
\end{lem}
\begin{proof} 
Let $x\in Z$. After a translation and rotation, we can assume that $x=0$ and $N(0) = (1,0,0,0)$. Then we can expand
$$
P(x_1,x_2,x_3,x_4)=x_1+ \sum_{\substack{|I|=2\\\textrm{yes}\ x_1}}a_Ix^I + \sum_{\substack{|I|=2\\\textrm{no}\ x_1}}a_Ix^I + \sum_{|I|>2}a_Ix^I,
$$
where the first sum is taken over all multi-indices $I$ of length two that include at least one $x_1$ term, and the second sum includes all the other multi-indices of length two. 

Let $v\in \pi_S(\Phi(x))$; we can write $v=(v_1,v_2,v_3,v_4)$. Since $v$ satisfies \eqref{smallDotProduct}, we have  $|v_1|\leq K\delta$. Define
$$
A=\Vert II(0)\Vert_{\infty}^{-1} \left[\begin{array}{ccc} a_{22} & a_{23} & a_{24} \\ a_{23} & a_{33} & a_{34} \\ a_{24} & a_{34} & a_{44} \end{array}\right].
$$
Since $v$ satisfies \eqref{smallDotProduct} and $\Vert II(0)\Vert_{\infty}\geq\kappa$, we have 

\begin{equation}\label{smallAProduct}
\Big| [v_2,v_3,v_4]^T\ A\ [v_2,v_3,v_4]\Big| \lesssim \kappa^{-1}K\delta.
\end{equation}
Now consider the function 
$$
Q(v_1,v_2,v_3) = [v_2,v_3,v_4]^T\ A\ [v_2,v_3,v_4].
$$

Since $0\in(2,2)\operatorname{-Broad}_w(\Phi)$ (remember, originally we had $z \in(2,2)\operatorname{-Broad}_w(\Phi)$, but we applied a translation sending $z$ to $0$), the set of unit vectors $(v_2,v_3,v_4)$ satisfying \eqref{smallAProduct} cannot be contained in the $w$--neighborhood of the union of two great circles in $S^2$. Thus by Lemma \ref{quadraticDichotomy}, we have that 
$$
(v_2,v_3,v_4)\in S^2\cap N_{c\kappa^{-1}K\delta/w^2}(Z(Q)),
$$
where $c>0$ is an absolute constant. Thus  
$$
v\subset N_{t}(\mathcal{C}_z),\quad\ \textrm{where}\ t \lesssim \delta K/(\kappa w^2).\qedhere
$$
\end{proof}
\begin{defn}\label{quadConeDefn}
In \eqref{coneOfX} above, we defined the set 
\begin{equation}\label{defnOfCx}
\mathcal{C}_z = \{v \in S^3\colon (v\cdot\nabla) P(z) = 0,\ (v\cdot\nabla)^2 P(z) = 0\}.
\end{equation}
We will call this the quadratic cone of $Z(P)$ with vertex $z$. More generally, any set of the form \eqref{defnOfCx} will be called a quadratic cone. Following Definition \ref{defnDegenerage}, we say that the quadratic cone $\mathcal{C}_z$ is $w$-degenerate if there exist great circles 
$$
S_1,S_2\subset \{v\in S^3\colon (v\cdot\nabla)P(z)=0\}
$$ 
so that $\mathcal{C}_z \subset N_{w}(S_1\cup S_2)$. Otherwise we say $\mathcal{C}_z$ is $w$-non-degenerate.

Define
$$
\tilde{\mathcal{C}}_z = z + \operatorname{span}(\mathcal{C}_z);
$$
this is a two-dimensional algebraic variety in $\RR^4$; it is the union of all lines that intersect $z$ and also intersect the curve $z+\mathcal{C}_z$. We say that $\tilde{\mathcal{C}}_z$ is $w$-non-degenerate if $\mathcal{C}_z$ is $w$-non-degenerate. Observe that $\tilde{\mathcal{C}_z}$ is a degree-two algebraic surface; it can be defined as the common zero locus of a degree one and a degree two polynomial in $\RR[x_1,x_2,x_3,x_4]$. If $\mathcal{C}_z$ is a quadratic cone, $\ell\in\lines,$ $z\in\ell,$ and $\operatorname{dist}\big(v(\ell), \mathcal{C}_z\big)=t$, then $\ell\cap B(0,1)\subset N_{t}(\tilde{\mathcal{C}}_z)$.
\end{defn}
\subsection{Unions of tubes}\label{unionsOfTubesSection}
For the next lemma, we will introduce some standard notation from the Kakeya problem. This notation will be used throughout the remainder of this section. Let $\tubes$ be a set of essentially distinct $\delta$-tubes, i.e.~a set of $\delta$-neighborhoods of unit line segments so that no tube is contained in the two-fold dilate of any other. For each tube $T\in\tubes$, let $Y(T)\subset T$. For each $x\in\RR^4$, define 
$$
\tubes(x) = \{T\in\tubes\colon x\in Y(T)\}.
$$ 
If the set $Y$ is ambiguous, we will sometimes use the notation $\tubes_Y(x)$ in place of $\tubes(x)$. For each $T\in\tubes,$ define $v(T)$ to be the direction of the line coaxial with $T$. Thus for example  
$$
v(\tubes(x))=\{v(T)\colon T\in\tubes(x)\}.
$$ 
For each $T_0\in\tubes,$ define 
$$
H(T_0)=\{T\in\tubes\colon Y(T_0)\cap Y(T)\neq\emptyset\}.
$$
If the set $Y$ is ambiguous, we will sometimes use the notation $H_Y(T_0)$ in place of $H(T_0)$.

The next lemma says that if $\tubes$ is a set of tubes, and if the tubes passing through a typical point lie near a non-degenerate cone, then the tubes in a typical hairbrush are mostly disjoint and thus their union has large volume. This is a variant of Wolff's ``hairbrush argument'' from \cite{W}. However, unlike in \cite{W} we do not assume that the tubes point in different directions.

\begin{lem}\label{hairbrushHasLargeVolumeLem}
Let $\delta,\lambda,t>0$. Let $\tubes$ be a set of essentially distinct $\delta$-tubes. For each $T\in\tubes$, let $Y(T)\subset T$ with $Y(T)\geq\lambda|T|$. Let $T_0\in\tubes$. Suppose that $|H(T_0)|\geq t\delta^{-2}$ and that for every $x\in Y(T_0)$, the vectors $v(\tubes(x))$ are contained in the $K\delta$--neighborhood of a $w$-non-degenerate cone $\mathcal{C}_x$. Then 
$$
\Big|\bigcup_{T\in H(T_0)}Y(T)\Big|\geq (w \lambda t/K)^{O(1)}\delta.
$$
\end{lem}
\begin{proof}
By pigeonholing, we can select a set of $\geq t\lambda/2\delta$ points $x\in Y(T_0)$ that are $\delta$ separated and that satisfy $|\tubes(x)| \geq \frac{1}{2}t\lambda\delta^{-1}$. The line coaxial with $T_0$ passes through $x$ and makes an angle $\leq K\delta$ with a line $\tilde\ell$ in the cone $\tilde{\mathcal{C}}_x$ (see Figure \ref{ConeFig}). Let $\Pi_x$ be the plane that is tangent to $\tilde{\mathcal{C}}_x$ along $\tilde\ell$. 

\begin{figure}[h!]
 \centering
\begin{overpic}[width=0.5\textwidth]{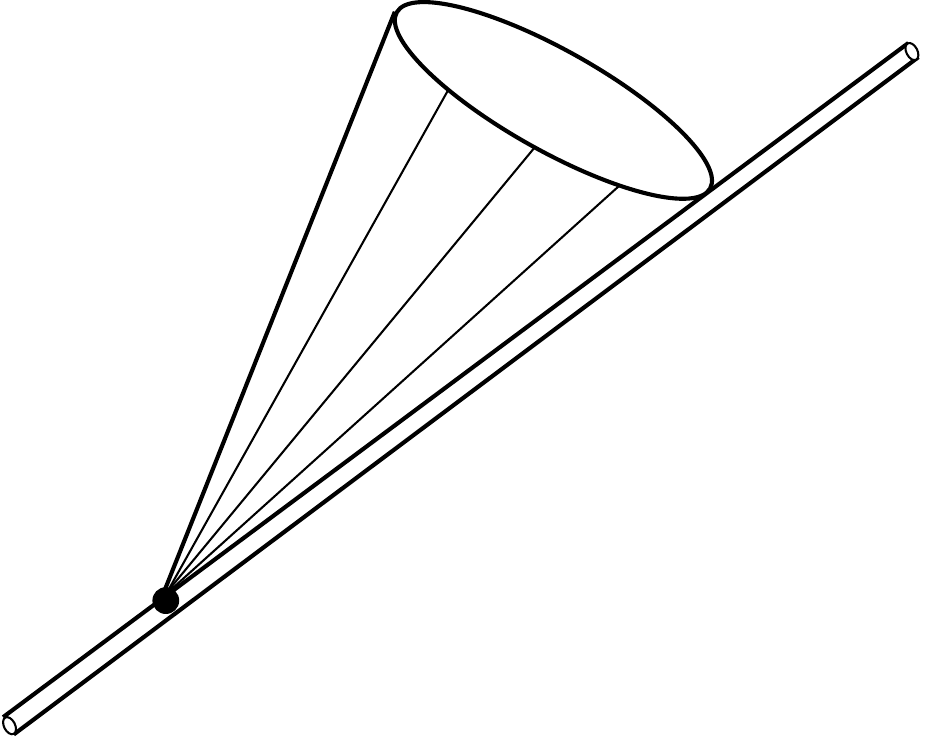}
\put (18,9) {$x$}
\put (90,61) {$T_0$}
\put (54,69) {$\tilde{\mathcal{C}}_x$}
\end{overpic}
\caption{The tube $T_0$, the point $x$, and the cone $\tilde{\mathcal{C}}_x$. The plane $\Pi_x$ (not pictured) contains the line coaxial with $T_0$ and is tangent to $\tilde{\mathcal{C}}_x$.}\label{ConeFig}
\end{figure}

Let $p  =  \frac{1}{16}wt^2\lambda^2 K^{-2}$, and let
$$
\tubes(x)^\prime =\{T\in \tubes(x)\colon \angle(v(T),\Pi(x))\geq p  \}.
$$
Since $\mathcal{C}_x$ is $w$--non-degenerate, by Lemma \ref{smallPieceGreatCircle}, we have that 
$$
|\{v\in S^2\colon \angle(v, \Pi_x)\leq p,\ v\in N_{K\delta}(\mathcal{C}_x)\}|\lesssim (K\delta)(p/w)^{1/2},
$$
where $|\cdot|$ denotes two-dimensional Haar measure on the sphere $S^2$; this set can contain at most $K(p/w)^{1/2}\delta^{-1}$ $\delta$-separated points on $S^2$, which implies that 
\begin{equation*}
\begin{split}
|\tubes(x)\backslash \tubes^\prime(x)|&\leq K(p/w)^{1/2}\delta^{-1}\\
& \leq |\tubes(x)|/2,
\end{split}
\end{equation*}
and thus $|\tubes^\prime(x)|\geq \frac{1}{4} t\lambda\delta^{-1}$ for each of the values of $x$ chosen above. Furthermore, for every plane $\Pi$ containing the line coaxial with $T$, we have that 
$$
|\{T\in\tubes^\prime(x)\colon T\subset N_{10\delta}(\Pi)\}|\lesssim \big( K/(w\lambda t)\big)^{O(1)}.
$$

Define $H^\prime(T)=\bigcup_x \tubes^\prime(x);$ all of these tubes intersect $T_0$. We have that
\begin{equation}\label{manyTubesInHPrimeT}
|H^\prime(T)|\geq \big(w\lambda t)/K\big)^{O(1)} \delta^{-2}.
\end{equation}
Furthermore, for each plane $\Pi$ containing the line coaxial with $T_0$; for each point $z\in T_0$; and for each $\delta\leq \rho\leq 1$, we have
\begin{equation}\label{fewTubesInPlane}
|\{T\in H^\prime(T_0)\colon T\subset N_{10\delta}(\Pi),\ \operatorname{dist}(z, T\cap T_0)\leq \rho\}| \lesssim \big( K/(w\lambda t)\big)^{O(1)}(\rho/\delta).
\end{equation}
Wolff's hairbrush argument from \cite{W} says that the union of any set of tubes intersecting $T_0$ that satisfy \eqref{manyTubesInHPrimeT} and \eqref{fewTubesInPlane} must have volume $\gtrsim\big( w\lambda t/K\big)^{O(1)}\delta$. Thus
$$
\Big|\bigcup_{T\in H(T_0)}Y(T)\Big| \geq \Big|\bigcup_{T\in H^\prime(T_0)}Y(T)\Big|\gtrsim \big( w\lambda t/K\big)^{O(1)}\delta.
\qedhere
$$
\end{proof}

The following lemma gives sufficient conditions for a semi-algebraic subset of a hypersurface in $\RR^4$ to be large (specifically, for it to have $\delta$-covering number roughly $\delta^{-3})$. In short, if a semi-algebraic subset of a hypersurface contains at least one line whose hairbrush contains many cones, then the set must be large.

\begin{lem}\label{hairbrushMakesBigEntropy}
Let $\delta,c,s,w,\kappa$ be positive real numbers. Let $P$ be a polynomial of degree at most $D$. Let 
\begin{equation}\label{ZIsCurved}
Z_2\subset Z_1 \subset \{x\in Z(P)\cap B(0,1),\ 1\leq|\nabla P(x)| \leq 2,\ \Vert II(x)\Vert_{\infty}\geq \kappa\}.
\end{equation}
Let $\emptyset\neq \Sigma_2\subset\Sigma_1$ with $\Sigma_i\subset \Sigma_{\delta,c}(Z_i)$ for $i=1,2$. Let $\Gamma_2\subset\Gamma_1$ with $\Gamma_i\subset \Gamma(N_{\delta}(Z_i), \Sigma_i)$ for $i=1,2.$ Suppose that the sets $Z_i,\Sigma_i,$ and $\Gamma_i$, $i=1,2$ are semi-algebraic of complexity at most $E$. For each $i=1,2$, let $\Phi_i\subset Z_i\times S^3$ be associated to $\Gamma_i$, in the sense of Definition \ref{defnAssociated}. Suppose that
\begin{align}
&Z_2 \subset 1\operatorname{-SBroad}_{s}(\Phi_1)\cap(2,2)\operatorname{-Broad}_{w}(\Phi_1) , \label{Z1BroadChain}\\
&|\Gamma_i(\ell)|\geq c\ \textrm{for each}\ \ell\in\Sigma_i,\ i=1,2, \label{allEllHaveBigIntersectionInSigma}\\
&|(v\cdot\nabla)^jP(z)|\leq K\delta,\ j=1,2,\ \textrm{for each}\ z\in Z_2\ \textrm{and each}\ v\in \pi_S(\Phi_2(z)). \label{smallDotProdEachI}
\end{align}
Then 
$$
\mathcal{E}_\delta(Z_1)\gtrsim_{D,E}(scw\kappa)^{O(1)}\delta^{-3}.
$$
\end{lem}
\begin{proof}
For $i=1,2$, define $\tubes_i$ to be a maximal $\delta$-separated subset of $\Sigma_i$ and for each $T\in\tubes_i$, define $Y_i(T)$ to be the $\delta$-neighborhood of $\Gamma_i(T)$. Since $\Sigma_2$ is non-empty, there exists a tube $T_0\in\tubes_2$. By Lemma \ref{quadraticConeOfP}, the tube $T_0$ and the pair $(\tubes,Y)$ satisfy the hypotheses of Lemma \ref{hairbrushHasLargeVolumeLem}. Applying Lemma \ref{hairbrushHasLargeVolumeLem}, we conclude that
$$
\Big|\bigcup_{T\in \tubes_1 \colon Y_1(T)\cap Y_2(T_0)\neq\emptyset} Y_1(T)\Big| \gtrsim (scw\kappa)^{O(1)}\delta.
$$
But since the above set is contained in $N_{\delta}(Z_1)$, we have $\mathcal{E}_\delta(Z_1)\gtrsim (scw\kappa)^{O(1)}\delta^{-3}.$
\end{proof}

\subsection{Lines in broad varieties lie near a quadratic hypersurface}
We are now ready to state the main result of this section.

\begin{prop}\label{discretizedSegreNonDegenProp}
Let $\delta,s,w,\kappa,t$ be positive real numbers. Let $P$ be a polynomial of degree at most $D$. Let 
\begin{equation}\label{ZIsCurved}
Z \subset \{x\in Z(P)\cap B(0,1),\ 1\leq|\nabla P(x)| \leq 2,\ \Vert II(x)\Vert_{\infty}\geq \kappa\}.
\end{equation}
Let $\Sigma\subset\lines$ with $|\mathcal{E}_\delta(\Sigma)|\geq L^{-1}\delta^{-3}$ and let $\Gamma\subset \Gamma(N_{\delta}(Z), \Sigma)$. Suppose that $Z,\Sigma,$ and $\Gamma$ are semi-algebraic of complexity at most $E$. Let $\Phi\subset Z\times S^3$ be associated to $\Gamma$, in the sense of Definition \ref{defnAssociated}. Suppose that
\begin{align}
&Z \subset 1\operatorname{-SBroad}_{s}(\Phi)\cap (2,2)\operatorname{-Broad}_{w}(\Phi) , \label{WIs1SBroadAndWIs22Broad}\\
%&Z \subset ,\\
&\mathcal{E}_\delta(Z)\geq t \delta^{-3}, \label{allEllHaveBigIntersection}\\
&|(v\cdot\nabla)^jP(z)|\leq K\delta,\quad j=1,2\quad \textrm{for each}\ z\in Z\ \textrm{and each}\ v\in \pi_S(\Phi(z)). \label{smallDotProd}
\end{align}
Then there is a set $\Sigma^\prime\subset\Sigma$ and a quadratic polynomial $Q$ so that
\begin{equation}
\mathcal{E}_{\delta}(\Sigma^\prime)\gtrsim_{D,E}\big(sw\kappa t/KL\big)^{O(1)}  \mathcal{E}_{\delta}(\Sigma),
\end{equation}
and for every $\ell^\prime\in\Sigma^\prime$, there is a line $\ell\subset Z(Q)$ with $\operatorname{dist}(\ell,\ell^\prime)\lesssim \big(sw\kappa t/KL\big)^{-O(1)}\delta$. 
\end{prop}
\begin{proof}
Since the constants $D$ and $E$ are fixed, all implicit constants may depend on these quantities; i.e.~we will write $\lesssim$ instead of $\lesssim_{D,E}$.

For each $\ell\in\Sigma$, let $T_{\ell}=N_{\delta}(\ell)\cap B(0,1)$ and define $Y(T)= N_{\delta}(\Gamma(\ell))$. Let $\tubes$ be a maximal essentially distinct subset of $\{T_{\ell}\colon \ell\in\Sigma\}$. Note that 
\begin{equation}\label{lotsOfTubes}
|\tubes|\sim\mathcal{E}_\delta(\Sigma)\geq L^{-1}\delta^{-3}.
\end{equation}

%
%For each $x\in\RR^4$, define $\tubes(x)=\{T\in\tubes\colon x\in Y(T)\}.$ 
By \eqref{WIs1SBroadAndWIs22Broad}, we have that for each $x\in \bigcup_{\tubes}Y(T)$,
\begin{equation}\label{manyTubesThroughEachPoint}
|\tubes(x)|\gtrsim s \delta^{-1}.
\end{equation}
By \eqref{ZIsCurved}, \eqref{WIs1SBroadAndWIs22Broad}, \eqref{smallDotProd}, and Lemma \ref{quadraticConeOfP}, we have that $v(\tubes(x))$ is contained in the $\lesssim  \delta (K/w\kappa)^{O(1)}$-neighborhood of the quadratic cone $\mathcal{C}_x$ of $P$ at $x$. By \eqref{WIs1SBroadAndWIs22Broad}, this cone is $\gtrsim w s^{O(1)}$-non-degenerate. Thus there exists a constant $A\lesssim (K/w\kappa)^{O(1)}$ so that for every $x\in Z$, we have that $v(\tubes(x))$ is contained in the $A\delta$-neighborhood of $\mathcal{C}_x$. In particular,
\begin{equation}\label{fewTubesThroughEachPoint}
|\tubes(x)|\lesssim A^{O(1)} \delta^{-1}.
\end{equation}
Since $t\delta\lesssim |N_{\delta}(Z)|\lesssim\delta$ (the lower bound comes from \eqref{allEllHaveBigIntersection} and the upper bound comes from the fact that $|N_{\delta}(Z)|\lesssim\delta$), we have 
\begin{equation}\label{integral}
st \lesssim \int_{N_{\delta}(Z)}\sum_{T\in\tubes}\chi_T(x)dx\lesssim A^{O(1)},
\end{equation}
and thus by \eqref{lotsOfTubes} and \eqref{integral},
\begin{equation}\label{sizeOfTubes}
L^{-1}\delta^{-3}\lesssim |\tubes|\lesssim A^{O(1)}\delta^{-3}. 
\end{equation}

By pigeonholing, we can select a point $x_0\in \bigcup Y(T)$ with
$$
\sum_{T\in \tubes(x_0)}|Y(T)|\gtrsim (st)^{O(1)} \delta^2.
$$

For each point $x\in \bigcup_{\tubes(x_0)} Y(T)$, define 
$$
N(x) = \Big|\big\{\tubes(x)\cap \bigcup_{T\in\tubes(x_0)} H(T)\big\}\Big|.
$$ 
$N(x)$ is an integer satisfying $0\leq N(x) \lesssim A^{O(1)}\delta^{-1}$. For each $T\in\tubes\backslash\tubes(x_0)$, define a new shading
$$
Y^\prime(\tube)=\{x\in Y(\tube)\colon N(x) \gtrsim (st)^{C}\delta^{-1}\}.
$$
If the constant $C$ is chosen sufficiently large, then 
$$
\sum_{T\in\tubes\backslash\tubes(x_0)}|Y^\prime(\tube)|\geq \frac{1}{2} \sum_{T\in\tubes}|Y(T)|.
$$
Thus by pigeonholing, we can select a set $\tubes^\prime\subset\tubes$ so that $|\tubes^\prime|\gtrsim(st/L)^{O(1)}|\tubes|$ and $|Y^\prime(T)|\gtrsim (st/L)^{C}|T|$ for all $T\in \tubes^\prime$.  Select a point $x_1$ with $\operatorname{dist}(x_0,x_1)\gtrsim (st/L)^{O(1)}$ and 
\begin{equation}\label{intersectionProperty}
|\tubes^\prime(x_1)|\gtrsim (st/L)^{O(1)}\delta^{-1}.
\end{equation}
For this value of $x_1$, if we select the constant $C\lesssim 1$ sufficiently large then there are $\gtrsim (st/L)^{O(1)}\delta^{-3}$ tubes $T\in\tubes$ that satisfy $|Y(T)|\geq (st/L)^C|T|$ and 
$$
\exists\ T_0\in \tubes(x_0),\ T_1 \in \tubes^\prime(x_1)\colon T\cap T_i\neq \emptyset,\ i=1,2,\ \dist(T\cap T_0,\ T\cap T_1)\geq C^{-1}(st/L)^C.
$$

Call this set of tubes $\tubes^{\prime\prime}$. Select a tube $T_0\in\tubes^{\prime\prime}$ with $|H(T_0)\cap \tubes^{\prime\prime}|\gtrsim (st/L)^{O(1)}\delta^{-2}$. Let $\mathcal{C}_{x_0}$ be the quadratic cone associated to $x_0$ and let $\tilde{\mathcal{C}}_{x_0} = x_0 + \operatorname{span}(\mathcal{C}_{x_0})$. Define $\tilde{\mathcal{C}}_{x_1}$ similarly, with $x_1$ in place of $x_0$. 

Equation \eqref{intersectionProperty} implies that 
$$
|N_{A\delta}(\tilde{\mathcal{C}}_{x_0}) \cap N_{A\delta}(\tilde{\mathcal{C}}_{x_1})\big|\gtrsim (st/L)^{O(1)}\delta^{3}. 
$$
Since the cones $\tilde{\mathcal{C}}_{x_0}$ are $\gtrsim A$ non-degenerate and their vertices are $\gtrsim (sc)^C$-separated, the set $N_{A\delta}(\tilde{\mathcal{C}}_{x_0}) \cap N_{A\delta}(\tilde{\mathcal{C}}_{x_1})$ is contained in the $\lesssim (AL/st)^{O(1)}\delta$--neighborhood of a curve. However, it need not be the case that the cones $\tilde{\mathcal{C}}_{x_0}$ and $\tilde{\mathcal{C}}_{x_1}$ themselves intersect. To overcome this annoying technicality, we will replace $\tilde{\mathcal{C}}_{x_1}$ by a different cone that is comparable to $\tilde{\mathcal{C}}_{x_1}$ but which does intersect $\tilde{\mathcal{C}}_{x_0}$ in a curve. We will call this cone $\tilde{\mathcal{C}}^*$; we will describe its construction in the next paragraph. 

By our choice of $x_1$, there exist three points $p_1,p_2,p_3\in \tilde{\mathcal{C}}_{x_0}\cap N_{A\delta}(\tilde{\mathcal{C}}_{x_1})$ so that all $3\times 3$ minors of $[p_1-x_1,\ p_2-x_1,\ p_3-x_3]$ have magnitude $\gtrsim (st/AL)^{O(1)}$. Let $H$ be the hyperplane passing through $x_1,p_1,p_2,p_3$ (our condition on the minors of $[p_1-x_1,\ p_2-x_1,\ p_3-x_3]$ ensures that this hyperplane is ``well conditioned'' in the sense that a small perturbation to one of the points $p_1,p_2,$ or $p_3$ will only cause a small change in the choice of hyperplane). Since $p_1,p_2,p_3\in  N_{A\delta}(\tilde{\mathcal{C}}_{x_1})$, and $\tilde{\mathcal{C}}_{x_1}\subset T_{x_1}(Z)$, the condition on the minors of $[p_1-x_1,\ p_2-x_1,\ p_3-x_3]$ implies that
\begin{equation}\label{angleHTx1}
\angle(H, T_{x_1}(Z))\lesssim (L/st)^{O(1)}A\delta.
\end{equation}
Define $\mathcal{C}^* = H \cap Z(P_1)$, where $P_1$ is the homogeneous polynomial of degree 2 arising from the Taylor expansion of $P$ around $x_1$. Since $H$ and $Z(P_1)$ intersect $\geq \kappa$ transversely (i.e. the tangent plane of $Z(P_1)$ and of $H$ make an angle $\geq \kappa$ at every point of intersection), \eqref{angleHTx1} implies that $\tilde{\mathcal{C}}_{x_1}$ and $\mathcal{C}^*$ are comparable in the sense that 
$$
B(0,1)\cap \tilde{\mathcal{C}}_{x_1} \subset  N_{(AL/(st\kappa))^{O(1)}\delta}(\mathcal{C}^*),\quad\textrm{and}\quad
B(0,1)\cap \mathcal{C}^*\subset N_{(AL/(st\kappa))^{O(1)}\delta}(\tilde{\mathcal{C}}_{x_1}).
$$
We also have that $\tilde{\mathcal{C}}_{x_0} \cap \mathcal{C}^*$ is a degree-two curve lying in the plane $T_{x_0}Z \cap H$. 

Let $\ell$ be a line with $\ell\cap B(0,1)\subset T_0$ so that $\ell$ intersects $\tilde{\mathcal{C}}_{x_0}$ and $\mathcal{C}^*$ at points that are $\gtrsim (st\kappa/AL)^{O(1)} $ separated. Observe that the cones $\tilde{\mathcal{C}}_{x_0}$ and $\mathcal{C}^*$ intersect in a one-dimensional degree-two curve, and the line $\ell$ intersects each of $\tilde{\mathcal{C}}_{x_0}$ and $\mathcal{C}^*$ at distinct points that are not on this curve.

We can now use the ``14 point'' argument from \cite{KZ} to find a monic polynomial $Q$ that vanishes on $\tilde{\mathcal{C}}_{x_0}, \mathcal{C}^*,$ and $\ell$. In brief, select 5 points $p_1,\ldots,p_5\in \tilde{\mathcal{C}}_{x_0}\cap \mathcal{C}^*$; any polynomial of degree $\leq 2$ that vanishes on $p_1,\ldots,p_5$ must vanish on the degree-two plane curve $\tilde{\mathcal{C}}_{x_0}\cap \mathcal{C}^*$. Let $p_6,p_7$ be the points of intersection of $\tilde{\mathcal{C}}_{x_0} \cap \ell$ and $\mathcal{C}^*\cap \ell$, and let $p_8$ be another point on $\ell$; any polynomial of degree $\leq 2$ that vanishes on $p_6,p_7,p_8$ must vanish on $\ell$. Let $p_9=x_0$ and let $p_{10}=x_1$. Let $p_{11}$ and $p_{12}$ be two points on $\tilde{\mathcal{C}}_{x_0} $, and let $p_{13}$ and $p_{14}$ be two points on $\mathcal{C}^*$. See Figure \ref{14ptFig}. Let $Q$ be a polynomial of degree $\leq 2$ that vanishes on $p_1,\ldots,p_{14}$. We can choose $Q$ so that its largest coefficient has magnitude 1. $Q$ will be the output from this proposition. The remainder of the proof is devoted to finding the set $\Sigma^\prime$ so that $Q$ and $\Sigma^\prime$ satisfy the conclusions of the proposition.

\begin{figure}[h!]
 \centering
\begin{overpic}[width=0.8\textwidth]{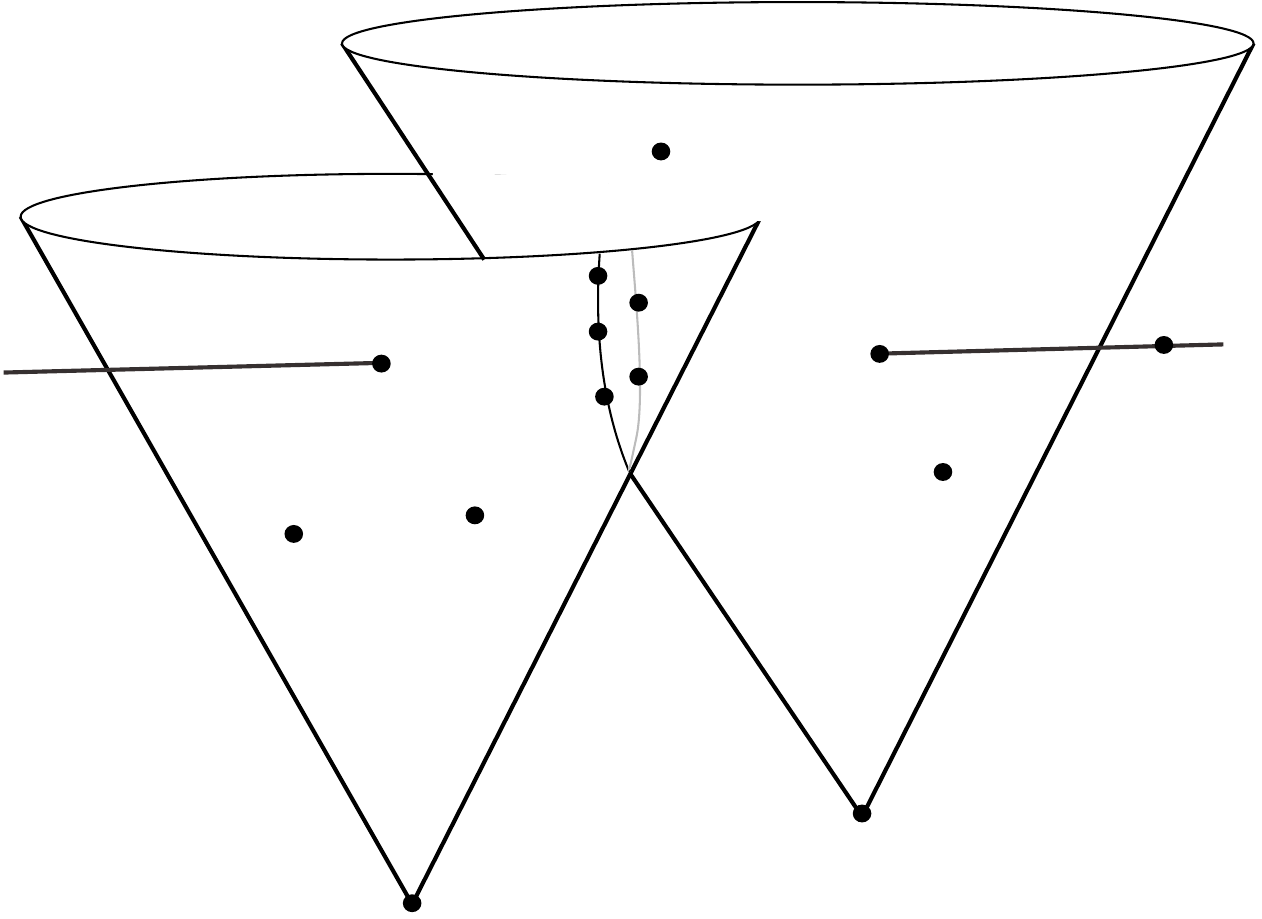}
\put (43,50) {$p_1$}
\put (43,46) {$p_2$}
\put (43,41) {$p_3$}
\put (52,49) {$p_4$}
\put (51,44) {$p_5$}
\put (31, 42) {$p_6$}
\put (65, 43) {$p_7$}
\put (92, 42) {$p_8$}
\put (28, 0) {$p_9$}
\put (63,6) {$p_{10}$}
\put (22,27) {$p_{11}$}
\put (36,29) {$p_{12}$}
\put (73,32) {$p_{13}$}
\put (54,60) {$p_{14}$}
\end{overpic}
\caption{The cones $\tilde{\mathcal{C}}_{x_0}$ (left), $\mathcal{C}^*$ (right), the line $\ell$, and the $14$ points $p_1,\ldots,p_{14}$.} \label{14ptFig}
\end{figure}
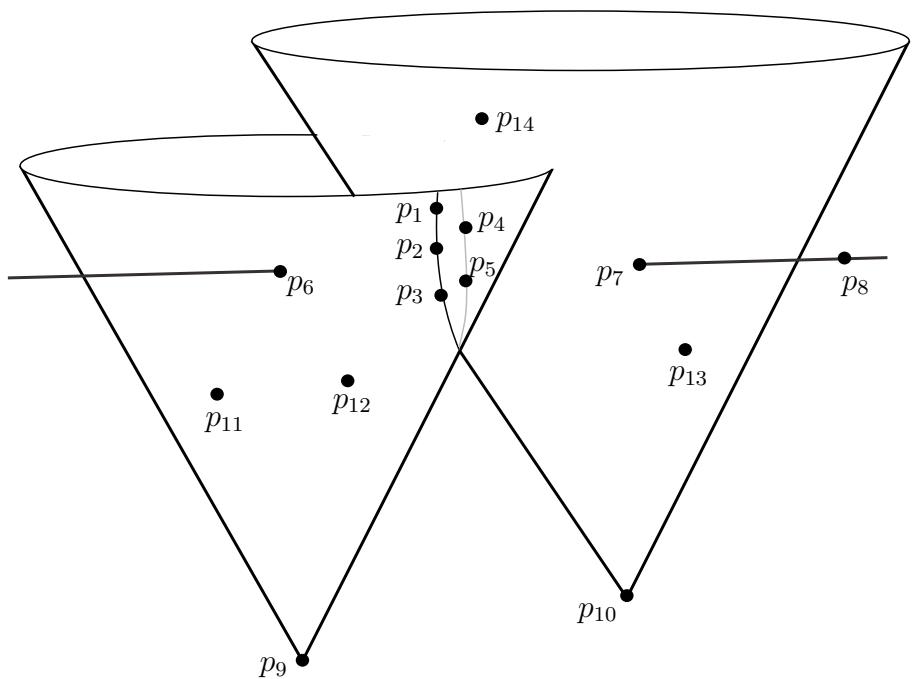

Let $\ell_1$ be the line passing through $p_9$ and $p_{11}$; this is a line in $\tilde{\mathcal{C}}_{x_0}$ passing through the vertex $p_9 = x_0$, so it intersects the curve $\tilde{\mathcal{C}}_{x_0}  \cap \mathcal{C}^*$ at some point $x$. Since $Q$ vanishes at the three collinear points $p_9,p_{11},$ and $x$, $Q$ must vanish on the entire line $\ell_1$. Similarly, $Q$ vanishes on the line $\ell_2$ passing through $p_9$ and $p_{12}$, and the line $\ell_3$ passing through $p_9$ and $p_6$. Thus $Q$ vanishes on the five-dimensional (reducible) curve $\tilde{\mathcal{C}}_{x_0} \cap \mathcal{C}^* \cup \ell_1 \cup \ell_2 \cup \ell_3$. Since $Q$ has degree at most 2 and $\tilde{\mathcal{C}}_{x_0} $ has degree at most 2, we conclude that $Q$ vanishes on $\tilde{\mathcal{C}}_{x_0} $. An identical argument shows that $Q$ vanishes on $\mathcal{C}^*$.

Recall that for each $T\in H(T_0)\cap \tubes^{\prime\prime}$, we have that $Z(Q)$ vanishes on $\gtrsim (st\kappa/AL)^{O(1)}\delta^{-1}$ distinct $\delta$-separated points on $T$. Since $Q$ is monic and has degree 2, we have 

$$
|Q(x)|\lesssim (AL/st\kappa)^{O(1)}\delta\quad\textrm{for all}\ x\in T.
$$ 
By the definition of $\tubes^{\prime\prime}$, we have
$$
\sum_{T\in H(T_0)\cap \tubes^{\prime\prime} }|H_Y(T)|\gtrsim (st\kappa/AL)^{O(1)} \delta^{-4}.
$$
Thus there exists a set $\tubes^{\prime\prime\prime}$ with $|\tubes^{\prime\prime\prime}|\gtrsim (st\kappa/AL)^{O(1)} \delta^{-3}$ and 
$$
\Big| Y(T^\prime)\cap \bigcup_{T\in H(T_0)} Y(T)\Big|\gtrsim (st\kappa/AL)^{O(1)} |T^\prime|\quad\textrm{for every}\ T^\prime\in \tubes^{\prime\prime\prime}.
$$

Since $Q$ is monic and $Z(Q)\cap B(0,1)\neq\emptyset$, we can assume that at least one non-constant term of $Q$ has size $\sim 1$. We can also assume that at least one degree-two term of $Q$ has magnitude $\gtrsim (\kappa st/AL)^{O(1)}$; if this were not the case, then $Z(Q)\cap B(0,1)$ would be contained in the $\sim (\kappa st/AL)^{O(1)}$-neighborhood of a hyperplane $H$, and thus
$$
|N_{\delta}(Z)\cap N_{(\kappa st/AL)^{O(1)}}(H)|\geq \bigcup_{T\in H(T_0)} Y(T)\gtrsim (\kappa st/AL)^{O(1)}\delta,
$$
but this would contradict the estimate
$$
|N_{\delta}(Z)\cap N_{(\kappa st/AL)^{O(1)}}(H)|\lesssim \delta^{3/2}(\kappa st/AL)^{-O(1)}
$$
coming from Lemma \ref{largeFundFormEscapesPlane}. Thus at least one degree-two term of $Q$ must have magnitude $\gtrsim (\kappa st/AL)^{O(1)}$. Next, the set
$$
\{x\in B(0,1)\colon |\nabla Q(x)|\leq (\kappa st/AL)^{O(1)}\}
$$
is contained in the $(\kappa st/AL)^{O(1)}$--neighborhood of a hyperplane $H^\prime$. By the same argument as above, we can choose a refinement $\tubes^{(\mathrm{iv})}\subset\tubes^{\prime\prime\prime}$ with 
\begin{equation}\label{sizeOfTubesiv}
|\tubes^{(\mathrm{iv})}|\gtrsim (\kappa st/AL)^{O(1)} \delta^{-3}
\end{equation}
and a shading $Y^{(\mathrm{iv})}(T)$ so that $|\nabla Q(x)|\gtrsim (\kappa sc/A)^{O(1)}$ for all $x\in Y^{(\mathrm{iv})}(T)$ and all $T\in\tubes^{(\mathrm{iv})}$. 

Again by pigeonholing, we can refine $Y^{(\mathrm{iv})}$ to get a shading $Y^{(\mathrm{v})}$ so that $|\tubes_{Y^{(\mathrm{v})}}(x)|\gtrsim (\kappa st/AL)^{O(1)}\delta^{-1}$ for all $x\in \bigcup_{T\in\tubes^{(\mathrm{iv})}}Y^{(\mathrm{v})}(T)$. Now fix a point $x\in \bigcup_{T\in\tubes^{(\mathrm{iv})}}Y^{(\mathrm{v})}(T)$. We will show that $T_x(Z(Q))\cap Z(Q)$ is a $\zeta$--non-degenerate cone, where $\zeta = (AL/(st\kappa))^{O(1)}$. Indeed, $v(\tubes^{(\mathrm{iv})}(x))$ is contained in the $A\delta$--neighborhood of the $w$--non-degenerate cone $\mathcal{C}_x$, and $|\tubes^{(\mathrm{iv})}(x)|\gtrsim (st\kappa/AL)^{O(1)}\delta^{-1}$. At most $(\zeta/w)^{1/2}A\delta$ $\delta$--separated vectors can be contained in the intersection of $N_{A}(\mathcal{C}_x)$ with the $\zeta$--neighborhood of a plane. We conclude that the cone $T_x(Z(Q))\cap Z(Q)$ is $\zeta$--non-degenerate for some $\zeta =  (s ct\kappa/A)^{O(1)}$.

We conclude that if $T\in\tubes^{(\mathrm{iv})}$ and $x\in Y^{(\mathrm{v})}(T)$, then $v(T)$ makes an angle $\lesssim (AL/(st\kappa))^{O(1)} \delta$ with the quadratic cone $T_x(Z(Q))\cap Z(Q)$ of $Q$ at $x$. However, since $Q$ is degree-two, if $v$ is a vector contained in the quadratic cone of $Q$ at $x$, then the line $\{x + v t\colon t\in\RR\}$ is contained in $Z(Q)$. Thus if $T\in\tubes^{(\mathrm{iv})}$ with $Y^{(\mathrm{v})}(T)\neq\emptyset$, then there is a line $\ell$ contained in $Z(Q)$ with $\ell\cap B(0,1)\subset N_{(AL/(st\kappa))^{O(1)}\delta}(T)$. 

By \eqref{sizeOfTubes} and \eqref{sizeOfTubesiv}, we have that 
\begin{equation*}
\begin{split}
|\tubes^{(\mathrm{iv})}|&\geq (st\kappa/AL)^{O(1)}|\tubes|\\
&\gtrsim \big(swt\kappa/KL\big)^{O(1)} |\tubes|\\
&\gtrsim \big(swt\kappa/KL\big)^{O(1)} \mathcal{E}_\delta(\Sigma).
\end{split}
\end{equation*}
Thus there is a set $\Sigma^\prime\subset\Sigma$ (note that $\Sigma^\prime$ need not be semi-algebraic) with 
$$
\mathcal{E}_\delta(\Sigma^\prime)\gtrsim\big(swt\kappa/KL\big)^{O(1)}\mathcal{E}_\delta(\Sigma)
$$ 
so that for all $\ell^\prime\in\Sigma^\prime$ there is a line $\ell$ contained in $Z(Q)$ with 
$$
\operatorname{dist}(\ell,\ell^\prime)\lesssim \big(swt\kappa/KL\big)^{O(1)}\delta.\qedhere
$$
\end{proof}

\section{Proof of Theorem \ref{discretizedSeveri}}\label{mainProofSection}
The following result allows us to separate the lines in $\Sigma_{\delta}(Z)$ into two sets---those that can be covered by a small number of one and two-dimensional rectangular prisms, and those that are amenable to Proposition \ref{discretizedSegreNonDegenProp}. 
\begin{prop}\label{partitioningResult}
Let $\delta,s,u,c,\kappa$ be positive real numbers. Let $P$ be a polynomial of degree at most $D$. Let 
\begin{equation}\label{ZIsCurved}
Z \subset \{x\in Z(P)\cap B(0,1),\ 1\leq|\nabla P(x)| \leq 2,\ \Vert II(x)\Vert_{\infty}\geq \kappa\}.
\end{equation}
Let $\Sigma\subset\Sigma_{\delta,c}(Z)$ and $\Gamma\subset \Gamma(N_{\delta}(Z), \Sigma)$. Suppose that $Z,\Sigma,$ and $\Gamma$ are semi-algebraic of complexity at most $E$. Let $\Phi\subset Z\times S^3$ be associated to $\Gamma$, in the sense of Definition \ref{defnAssociated}. Suppose that
\begin{align}
&|\Gamma(\ell)|\geq c,\quad\textrm{for all}\ \ell\in\Sigma,\\
&|(v\cdot\nabla)^jP(z)|\leq K\delta,\quad j=1,2\quad \textrm{for each}\ z\in Z\ \textrm{and each}\ v\in \pi_S(\Phi(z)).\label{nablaCondition}
\end{align}

Then there is a number $w\gtrsim_{D,E}(us\kappa c)^{O(1)}$ and sets $\Sigma^\prime,\Sigma^{\prime\prime},\Sigma^{\prime\prime\prime}$, $Z^{\prime\prime\prime}$ and $\Gamma^{\prime\prime\prime}$ with $\Sigma=\Sigma^\prime\cup\Sigma^{\prime\prime}\cup\Sigma^{\prime\prime\prime},\ \Sigma^{\prime\prime\prime}\subset \Sigma_{\delta, c/27}(Z^{\prime\prime\prime}),$ and $\Gamma^{\prime\prime\prime}\subset \Gamma(N_{\delta}(Z^{\prime\prime\prime}), \Sigma^{\prime\prime\prime}),$ so that
\begin{itemize}
\item The lines in $\Sigma^{\prime}$ can be covered by $O_{D,E}(c^{-O(1)}s^{-2})$ rectangular prisms of dimensions $2\times s\times s\times s$.
\item The lines in $\Sigma^{\prime\prime}$ can be covered by $O_{D,E}((sc)^{-O(1)}u^{-1})$ rectangular prisms of dimensions $2\times 2\times u\times u$.
\item $\Sigma^{\prime\prime\prime}$, $Z^{\prime\prime\prime}$ and $\Gamma^{\prime\prime\prime}$ are semi-algebraic of complexity $O_{D,E}(1)$. We have 
\begin{equation}\label{SigmaPPPNotTooBig}
|N_{\delta}(\Sigma^{\prime\prime\prime})|\lesssim_{D} (K/(us\kappa c))^{O(1)}\delta^{-3}.
\end{equation} 
Finally, let $\Phi^{\prime\prime\prime}$ be the set associated to $\Gamma^{\prime\prime\prime}.$ Then 

\begin{align}
&Z^{\prime\prime\prime} \subset 1\operatorname{-SBroad}_{s}(\Phi^{\prime\prime\prime}) \cap (2,2)\operatorname{-Broad}_{w}(\Phi^{\prime\prime\prime}),\label{ZpppIsSBroadAndBroad}\\
&\mathcal{E}_\delta(Z^{\prime\prime\prime})\gtrsim (scu\kappa)^{O(1)}\delta^{-3}.\label{ZpppHasBigEntropy}
\end{align}
\end{itemize}
\end{prop}
\begin{proof}
Let $w\gtrsim_{D,E}(us\kappa c)^{O(1)}$ be the constant from Proposition \ref{22NarrowCovered} associated to the values $u,s,\kappa,c,D,E$. Define $\Sigma_1 = \Sigma, \Gamma_1=\Gamma$, and $\Phi_1 = \Phi.$ For each $i=1,2,3$, inductively define
\begin{align*}
&\Phi_i^\prime = \{(z,v)\in\Phi_i\colon z \not\in 1\operatorname{-SBroad}_{s}(\Phi_i) \},\\
&\Phi_i^{\prime\prime} = \{(z,v)\in\Phi_i\colon z \in (2,2)\operatorname{-Narrow}_{w}(\Phi_i) \},\\
&\Phi_i^{\prime\prime\prime}=\Phi_i\backslash(\Phi_i^\prime\cup\Phi_i^{\prime\prime}).
\end{align*}
Let $\Gamma_i^\prime$, $\Gamma_i^{\prime\prime},$ and $\Gamma_i^{\prime\prime\prime}$ be the pre-images of $\Phi_i^\prime$, $\Phi_i^{\prime\prime}$, and $\Phi_i^{\prime\prime\prime}$ (respectively) under the map $\Gamma_i\to\Phi_i$.  Define
\begin{align*}
\Sigma_i^\prime & = \{\ell\in\Sigma_i\colon |\Gamma_i^\prime(\ell)|\geq 3^{-i}c\},\\
\Sigma_i^{\prime\prime} & = \{\ell\in\Sigma_i\colon |\Gamma_i^{\prime\prime}(\ell)|\geq 3^{-i}c\},\\
\Sigma_i^{\prime\prime\prime} & = \{\ell\in\Sigma_i\colon |\Gamma_i^{\prime\prime\prime}(\ell)|\geq 3^{-i}c\},
\end{align*}
so $\Sigma_i\subset\Sigma_i^\prime\cup\Sigma_i^{\prime\prime}\cup\Sigma_i^{\prime\prime\prime}$.

By Proposition \ref{narrowVarietiesProp}, the lines in $\Sigma_i^\prime$ can be covered by $O_{D,E}(c^{-O(1)}s^{-2})$ rectangular prisms of dimensions $2\times s\times s\times s$; these lines will be placed in $\Sigma^{\prime}$. By Proposition \ref{22NarrowCovered}, the lines in $\Sigma_i^{\prime\prime}$ can be partitioned into two sets; the first set can be covered by $O_{D,E}(c^{-O(1)}s^{-2})$ rectangular prisms of dimensions $2\times s\times s\times s$ and the second set can be covered by $O_{D,E}((cs)^{-O(1)}u^{-1})$ rectangular prisms of dimensions $2\times 2\times u\times u$. We will place these lines into $\Sigma^{\prime}$ and $\Sigma^{\prime\prime}$, respectively.  

Define 
$$
\Sigma_{i+1} = \Sigma_i^{\prime\prime\prime}
$$
and define
$$
\Gamma_{i+1}=\Gamma_i^{\prime\prime\prime}.
$$ 
Define $Z_{i+1}$ to be the image of $\Phi_i$ under the map $(z,v)\mapsto z$. Then we have
\begin{itemize}
\item $Z_{i+1}\subset Z_i$.
\item $\Sigma_{i+1}\subset \Sigma_{\delta, 3^{-i}c}(Z_i)$.
\item  $\Sigma_{i+1}\subset\Sigma_i$.
\item $\Gamma_{i+1}\subset\Gamma_i$.
\item \itemizeEqnVSpacing
\begin{equation}\label{Zip1InSBroadBroad}
Z_{i+1}\subset 1\operatorname{-SBroad}_{s}(\Phi_i)\cap(2,2)\operatorname{-Broad}_{w}(\Phi_i).
\end{equation}
\item \itemizeEqnVSpacing
\begin{equation}\label{GammaiIsBig}
|\Gamma_i(\ell)|\geq 3^{-i}c\geq c/27.
\end{equation}
\item By \eqref{nablaCondition},
\begin{equation}\label{dotProdCondition}
|(v\cdot\nabla)^jP(z)|\leq K\delta,\ j=1,2,\ \textrm{for each}\ z\in Z_i\ \textrm{and each}\ v\in \pi_S(\Phi_i(z)).
\end{equation}
\end{itemize} 
%The final item comes from 

If $\Sigma_3 =\emptyset$ then define $\Sigma^{\prime\prime\prime}=\emptyset,$ $Z^{\prime\prime\prime}=\emptyset$ and $\Gamma^{\prime\prime\prime}=\emptyset$ and we are done. If not, define $Z^{\prime\prime\prime} = Z_2$, $\Sigma^{\prime\prime\prime} = \Sigma_1,$ and $\Gamma^{\prime\prime\prime} = \Gamma_2$. \eqref{ZpppIsSBroadAndBroad} follows from \eqref{Zip1InSBroadBroad}. Applying Lemma \ref{hairbrushMakesBigEntropy} to the sets $Z_3\subset Z_2$; $\Sigma_3\subset\Sigma_2;$ $\Gamma_3\subset\Gamma_2$ (equations \eqref{Zip1InSBroadBroad}, \eqref{GammaiIsBig} and \eqref{dotProdCondition} guarantee that the hypotheses of Lemma \ref{hairbrushMakesBigEntropy} are met), we conclude that 
$$
\mathcal{E}_\delta(Z^{\prime\prime\prime})=\mathcal{E}_\delta(Z_2)\gtrsim (scw\kappa)^{O(1)}\delta^{-3},
$$
which is \eqref{ZpppHasBigEntropy}. 

Finally, it remains to prove \eqref{SigmaPPPNotTooBig}. But this follows from the observation that $|N_{\delta}(Z)|\lesssim_D \delta^{-3}$, and that for each $z\in Z$,  
$$
N_{\delta}( \Phi^{\prime\prime\prime}(N_{\delta}(z))\lesssim_{D,E}(K/(us\kappa c))^{O(1)}\delta^{-1}.\qedhere
$$
\end{proof}
We are now ready to prove Theorem \ref{discretizedSeveri}. First, we will state a slightly more technical version that will be useful for applications.
\begin{techDiscretizedSeveriThm}
Let $P\in\RR[x_1,x_2,x_3,x_4]$ be a polynomial of degree $D$, and let $Z = Z(P)\cap B(0,1)$. Let $\delta,\kappa,u,s\in (0,1)$ be numbers satisfying $0<\delta<  u < s<1$ and $\delta<\kappa<1$ (if these conditions are not satisfied the theorem is still true, but it has no content). 

Define
$$
\Sigma = \{\ell\in\lines\colon |\ell\cap N_{\delta}(Z)|\geq 1 \},
$$
and let $\Sigma^\prime\subset\Sigma$ be a semi-algebraic set of complexity at most $E$. Then we can write $\Sigma^\prime = \Sigma_1\cup\Sigma_2\cup\Sigma_3\cup\Sigma_4$, where

\begin{itemize}
\item There is a collection of $O_{D,E}\big(|\log\delta|^{O(1)}s^{-2}\big)$ rectangular prisms of dimensions $2\times s\times s \times s$ so that every line from $\Sigma_1$ is covered by one of these prisms.
\item There is a collection of $O_{D,E}\big((|\log\delta|/s)^{O(1)}u^{-1}\big)$ rectangular prisms of dimensions $2\times 2\times u \times u$ so that every line from $\Sigma_2$ is covered by one of these prisms.
\item There is a collection of $O_{D,E}\big(|\log\delta|^{O(1)}\big)$ rectangular prisms of dimensions $2 \times 2 \times 2 \times \kappa$ so that every line in $\Sigma_3$ is covered by one of these prisms.
\item There is a set $\Sigma_4^\prime\subset\Sigma_4$ with 
$$
\mathcal{E}_\delta(\Sigma_4^\prime)\gtrsim_{D,E}\big(us\kappa/|\log\delta|\big)^{O(1)}\mathcal{E}_\delta(\Sigma_4)
$$ 
and a quadratic hypersurface $Q$ so that for every line $\ell^\prime\in \Sigma_4^\prime$, there is a line $\ell$ contained in $Z(Q)$ with 
$$
\operatorname{dist}(\ell,\ell^\prime)\lesssim_D \big(|\log\delta|/(us\kappa)\big)^{O(1)}\delta.
$$
\end{itemize}
\end{techDiscretizedSeveriThm}
\begin{proof}
Apply Proposition \ref{goodPolyProp} to $\Sigma$ and $P$. Let $P_1,\ldots,P_b,\ \Sigma_1,\ldots,\Sigma_b,$ and $\Gamma_1,\ldots,\Gamma_b,$ $b = O_{D,E}(|\log\delta|)$ be the output from the proposition. By Item \ref{GammaJCovers} from Proposition \ref{goodPolyProp}, there exists a number $c \gtrsim_{D}|\log\delta|^{-1}$ so that $|\Gamma_{j}(\ell)|\geq c$ for each $j=1,\ldots,b$ and each $\ell\in\Sigma_j$. 

For each index $j$, define
$$
Z_j = \{ x\in Z(P_j)\cap B(0,1)\colon 1\leq |\nabla P_j(x)|\leq 2\}.
$$
By Item \ref{GammaJInZJ} from Proposition \ref{goodPolyProp}, we have that for all $\ell\in\Sigma_j$ and all $x\in\Gamma_j(\ell)$, $x\in N_{\delta}(Z_j)$. Let $\Phi_j\subset Z_j\times S^3$ be associated to $\Gamma_j$, in the sense of Definition \ref{defnAssociated}. We have that for each $z\in Z_j$ and each $v\in \Pi_S(\Phi_j(z))$,  
\begin{equation}\label{smallDirectionalDerivativeConclusion}
|(v(\ell)\cdot \nabla)^i P_j(z)| \lesssim_{D} |\log\delta|^2 \delta,\ i=1,2.
\end{equation}

Define
\begin{align*}
Z_j^\prime &= \{z\in Z_j\colon \Vert II(z)\Vert_{\infty}\leq \kappa \},\\
Z_j^{\prime\prime} &= \{z\in Z_j\colon \Vert II(z)\Vert_{\infty}> \kappa \}.
\end{align*}

Define
\begin{equation*}
\begin{split}
\Phi_j^\prime&=\{(z,v)\in \Phi_j\colon z\in Z_j^\prime\},\\
\Phi_j^{\prime\prime}&=\{(z,v)\in \Phi_j\colon z\in Z_j^{\prime\prime}\}.
\end{split}
\end{equation*}

Let $\Gamma_j^\prime$ and $\Gamma_j^{\prime\prime}$ be the pre-images of $\Phi_j^\prime$ and $\Phi_j^{\prime\prime}$, respectively, under the map from $\Phi_j\to\Gamma_j$. Define
\begin{align*}
\Sigma_j^\prime &= \{\ell\in\Sigma_j \colon |\Gamma_j^\prime(\ell)|\geq c/2\},\\
\Sigma_j^{\prime\prime} &= \{\ell\in\Sigma_j \colon |\Gamma_j^{\prime\prime}(\ell)|\geq c/2\}.
\end{align*}

Apply Proposition \ref{FlatVarietiesNearAPlaneProp} to $P_j,\ \Sigma_j^{\prime}$ and $\Gamma_j^{\prime}$. Define $\Sigma_j^{(1)}=\Sigma_j^{\prime}.$ We conclude that the lines in $\Sigma_j^{(1)}$ can be covered by $\lesssim_{D,E} (|\log\delta|/c)^{O(1)}\lesssim_{D,E}|\log\delta|^{O(1)}$ rectangular prisms of dimensions $2\times 2\times 2\times \kappa$.

Apply Proposition \ref{partitioningResult} to $P_j$, $Z_j^{\prime\prime}$, $\Sigma_j^{\prime\prime}$, and $\Gamma_j^{\prime\prime}$, with the parameters $\delta,s,u,c/2,$ and $\kappa$. We obtain a number $w\gtrsim_{D,E}(us\kappa c)^{O(1)} \gtrsim_{D,E}(us\kappa /|\log\delta|)^{O(1)}$; sets of lines $\Sigma_j^{(2)}$, $\Sigma_j^{(3)}$, and $\Sigma_j^{(4)}$; and sets $Z_j^{(4)}$ and $\Gamma_j^{(4)}$ so that $\Sigma_j^{\prime\prime}=\Sigma_j^{(2)}\cup \Sigma_j^{(3)}\cup \Sigma_j^{(4)},$ and
\begin{itemize}
\item The lines in $\Sigma_j^{(2)}$ can be covered by $O_{D,E}(c^{-O(1)}s^{-2})=O_{D,E}(|\log\delta|^{O(1)}s^{-2})$ rectangular prisms of dimensions $2\times s\times s\times s$. 
\item The lines in $\Sigma_j^{(3)}$ can be covered by $O_{D,E}((sc)^{-O(1)}u^{-1})=O_{D,E}((|\log\delta|/s)^{O(1)}u^{-1})$ rectangular prisms of dimensions $2\times 2\times u\times u$.
\item $\Sigma_j^{(4)}$, $Z_j^{(4)}$ and $\Gamma_j^{(4)}$ are semi-algebraic of complexity $O_{D,E}(1)$. If $\Phi_j^{(4)}$ is the set associated to $\Gamma_j^{(4)},$ then 
\begin{align}
&Z_j^{(4)} \subset 1\operatorname{-SBroad}_{s}(\Phi_j^{(4)}) \cap (2,2)\operatorname{-Broad}_{w}(\Phi_j^{(4)}),\\
&\mathcal{E}_\delta(Z_j^{(4)})\gtrsim_{D,E} (scu\kappa)^{O(1)}\delta^{-3}\gtrsim_{D,E} (su\kappa/|\log\delta|)^{O(1)}\delta^{-3}.
\end{align}
\end{itemize}

The sets $Z_j^{(4)}$, $\Sigma_j^{(4)}$, and $\Gamma_j^{(4)}$ satisfy the hypotheses of Proposition \ref{discretizedSegreNonDegenProp}. Thus there exists a set $(\Sigma_j^{(4)})^\prime\subset \Sigma_j^{(4)}$ and a quadratic polynomial $Q_j$ so that
\begin{equation}
\mathcal{E}_{\delta}\big((\Sigma_j^{(4)})^\prime\big)\gtrsim_{D,E}\big(su\kappa/|\log\delta|\big)^{O(1)}  \mathcal{E}_{\delta}\big(\Sigma_j^{(4)}\big),
\end{equation}
and for every $\ell^\prime\in(\Sigma_j^{(4)})^\prime$, there is a line $\ell\subset Z(Q_j)$ with 
$$
\operatorname{dist}(\ell,\ell^\prime)\lesssim \big(su\kappa/|\log\delta|\big)^{-O(1)}\delta.
$$

For $i=1,2,3,4$, define  
$$
\Sigma_i  = \bigcup_{j=1}^b\Sigma_j^{(i)}.
$$
%Recall that the union of the lines in $\Sigma_1,\ldots,\Sigma_b$ need not contain all (or indeed any) of the lines in $\Sigma$. Instead, each line in $\Sigma$ is $\delta$-close to some line from one of the sets $\Sigma_1,\ldots,\Sigma_b$. As a final step, we will take care of this technicality. For $i=1,2,3,4$, define $\Sigma_i$ to be the set of lines $\ell\in\Sigma$ so that $\operatorname{dist}(\ell,\ell^*)\leq\delta$ for some $\ell^*\in\Sigma_i^*$. 
Then $\Sigma=\Sigma_1\cup\Sigma_2\cup\Sigma_3\cup\Sigma_4$, and the sets $\Sigma_1,\Sigma_2,\Sigma_3,$ and $\Sigma_4$ satisfy the conclusions of Theorem \ref{discretizedSeveri} (the set $\Sigma_4^\prime$ is the set of the form $(\Sigma_j^{(4)})^\prime$ that maximizes $\mathcal{E}_\delta\big((\Sigma_j^{(4)})^\prime\big)\ $).
\end{proof}

\section{Proof of Corollary \ref{mainThm}}\label{proofOfMainCorSection}
Corollary \ref{mainThm} will be proved by combining induction on scale and re-scaling arguments. First, we will state a variant of Corollary \ref{mainThm} that is more amenable to induction. 

\begin{prop}\label{mainProp}
For each $D,\eps>0$, there exists a constant $C_{D,\eps}$ so that the following holds for all $0<\delta\leq 1$.

Let $P\in\RR[x_1,\ldots,x_4]$ be a polynomial of degree at most $D$ and let $Z= Z(P)\cap B(0,1)$.
Then
$$
\mathcal{E}_{\delta}\big(v(\Sigma_{\delta}(Z))\big) \big) \leq C_{D,\eps}\delta^{-2-\eps}.
$$
\end{prop}

\subsection{Re-scaling arguments}
\begin{lem}\label{rescaleT}
Fix $D$ and $\eps>0$ and let $\delta>0$. Suppose that Proposition \ref{mainProp} holds for all values of $\delta^\prime$ with $\delta<\delta^\prime\leq 1$, and let $C_{D,\eps}$ be the associated constant. Let $P$ be a polynomial of degree at most $D$, and let $Z = Z(P)\cap B(0,1)$. Let $R$ be a rectangular prism (of arbitrary orientation) that has $1\leq d\leq 3$ ``long'' directions and $4-d$ ``short'' directions; suppose that $R$ has length $2$ in the long directions and length $t$ in the short directions (i.e. inside $B(0,1)$,  $R$ is comparable to the $t$--neighborhood of a $d$-dimensional affine hyperplane). Then 
\begin{equation}\label{rescaledDirections}
\mathcal{E}_\delta\big( v(\Sigma_\delta(Z\cap R))\big) \leq C\cdot C_{D,\eps} t^{3-d+\eps}\delta^{-2-\eps},
\end{equation}
where $C$ is an absolute constant.
\end{lem}
\begin{proof}
Let $H$ be a $d$-dimensional hyperplane with $N_{t}(H)\cap B(0,1)$ comparable to $R\cap B(0,1)$. After applying a translation, we can assume that $H$ contains the origin. First, we can assume $t\leq 1/100$, or the theorem is trivial. Second, we can assume that $H$ makes an angle $\leq 1/5$ with the vector $e_1$, since otherwise $\Sigma_\delta(Z\cap N_{t}(H))$ is empty. Apply a rotation so that $H$ is the $d$-dimensional hyperplane given by $x_{d+1}=0,\ldots,x_{4}=0$. After applying this rotation, it is still true that every line in $\Sigma_\delta(Z\cap N_{t}(H))$ makes an angle $\leq 1/3$ with the $e_1$ direction. Note that the map 
\begin{equation*}
\begin{split}
\{v\in S^4\subset\RR^4\colon \angle(v,e_1)\leq 1/3\} & \to \RR^3,\\
(v_1,v_2,v_3,v_4)&\mapsto (v_2,v_3,v_4),
\end{split}
\end{equation*}
is bi-Lipschitz with constant $\sim 1$. In particular, if $\ell,\ell^\prime\in \Sigma_\delta(Z\cap N_{t}(H))$ with $\angle(\ell,\ell^\prime)\geq\delta$, and if $v=v(\ell),\ v^\prime=v(\ell^\prime)$, then $\max(|v_2-v_2^\prime|,\ |v_3-v_3^\prime|,\ |v_4-v_4^\prime|)\gtrsim\delta$. Note as well that if $\ell\in \Sigma_\delta(Z\cap N_{t}(H))$ and if $v=v(\ell)$, then $|v_{d+1}|,\ldots,|v_{4}|\lesssim t$. 

Let $f\colon\RR^4\to\RR^4$ be the linear map that dilates $x_{d+1},\ldots,x_{4}$ by a factor of $1/t$ and leaves $x_{1},\ldots,x_d$ unchanged. Observe that if $\ell\in \Sigma_\delta(Z\cap N_{t}(H))$, and if $v=v(\ell),\ \tilde v = v(f(\ell))$, then $v_i\sim \tilde v_i/t$ for $i=d+1,\ldots,4$, and $v_i\sim \tilde v_i$ for $i=2,\ldots, d$. Thus if $\ell,\ell^\prime\in \Sigma_\delta(Z\cap N_{t}(H))$ with $\angle(\ell,\ell^\prime)\geq\delta$, and if $\tilde v = v(f(\ell))$ and $\tilde v^\prime = v(f(\ell^\prime))$, then 

\begin{equation}\label{spacingOfTildeV}
\max(|\tilde v_2-\tilde v_2^\prime|,\ldots, |\tilde v_{d}-\tilde v_{d}^\prime|,\ t|\tilde v_{d+1}-\tilde v_{d+1}^\prime|,\ldots,t|\tilde v_4-\tilde v_4^\prime|)\gtrsim \delta.
\end{equation}

Let $\lines_1\subset \Sigma_\delta(Z\cap N_{t}(H))$ be a set of lines pointing in $\delta$-separated directions. We will ``thin out'' the set of lines in $\lines_1$ by a factor of $t^{-1}$ in each of the $d-1$ directions $e_{2},\ldots,e_d$. More precisely, let $\lines_2\subset\lines_1$ be a set of lines with $|\lines_2|\gtrsim t^{d-1}|\lines_1|$, so that if $\ell,\ell^\prime\in \lines_2$ are distinct, and if $v=v(\ell),\ v^\prime=v(\ell^\prime)$, then at least one of $t|v_2-v_2^\prime|,\ldots, t|v_{d}- v_{d}^\prime|,$ or at least one of $| v_{d+1}- v_{d+1}^\prime|,\ldots,| v_4- v_4^\prime|)$ is $\gtrsim \delta.$

By \eqref{spacingOfTildeV}, we have that if $\ell,\ell^\prime\in \lines_2$, then $\angle\big(v(f(\ell)),\ v(f(\ell^\prime))\big)\geq\delta/t$. For each $\ell\in\lines_2$, we have that 
$$
f(\ell)\in\Sigma_{\delta/t}\big( B(0,1) \cap f(Z\cap N_{t}(H))\big).
$$
Applying Proposition \ref{mainProp} with $\delta^\prime = \delta/t$ and the same values of $D$ and $\eps$ as above, we conclude that
$$
|\lines_2|\leq C_{D,\eps} (\delta/t)^{-2-\eps},
$$
and thus
$$
|\lines_1|\lesssim C_{D,\eps} t^{3-d+\eps} \delta^{-2-\eps}.
$$

Since $\lines_1$ was an arbitrary set of lines in $\Sigma_\delta(Z\cap N_{t}(H))$ pointing in $\delta$-separated directions, we obtain \eqref{rescaledDirections}.
\end{proof}

\subsection{Proof of Proposition  \ref{mainProp}}
\begin{proof}
For each fixed value of $D$ and $\eps$, we will prove Proposition \ref{mainProp} by induction on $\delta$. Fix $D$ and $\eps$, and suppose that Proposition  \ref{mainProp} has been proved for all $\delta<\delta^\prime\leq 1$; let $C_{D,\eps}$ be the corresponding constant. We will show that if $C_{D,\eps}$ is sufficiently large (depending only on $D$ and $\eps$), then Proposition  \ref{mainProp} holds for $\delta$.

Let $s,\kappa,$ and $u$ be parameters that will be determined below; for the impatient reader, $s,\kappa$ and $u$ will be of size roughly $|\log\delta|^{-O_\eps}(1)$. Let $P$ be a polynomial of degree at most $D$. Apply Lemma \ref{selPtFiberProp} to the map $v\colon \Sigma_{\delta}(Z)\to S^3$ to select a semi-algebraic set $\Sigma\subset \Sigma_{\delta}(Z)$ whose lines point in different directions. We have that $\Sigma\subset S^3$ is a semi-algebraic set of complexity $O_D(1)$, and 
$$
\mathcal{E}_\delta\big(v(\Sigma)\big)=\mathcal{E}_\delta\big( v(\Sigma_{\delta}(Z))\big).
$$
Apply Theorem \ref{discretizedSeveri} to $Z(P)$ and $\Sigma$, and let $\Sigma_1,\Sigma_2,\Sigma_3,\Sigma_4$ be the resulting sets of lines. 

{\bf 1.} The lines in $\Sigma_1$ can be covered by $O_{D}(|\log\delta|^{O(1)}s^{-2})$ rectangular prisms of dimensions $2\times s\times s\times s$. Applying Lemma \ref{rescaleT} to each of these prisms, we conclude that
\begin{align*}
\mathcal{E}_\delta\big(v(\Sigma_1)\big) & \lesssim_{D} \Big(|\log\delta|^{O(1)}s^{-2}\Big) C_{D,\eps} s^{2+\eps} \delta^{-2-\eps}\\
&\leq \Big(C_{D}|\log\delta|^{O(1)} s^{\eps}\Big) C_{D,\eps}\delta^{-2-\eps}.
\end{align*}
Thus there exist constants $c_1>0$ and $C_1$, depending only on $D$ and $\eps$, so that if we define $s= c_1 |\log\delta|^{-C_1}$ then 
\begin{equation}\label{Sigma1Small}
\mathcal{E}_\delta\big(v(\Sigma_1)\big) \leq\frac{C_{D,\eps}}{4}\delta^{-2-\eps}.
\end{equation}

{\bf 2.} The lines in $\Sigma_2$ can be covered by $O_{D}\big((|\log\delta|/s)^{O(1)}u^{-1}\big)$ rectangular prisms of dimensions $2\times 2\times u\times u$. Applying Lemma \ref{rescaleT} to each of these prisms, we conclude that
\begin{align*}
\mathcal{E}_\delta\big(v(\Sigma_2)\big) & \lesssim_{D} \Big((|\log\delta|/s)^{O(1)}u^{-1}\Big) C_{D,\eps}u^{1+\eps} \delta^{-2-\eps}\\
&\leq \Big(C_{D}|\log\delta|^{O(1)}s^{-O(1)} u^{\eps}\Big) C_{D,\eps}\delta^{-2-\eps}.
\end{align*}
Thus there exist constants $c_2>0$ and $C_2$, depending only on $D$ and $\eps$ (also on $c_1$ and $C_1$, but this depends only on $D$ and $\eps)$, so that if we define $u= c_2 |\log\delta|^{-C_2}$ then 
\begin{equation}\label{Sigma2Small}
\mathcal{E}_\delta\big(v(\Sigma_2)\big) \leq\frac{C_{D,\eps}}{4}\delta^{-2-\eps}.
\end{equation}

{\bf 3.} The lines in $\Sigma_3$ can be covered by $O_{D}\big(|\log\delta|^{O(1)}\big)$ rectangular prisms of dimensions $2\times 2\times 2\times \kappa$. Applying Lemma \ref{rescaleT} to each of these prisms, we conclude that
\begin{align*}
\mathcal{E}_\delta\big(v(\Sigma_3)\big) & \lesssim_{D} \Big(|\log\delta|^{O(1)}\Big) C_{D,\eps} \kappa^{\eps} \delta^{-2-\eps}\\
&\leq \Big(C_{D}|\log\delta|^{O(1)} \kappa^{\eps}\Big) C_{D,\eps}\delta^{-2-\eps}.
\end{align*}
Thus there exist constants $c_3>0$ and $C_3$, depending only on $D$ and $\eps$, so that if we define $\kappa= c_3 |\log\delta|^{-C_3}$ then 
\begin{equation}\label{Sigma3Small}
\mathcal{E}_\delta\big(v(\Sigma_3)\big) \leq\frac{C_{D,\eps}}{4}\delta^{-2-\eps}.
\end{equation}

{\bf 4.} Finally, there exists a set $\Sigma_4^\prime\subset\Sigma_4$ and a quadratic polynomial $Q$ so that $\mathcal{E}_\delta(\Sigma_4^\prime)\gtrsim_{D} (s\kappa u/|\log\delta|)^{O(1)}\mathcal{E}_\delta(\Sigma_4)$, and for every line $\ell^\prime\in\Sigma_4^\prime$, there is a line $\ell$ contained in $Z(Q)$ with $\operatorname{dist}(\ell,\ell^\prime)\lesssim_D \big(|\log\delta|/(s\kappa u)\big)^{O(1)}\delta.$ Note that for every quadratic polynomial $Q$, we have 
$$
\mathcal{E}_{\delta}\Big(\big\{ \ell\in\lines\colon\ \textrm{there exists}\ \ell^\prime\subset Z(Q)\ \textrm{with}\ \operatorname{dist}(\ell,\ell^\prime)\lesssim_D (s\kappa u)^{-O(1)}\delta\big\}\Big)\lesssim_{D} (s\kappa u)^{-O(1)}\delta^{-2},
$$
and thus
$$
\mathcal{E}_{\delta}\big(v(\Sigma_4^\prime)\big)  \lesssim_{D} \big(|\log\delta|/s\kappa u\big)^{-O(1)}\delta^{-2},
$$
so
$$
\mathcal{E}_\delta\big(v(\Sigma_4)\big)\lesssim_D |\log\delta|^{O(1)}(s\kappa u)^{-O(1)}\delta^{-2}.
$$ 
Thus there exist constants $c_4>0$ and $C_4$, depending only on $D$ and $\eps$ (also on $c_1,c_2,c_3,C_1,C_2,C_3),$ but these in turn only depend on $D$ and $\eps$) so that

$$
\mathcal{E}_\delta\big(v(\Sigma_4)\big)\leq c_4 |\log\delta|^{-C_4}\delta^{-2}.
$$
If $C_{D,\eps}$ is sufficiently large (depending only on $D$ and $\eps$), then $c_4 |\log\delta|^{-C_4}\leq \frac{C_{D,\eps}}{4}\delta^{-2-\eps}$, and thus 
\begin{equation}\label{Sigma4Small}
\mathcal{E}_\delta\big(v(\Sigma_4)\big) \leq\frac{C_{D,\eps}}{4}\delta^{-2-\eps}.
\end{equation}
Combining \eqref{Sigma1Small}, \eqref{Sigma2Small}, \eqref{Sigma3Small} and \eqref{Sigma4Small}, we conclude that
$$
\mathcal{E}_\delta\big(v(\Sigma_\delta(Z))\big)\leq C_{D,\eps}\delta^{-2-\eps}.
$$
This completes the induction and concludes the proof. 
\end{proof}

\section{Improved Kakeya estimates in $\mathbb{R}^4$}\label{KakeyaSection}
In this section, we will show that Corollary \ref{mainThm} implies Theorem \ref{maximalFnThm}. First, we will recall several definitions from \cite{GZ}. Throughout this section, a $\delta$-tube is the $\delta$-neighborhood of a unit line segment contained in $B(0,1)$. 

\begin{defn}[Two-ends condition]
Let $\tubes$ be a set of $\delta$-tubes. For each $T\in\tubes$, let $Y(T)\subset T$. We say that the tubes in $\tubes$ satisfy the two-ends condition with exponent $\rho$ and error $\alpha$ if for all $T\in\tubes$ and for all balls $B(x,r)$ of radius $r$, we have
\begin{equation}
|Y(T)\cap B(x,r)|\leq \alpha r^{\rho}|Y(T)|.
\end{equation}
\end{defn}  

\begin{defn}[Robust transversality condition]
Let $\tubes$ be a set of $\delta$-tubes. For each $T\in\tubes$, let $Y(T)\subset T$. We say that $\tubes$ is $\beta$--robustly transverse if for all $x\in\RR^4$ and all vectors $v$, we have
\begin{equation}\label{quantTrans}
|\{T\in\tubes \colon x\in Y(T),\ \angle(T,v)< \beta \}|\leq \frac{1}{100} |\{T\in\tubes\colon x\in Y(T)\}|.
\end{equation}
\end{defn}

\begin{defn}
Let $\tubes$ be a set of $\delta$-tubes. We say that $\tubes$ satisfies the linear Wolff axioms if for every rectangular prism $R$ of dimensions $1\times t_1\times t_2\times t_3$, at most $100 t_1t_2t_3\delta^{-3}$ tubes from $\tubes$ can be contained in $R$.
\end{defn}

With these two definitions, we can now state Proposition 6.2 from \cite{GZ}:
\begin{prop}[\cite{GZ}, Proposition 6.2]\label{GZProp}
For each $\epsilon>0$ and $\rho>0$ there exist constants $c$, $C$, and $D$ so that the following holds. Let $\tubes$ be a set of $\delta$--tubes in $\RR^4$.  Suppose that $\tubes$ satisfies the linear Wolff axioms. Suppose furthermore that for every integer $1\leq E\leq D,$ for every polynomial $P\in\RR[x_1,x_2,x_3,x_4]$ of degree $E$, for every ball $B(x,r)$ of radius $r$, and for every $w>0$, we have
\begin{equation}\label{fewTubesInAHypersurface}
|\{T\in\tubes\colon T\cap B(x,r)\subset N_{10\delta}(Z)\}|\leq K_{E,w} \ r^{-1}\delta^{-2-w}.
\end{equation}

For each $T\in\tubes$, let $Y(T)\subset T$ with $\lambda\leq |Y(T)|/|T|\leq 2\lambda$. Suppose that $(\tubes,Y)$ is $s$--robustly transverse and that each tube $T\in\tubes$ satisfies the two-ends condition with exponent $\rho$ and error $\alpha$. Then 
\begin{equation} \label{volumeBoundExplicitTrans}
\Big| \bigcup_{T\in\tubes}Y(T)\Big| \geq c_{s} \alpha^{-C}\lambda^{3+1/28} K^{-1} \delta^{1-1/28+\epsilon}\big(\delta^3|\tubes|\big),
\end{equation}
where $ K=\max_{1\leq E\leq D}{K_E}$.
\end{prop}

Observe that inequality \eqref{fewTubesInAHypersurface} is a re-scaled version of Theorem \ref{mainThm}. Indeed, Theorem \ref{mainThm} asserts that if $\tubes$ is a set of $\delta$-tubes pointing in $\delta$-separated directions, then for every polynomial $P\in\RR[x_1,x_2,x_3,x_4]$ of degree $E$, for every ball $B(x,r)$ of radius $r$, and for every $w>0$, we have
$$
|\{T\in\tubes\colon T\cap B(x,r)\subset N_{10\delta}(Z)\}|\leq C_{E,w} r^{-3} (\delta/r)^{-2-w} \leq C_{E,w} r^{-1}\delta^{-2-w}.
$$
Since every set of $\delta$-tubes pointing in $\delta$-separated directions satisfies the linear Wolff axioms, we obtain the following variant of Proposition \ref{GZProp}.
\begin{prop}\label{alostThereProp}
For each $\epsilon>0$ and $\rho>0$ there exist constants $c$ and $C$ so that the following holds. Let $\tubes$ be a set of $\delta$--tubes in $\RR^4$ pointing in $\delta$-separated directions. For each $T\in\tubes$, let $Y(T)\subset T$ with $\lambda\leq |Y(T)|/|T|\leq 2\lambda$. Suppose that $(\tubes,Y)$ is $s$--robustly transverse and that each tube $T\in\tubes$ satisfies the two-ends condition with exponent $\rho$ and error $\alpha$. Then 
\begin{equation} \label{volumeBoundExplicitTrans}
\Big| \bigcup_{T\in\tubes}Y(T)\Big| \geq c_{s} c \alpha^{-C}\lambda^{3+1/28} \delta^{1-1/28+\epsilon}\big(\delta^3|\tubes|\big).
\end{equation}
\end{prop}

Finally, Theorem \ref{maximalFnThm} follows from Proposition \ref{alostThereProp} using the standard two-ends reduction and bilinear reduction. See e.g. \cite[Section 2]{GZ} for details.

\end{document}